\documentclass[10pt]{article}
\usepackage{amsmath,amssymb,amsthm,comment}
\usepackage{mathtools}
\usepackage{color}
\usepackage[nohead,margin=1.1in]{geometry}
\usepackage{graphicx,psfrag,epstopdf}
\usepackage{verbatim}
\usepackage{url}
\graphicspath{{./pics/}}

\allowdisplaybreaks

\theoremstyle{plain}
\newtheorem{theorem}{Theorem}[section]
\newtheorem{remark}{Remark}[section]

\newtheorem{lemma}{Lemma}[section]
\newtheorem{ass}{Assumption}[section]
\newtheorem{prop}{Proposition}[section]

\numberwithin{equation}{section}
\def\E{\mathbb{E}}
\def\d{\mathrm{d}}

\title{On the Regularizing Property of Stochastic Gradient Descent}

\author{Bangti Jin\thanks{Department of Computer Science, University College London, Gower Street, London WC1E 2BT, UK (\texttt{b.jin@ucl.ac.uk,bangti.jin@gmail.com})}
\and Xiliang Lu\thanks{School of Mathematics and Statistics and Hubei Key Laboratory of Computational Science, Wuhan University, Wuhan 430072, People's Republic of China (\texttt{xllv.math@whu.edu.cn})}}

\begin{document}
\maketitle

\begin{abstract}
Stochastic gradient descent (SGD) and its variants are among the most successful approaches for solving large-scale
optimization problems. At each iteration, SGD employs an unbiased estimator of the full gradient computed from
one single randomly selected data point. Hence, it scales well with problem size and is very attractive for handling
truly massive dataset, and holds significant potentials for solving large-scale inverse problems. In this work, we
rigorously establish its regularizing property under \textit{a priori} early stopping rule for linear
inverse problems, and also prove convergence rates under the canonical sourcewise condition.
This is achieved by combining tools from classical regularization theory and stochastic analysis. Further, we analyze its
preasymptotic weak and strong convergence behavior, in order to explain the fast initial convergence typically
observed in practice. The theoretical findings shed insights into the performance of the algorithm, and are
complemented with illustrative numerical experiments.
\end{abstract}

\section{Introduction}

In this paper, we consider the following finite-dimensional linear inverse problem:
\begin{equation}\label{eqn:lin}
  Ax = y^\dag,
\end{equation}
where $A\in\mathbb{R}^{n\times m}$ is a matrix representing the data formation mechanism,
$x\in\mathbb{R}^m$ is the unknown signal of interest, and $y^\dag\in\mathbb{R}^n$
is the exact data formed by $y^\dag= Ax^\dag$, with $x^\dag$ being the true solution.
In practice, we can only access the noisy data $y^\delta\in\mathbb{R}^n$ defined by
\begin{equation*}
  y^\delta = y^\dag + \xi,
\end{equation*}
where the vector $\xi\in \mathbb{R}^n$ is the noise in the data, with a noise level $\delta = \|\xi\|$ (and $\bar\delta
=n^{-\frac{1}{2}}\delta$). The noise $\xi$ is assumed to be a realization
of an independent identically distributed (i.i.d.) mean zero Gaussian random vector. Throughout, we denote the $i$th row of the matrix
$A$ by a column vector $a_i\in\mathbb{R}^m$, and the $i$th entry of the vector $y^\delta$ by $y_i^\delta$. The model \eqref{eqn:lin}
is representative of many discrete linear inverse problems, including linearized (sub)problems
of nonlinear inverse problems. Hence, the stable and efficient numerical solution of the model
\eqref{eqn:lin} has been the topic of many research works, and plays an important role in developing practical inversion techniques
(see, e.g., \cite{EnglHankeNeubauer:1996,ItoJin:2015}).

Stochastic gradient descent (SGD), dated at least back to Robbins and Monro \cite{RobbinsMonro:1951}, represents an
extremely popular solver for large-scale least square type problems and statistical inference, and its accelerated
variants represent state-of-the-art solvers for training (deep) neural networks \cite{Bottou:2010,JohnsonZhang:2013,
BottouCurtisNocedal:2018}. Such methods hold significant potentials for solving large-scale inverse problems. For
example, the randomized Kaczmarz method \cite{StrohmerVershynin:2009}, which has long been very popular and successful
in computed tomography \cite{Natterer:2001}, can be viewed as SGD with weighted sampling (see, e.g.,
\cite{NeedellSrebroWard:2016} and \cite[Prop. 4.1]{JiaoJinLu:2017}). Besides the randomized Kaczmarz method, there are also
several experimental evaluations on SGD for inverse problems \cite{JiaLam:2010,ChenLiLiu:2018}. Hence, it is extremely
important to understand theoretical properties of such stochastic reconstruction methods, which, to the best of our
knowledge, have not been addressed in the context of ill-posed inverse problems.

In this work, we contribute to the theoretical analysis of SGD for inverse problems. Consider the following basic
version of SGD: given an initial guess $x_1\in\mathbb{R}^m$, update the iterate $x_{k+1}^\delta$ by
\begin{equation}\label{eqn:sgd}
  x_{k+1}^\delta = x_k^\delta - \eta_k ((a_{i_k},x_k^\delta)-y_{i_k}^\delta)a_{i_k},\quad k=1,\ldots
\end{equation}
where the index $i_k$ is drawn i.i.d. uniformly from the set $\{1,\ldots,n\}$, $\eta_k>0$ is the step size at the
$k$th iteration, and $(\cdot,\cdot)$ denotes the Euclidean inner product on $\mathbb{R}^m$. The update \eqref{eqn:sgd}
can be derived by computing an unbiased gradient estimate $((a_{i_k},x)-y_{i_k}^\delta)a_{i_k}$ from the functional
$\frac{1}{2}(y_{i_k}^\delta-(a_{i_k},x))^2$ for a randomly sampled single datum $\{a_{i_k},y_{i_k}^\delta\}$, instead of
the gradient $n^{-1}A^t(Ax-y^\delta)$ of the functional $\frac{1}{2n}\sum_{i=1}^n(y_i^\delta-(a_i,x))^2$ for the full data.
Thus, the SGD iteration \eqref{eqn:sgd} is a randomized version of the classical Landweber iteration:
\begin{equation}\label{eqn:Landweber}
  x_{k+1}^\delta = x_k^\delta - \eta_k n^{-1}A^t(Ax_k^\delta-y^\delta).
\end{equation}

In comparison with Landweber iteration \eqref{eqn:Landweber}, SGD requires only evaluating one datum
$\{a_{i_k},y_{i_k}^\delta\}$ per iteration, and thus the per-iteration cost is drastically reduced,
which is especially attractive for large-scale problems. In theory, Landweber method is known to be
regularizing \cite[Chapter 6]{EnglHankeNeubauer:1996}. However, the regularizing property of SGD
remains to be established, even though it was conjectured and empirically examined (see, e.g.,
\cite{Keskar:2017,Jastrzbski:2017,ZhuWuYu:2018}). Numerically, one observes a semiconvergence phenomenon
for SGD: the iterate $x_k^\delta$ first converges to the true solution $x^\dag$, and then diverges as the
iteration further proceeds. Semiconvergence is characteristic of (deterministic) iterative regularization
methods, and early stopping is often employed \cite{EnglHankeNeubauer:1996,KaltenbacherNeubauerScherzer:2008}.
Below we describe the main theoretical contributions of this work, which are
complemented with numerical experiments in Section \ref{sec:numer}.

The first contribution is to analyze SGD with a polynomially decaying sequence of step sizes (see Assumption
\ref{ass:eta}) through the lens of regularization theory. In Theorems \ref{thm:consistency} and \ref{thm:err-total},
we prove that SGD is regularizing in the sense that iterate $x_k^\delta$ converges to the exact solution $x^\dag$
in the mean squared norm as the noise level $\delta$ tends to zero, under \textit{a priori} early stopping rule,
and also $x_k^\delta$ converges to $x^\dag$ at certain rates under canonical source condition. To the best of our
knowledge, this  is the first result on regularizing property of a stochastic iteration method. The analysis relies
on decomposing the error into three components: approximation error due to early stopping, propagation error due to
the presence of data noise, and stochastic error due to the random index $i_k$. The first two parts are deterministic
and can be analyzed in a manner similar to Landweber method \cite[Chapter 6]{EnglHankeNeubauer:1996}; see Theorem
\ref{thm:err-approx} and \ref{thm:err-prop}. The last part on the variance of the iterate constitutes the
main technical challenge in the analysis. It is overcome by relating the iterate variance to the
expected square residuals and analyzing the evolution of the latter; see Theorems \ref{thm:err-var1}
and \ref{thm:res-noisy}.

The second contribution is to analyze the preasymptotic convergence in both weak and strong sense.
In practice, it is often observed that SGD can decrease the error very fast during initial iterations.
We provide one explanation of the phenomenon by means of preasymptotic convergence, which extends the
recent work on the randomized Kaczmarz method \cite{JiaoJinLu:2017}. It is achieved by dividing the
error into low- and high-frequency components according to right singular vectors, and studying their
dynamics separately. In Theorems \ref{thm:asympt-weak-noisy} and \ref{thm:asympt-noisy}, we prove that the
low-frequency error can decay much faster than the high-frequency one in either weak or strong norm.
In particular, if the initial error is dominated by the low-frequency components, then SGD decreases
the total error very effectively during the first iterations. The analysis sheds important insights into
practical performance of SGD. Further, under the canonical source type condition, the low-frequency
error is indeed dominating, cf. Proposition \ref{prop:source}.

Now we situate this work in the existing literature in two related areas: inverse problems with random noise,
and machine learning. Inverse problems with random noise have attracted much attention over the last decade.
In a series of works, Hohage and his collaborators \cite{BauerHohageMunk:2009,BissantzHohageMunk:2004,
BissantzHohageMunk:2007} studied various regularization methods, e.g., Tikhonov and iterative regularization,
for solving linear and nonlinear inverse problems with random noise. For example, Bissantz et al
\cite{BissantzHohageMunk:2004} analyzed Tikhonov regularization for nonlinear inverse problems, and analyzed
consistency and convergence rate. In these works, randomness enters into the problem formulation via the data
$y^\delta$ directly as a Hilbert space valued process, which is fixed (though random) when applying
regularization techniques. Thus, it differs greatly from SGD, for which randomness arises due to the random
row index $i_k$ and changes at each iteration. Handling the iteration noise requires different techniques
than that in these works.

There are also several relevant works in the context of machine learning \cite{YingPontil:2008,TarresYao:2014,LinRosasco:2017,
DieuleveutBach:2016}. Ying and Pontil \cite{YingPontil:2008} studied an online least-squares gradient
descent algorithm in a reproducing kernel Hilbert space (RKHS), and presented a novel capacity independent
approach to derive bounds on the generalization error. Tarres and Yao \cite{TarresYao:2014} analyzed the
convergence of a (regularized) online learning algorithm closely related to SGD. Lin and Rosasco
\cite{LinRosasco:2017} analyzed the influence of batch size on the convergence of mini-batch SGD. See
also the recent work \cite{DieuleveutBach:2016} on SGD with averaging for nonparametric regression in RKHS.
All these  works analyze the method in the framework of statistical learning, where the noise arises mainly
due to finite sampling of the (unknown) underlying data distribution, whereas for inverse problems, the
noise arises from imperfect data acquisition process and enters into the data $y^\delta$ directly.
Further, the main focus of these works is to bound the generalization error, instead of error estimates
for the iterate. Nonetheless, our proof strategy in decomposing the total error into three different
components shares similarity with these works.

The rest of the paper is organized as follows. In Section \ref{sec:main}, we present and discuss the main results,
i.e., regularizing property and preasymptotic convergence. In Section \ref{sec:basicest}, we derive bounds on
three parts of the total error. Then in Section \ref{sec:reg}, we analyze the regularizing property of SGD with
\textit{a priori} stopping rule, and prove convergence rates under classical source condition. In Section
\ref{app:preasympt}, we discuss the preasymptotic convergence of SGD.  Some numerical results are given in
Section \ref{sec:numer}. In an appendix, we collect some useful inequalities. We conclude this section with
some notation. We use the superscript $\delta$ in $x_k^\delta$ to indicate SGD iterates for noisy data
$y^\delta$, and denote by $x_{k}$ that for the exact data $y^\dag$. The notation $\|\cdot\|$
denotes Euclidean norm for vectors and spectral norm for matrices, and $[\cdot]$ denotes the integral part of
a real number. $\{\mathcal{F}_k\}_{k\geq 1}$ denotes a sequence of increasing $\sigma$-fields generated by the
random index $i_k$ up to the $k$th iteration.  The notation $c$, with or without subscript, denotes a generic
constant that is always independent of the iteration index $k$ and the noise level $\delta$.

\section{Main results and discussions}\label{sec:main}

In this part, we present the main results of the work, i.e., regularizing property of SGD and
preasymptotic convergence results. The detailed proofs are deferred to Sections \ref{sec:reg}
and \ref{app:preasympt}, which in turn rely on technical estimates derived in Section
\ref{sec:basicest}. Throughout, we consider the following step size schedule, which
is commonly employed for SGD.
\begin{ass}\label{ass:eta}
The step size $\eta_j=c_0j^{-\alpha}$, $j=1,2,\ldots,$ $\alpha\in(0,1)$, with $c_0\max_i\|a_i\|^2\leq 1$.
\end{ass}

Due to stochasticity of the row index $i_k$, the iterate $x_k^\delta$ is random. We measure the approximation
error $x_k^\delta-x^\dag$ to the true solution $x^\dag$ by the mean squared error $\mathbb{E}[\|x_k^\delta -x ^\dag\|^2]$,
where the expectation $\E[\cdot]$ is with respect to the random index $i_k$. The reference solution $x^\dag$ is taken
to be the unique minimum norm solution (relative to the initial guess $x_1$):
\begin{equation}\label{eqn:sol-min}
  x^\dag = \arg\min_{x\in\mathbb{R}^m}\left\{\|x-x_1\|\quad \mbox{s.t.}\quad  Ax^\dag = y^\dag\right\}.
\end{equation}

Now we can state the regularizing property of SGD \eqref{eqn:sgd} under \textit{a priori} stopping rule: the error
$\mathbb{E}[\|x_{k(\delta)}^\delta -x ^\dag\|^2]$ tends to zero as the noise level $\delta\to0$, if the stopping
index $k(\delta)$ is chosen properly in relation to the noise level $\delta$. Thus, SGD equipped with suitable
\textit{a priori} stopping rule is a regularization method. Note that condition \eqref{eqn:cond-k} is analogous to
that for classical regularization methods.
\begin{theorem}\label{thm:consistency}
Let Assumption \ref{ass:eta} be fulfilled. If the stopping index $k(\delta)$ satisfies
\begin{equation}\label{eqn:cond-k}
  \lim_{\delta\to 0^+}k(\delta) = \infty \quad \mbox{and}\quad \lim_{\delta\to0^+} k(\delta)^{\frac{\alpha-1}{2}} \delta = 0,
\end{equation}
 then the iterate $x_{k(\delta)}^\delta$ satisfies
\begin{equation*}
  \lim_{\delta\to 0^+} \E[\|x_{k(\delta)}^\delta - x^\dag\|^2] = 0.
\end{equation*}
\end{theorem}

To derive convergence rates, we employ the source condition in classical regularization theory
\cite{EnglHankeNeubauer:1996,ItoJin:2015}. Recall that the canonical source condition reads: there exists
some $w\in\mathbb{R}^m$ such that
\begin{equation}\label{eqn:source}
  x^\dag - x_1 = B^pw,\quad p\ge0,
\end{equation}
where the symmetric and positive semidefinite $B\in\mathbb{R}^{m\times m}$ is defined in \eqref{eqn:B} below,
and $B^p$ denotes the usual fractional power (via spectral decomposition). Condition \eqref{eqn:source}
represents a type of smoothness of the initial error $x^\dag-x_1$, and the exponent $p$ determines the degree
of smoothness: the larger the exponent $p$ is, the smoother the initial error $x^\dag-x_1$ becomes. It controls
the approximation error due to early stopping (see Theorem \ref{thm:err-approx} below for the precise statement).
The source type condition is one of the most classical approaches to derive convergence rates in classical
regularization theory \cite{EnglHankeNeubauer:1996,ItoJin:2015}.

Next we can state convergence rates under \textit{a priori} stopping index.
\begin{theorem}\label{thm:err-total}
Let Assumption \ref{ass:eta} and the source condition \eqref{eqn:source} be fulfilled. Then there holds
\begin{equation*}
    \mathbb{E}[\|x_{k+1}^\delta - x^\dag\|^2] \leq ck^{-\min(2\alpha,\min(1,2p)(1-\alpha))}\ln^2 k+c' k^{1-\alpha}\bar\delta^2 + c''\delta^2,
\end{equation*}
where the constants $c,c'$ and $c''$ depend on $\alpha,p,\|w\|,\|Ax_1-y^\delta\|$ and $\|A\|$.
\end{theorem}

\begin{remark}\label{rmk:rate}
Theorem \ref{thm:err-total} indicates a semiconvergence for the iterate $x_k^\delta$: the first term
is decreasing in $k$ and dependent of regularity index $p$ and the step size parameter $\alpha\in(0,1)$,
and the second term $k^{1-\alpha}\bar\delta^2 $ is increasing in $k$ and dependent of the noise level.
The first term $k^{-\min(2\alpha,\min(1,2p)(1-\alpha))}\ln^2 k$ contains both approximation error 
{\rm(}indicated by $p${\rm)} and stochastic error. By properly balancing the first two terms in the 
estimate, one can obtain a convergence rate.  The best possible convergence rate depends on both the decay rate
$\alpha$ and the regularity index $p$ in \eqref{eqn:source}, and it is suboptimal for any $p>\frac12$
when compared with Landweber method. That is, the vanilla SGD seems to suffer from saturation, due
to the stochasticity induced by the random row index $i_k$.
\end{remark}

In practice, it is often observed that SGD decreases the error rapidly during the initial iterations. This
phenomenon cannot be explained by the regularizing property. Instead, we analyze the preasymptotic
convergence by means of SVD, in order to explain the fast initial convergence. Let $n^{-\frac12}A = U\Sigma
V^t$, where $U\in\mathbb{R}^{n\times n}, V=[v_1\ v_2\ \cdots\ v_m]\in\mathbb{R}^{m\times m}$ are orthonormal,
and $\Sigma =\mathrm{diag}(\sigma_1,\ldots,\sigma_r,0,\ldots,0)\in \mathbb{R}^{n\times m}$ is diagonal with
the diagonals ordered nonincreasingly and $r$ the rank of $A$. For any fixed truncation level $1\leq L\leq r$, we
define the low- and high-frequency solution spaces $\mathcal{L}$ and $\mathcal{H}$ respectively by
\begin{equation*}
  \mathcal{L} = \mathrm{span}(\{v_i\}_{i=1}^L)\quad \mbox{and}\quad \mathcal{H} = \mathrm{span}(\{v\}_{i=L+1}^{\min(n,m)}).
\end{equation*}
Let $P_\mathcal{L}$ and $P_\mathcal{H}$ be the orthogonal projection onto $\mathcal{L}$ and $\mathcal{H}$,
respectively. The analysis relies on decomposing the error $e_k^\delta=x_k^\delta-x^\dag$ into the low-
and high-frequency components $P_\mathcal{L}e_k^\delta$ and $P_\mathcal{H}e_k^\delta$, respectively, in
order to capture their essentially different dynamics.

We have the following preasymptotic weak and strong convergence results,
which characterize the one-step evolution of the low- and high-frequency errors. The proofs
are given in Section \ref{app:preasympt}.
\begin{theorem}\label{thm:asympt-weak-noisy}
If $\eta_k\leq c_0 $ with $c_0\max_i\|a_i\|^2\leq 1$,
then there hold
\begin{align*}
  \|\mathbb{E}[P_\mathcal{L}e_{k+1}^\delta]\| &\leq (1-\eta_k\sigma_L^2) \|\mathbb{E}[P_\mathcal{L}e_k^\delta]\|+c_0^{-\frac12}\eta_k\bar\delta,\\
  \|\mathbb{E}[P_\mathcal{H}e_{k+1}^\delta]\| &\leq \|\mathbb{E}[P_\mathcal{H}e_k^\delta]\|+\eta_k\sigma_{L+1}\bar\delta.
\end{align*}
\end{theorem}

\begin{theorem}\label{thm:asympt-noisy}
If $\eta_k\leq c_0 $ with $c_0\max_i\|a_i\|^2\leq 1$, then with $c_1=\sigma_L^2$,
and $c_2=\sum_{i=L+1}^r\sigma_i^2$, there hold
\begin{align*}
  \mathbb{E}[\|P_\mathcal{L}e_{k+1}^\delta\|^2|\mathcal{F}_{k-1}]  &\leq (1-c_1\eta_k) \|P_\mathcal{L}e_k^\delta\|^2 + c_2c_0^{-1}\eta_k^2\|P_\mathcal{H}e_k^\delta\|^2
  +c_0^{-1}\eta_k\bar\delta(\eta_k\bar\delta+2\sqrt{2}\sigma_1\|e_k^\delta\|),\\
  \mathbb{E}[\|P_\mathcal{H}e_{k+1}^\delta\|^2|\mathcal{F}_{k-1}] &\leq c_2c_0^{-1}\eta_k^2\|P_\mathcal{L}e_k^\delta\|^2 +(1+c_2c_0^{-1}\eta_k^2)\|P_\mathcal{H}e_k^\delta\|^2
+c_0^{-1}\eta_k^{2}\bar\delta^2 \\
 &\quad +2\sqrt{2}c_2^\frac{1}{2}\eta_k \bar\delta\Big(\|P_\mathcal{H}e_k^\delta\|^2 + c_0^{-2}\eta_k^2\|e_k^\delta\|^2\Big)^\frac{1}{2}.
\end{align*}
\end{theorem}

\begin{remark}
It is noteworthy that in Theorems \ref{thm:asympt-weak-noisy} and \ref{thm:asympt-noisy}, the
step size $\eta_k$ is not required to be polynomially decaying.
Theorems \ref{thm:asympt-weak-noisy} and \ref{thm:asympt-noisy} indicate that the low-frequency
error can decrease much faster than the high-frequency error in either the weak or mean squared
norm sense. Thus, if the initial error $e_1$ consists mostly of low-frequency modes, SGD can
decrease the low-frequency error and thus also the total error rapidly, resulting in
fast initial convergence.
\end{remark}

\section{Preliminary estimates}\label{sec:basicest}
In this part, we provide several technical estimates for the SGD iteration \eqref{eqn:sgd}. By
bias-variance decomposition and triangle inequality, we have
\begin{align}
  \E[\|x_k^\delta - x^\dag\|^2] & = \|\E[x_k^\delta]-x^\dag\|^2 + \E[\|\E[x_k^\delta]-x_k^\delta\|^2]\nonumber\\
     & \leq 2\|\E[x_k]-x^\dag\|^2 + 2\|\E[x_k-x_k^\delta]\|^2 + \E[\|\E[x_k^\delta]-x_k^\delta\|^2],\label{eqn:decomp}
\end{align}
where $x_k$ is the random iterate for exact data $y^\dag$. Thus, the total error is decomposed into three components:
approximation error due to early stopping, propagation error due to noise and stochastic error due to the random
index $i_k$. The objective below is to derive bounds on the three terms in \eqref{eqn:decomp}, which are crucial for
proving Theorems \ref{thm:consistency} and \ref{thm:err-total} in Section \ref{sec:reg}. The approximation and
propagation errors are given in Theorems \ref{thm:err-approx} and \ref{thm:err-prop}, respectively. The stochastic
error is analyzed in Section \ref{ssec:residual}: first in terms of the expected squared residuals in Theorem
\ref{thm:err-var1}, and then bound on the latter in Theorem \ref{thm:res-noisy}. The analysis of the stochastic error
represents the main technical challenge.

\subsection{Approximation and propagation errors}

For the analysis, we first introduce auxiliary iterations. Let $e_{k}^\delta = x_k^\delta-x^\dag$ and $e_k = x_k -x^\dag$
be the errors for SGD iterates $x_k^\delta$ and $x_k$, for $y^\delta$ and $y^\dag$, respectively.
They satisfy the following recursion:
\begin{align}
  e_{k+1} & = e_k - \eta_k ((a_{i_k},x_k)-y_{i_k}^\dag)a_{i_k}
       = e_k - \eta_k(a_{i_k},e_k)a_{i_k},\label{eqn:iter-ex}\\
  e_{k+1}^\delta & = e_k^\delta - \eta_k ((a_{i_k},x_k^\delta)-y_{i_k}^\delta)a_{i_k}
    = e_k^\delta - \eta_k(a_{i_k},e_k^\delta)a_{i_k} +\eta_k\xi_{i_k}a_{i_k}\label{eqn:iter-noisy}.
\end{align}
Then we introduce two auxiliary matrices: for any vector $b\in\mathbb{R}^n$,
\begin{align}\label{eqn:B}
  B:=\E[a_ia_i^t] \quad\mbox{and}\quad  \bar A^tb:=\E[a_ib_i].
\end{align}
Under i.i.d. uniform sampling of the index $i_k$, $B=n^{-1}A^tA$ and $\bar A^t=n^{-1}A^t$. Below,
let
\begin{equation}\label{eqn:Pi}
  \Pi_{j}^k(B) = \prod_{i=j}^k(I-\eta_iB), \quad j\leq k,
\end{equation}
with the convention $\Pi_{k+1}^k(B)=I$,

Now we bound the weighted norm $\|B^s\mathbb{E}[e_{k}]\|$ of the approximation error $\E[e_{k}]$. The cases $s=0$ and
$s=1/2$ will be used for bounding the approximation error and
the residual, respectively.
\begin{theorem}\label{thm:err-approx}
Let Assumption \ref{ass:eta} be fulfilled. Under the source condition \eqref{eqn:source} and
for any $s\geq0$, with $c_{p,s}=(\frac{(p+s)(1-\alpha)}{c_0e(2^{1-\alpha}-1)})^{p+s}\|w\|$, there holds
\begin{equation*}
  \|B^s\mathbb{E}[e_{k+1}]\|\leq c_{p,s}k^{-(p+s)(1-\alpha)}.
\end{equation*}
\end{theorem}
\begin{proof}
It follows from \eqref{eqn:iter-ex} and the identity $y_i^\dag=(a_i,x^\dag)$ that the error $e_{k}$ satisfies
\begin{align*}
  \mathbb{E}[e_{k+1}|\mathcal{F}_{k-1}] &= (I-\eta_k\mathbb{E}[a_ia_i^t])e_k = (I-\eta_k B)e_k.
\end{align*}
Taking the full expectation yields
\begin{equation}\label{eqn:iter-mean}
\mathbb{E}[e_{k+1}] = (I-\eta_k B)\mathbb{E}[e_k].
\end{equation}
Repeatedly applying the recursion \eqref{eqn:iter-mean} and noting that $e_1$ is deterministic give
\begin{equation*}
  \mathbb{E}[e_{k+1}] = \prod_{i=1}^k(I-\eta_iB)\mathbb{E}[e_1] = \prod_{i=1}^k(I-\eta_iB)e_1.
\end{equation*}
From the source condition \eqref{eqn:source}, we deduce
\begin{equation*}
  \|B^s\mathbb{E}[e_{k+1}]\| \leq \|\Pi_{1}^k(B)B^{p+s}\|\|w\|.
\end{equation*}
By Lemmas \ref{lem:kernel} and \ref{lem:basicest}, we arrive at
\begin{align*}
  \|\mathbb{E}[e_{k+1}]\| & \leq \frac{(p+s)^{p+s}}{e^{p+s}(\sum_{i=1}^k\eta_i)^{p+s}}\|w\|
   \leq c_{p,s} k^{-(p+s)(1-\alpha)},
\end{align*}
with a constant $c_{p,s}=(\frac{(p+s)(1-\alpha)}{c_0e(2^{1-\alpha}-1)})^{p+s}\|w\|$.
This completes the proof of the theorem.
\end{proof}

\begin{remark}
The constant $c_{p,s}$ is uniformly bounded in $\alpha\in[0,1]$: $\lim_{\alpha\to1^-}\frac{1-\alpha}{2^{1-\alpha}-1} = \frac{1}{\ln 2}$.
\end{remark}

Next we bound the weighted norm of the propagation error
$\E[x_k^\delta-x_k]$ due to data noise $\xi$.
\begin{theorem}\label{thm:err-prop}
Let Assumption \ref{ass:eta} be fulfilled, $s\in[-\frac12,\frac{1}{2}]$, and $r=\frac{1}{2}+s$. Then there holds
\begin{equation*}
  \|B^s\E[x_{k+1}-x_{k+1}^\delta]\| \leq c_{r,\alpha}\bar\delta \left\{\begin{array}{ll}
        k^{(1-r)(1-\alpha)}, & 0\leq r <1,\\
      \max(1,\ln k), & r= 1,
     \end{array}\right.
\end{equation*}
with $c_{r,\alpha}$ given by
\begin{equation*}
  c_{r,\alpha}=c_0^{1-r}\left\{\begin{array}{ll}
     \frac{r^r}{e^r}B(1-\alpha,1-r)+1, & r<1,\\
     \frac{r^r}{e^r}2^\alpha\frac{2-\alpha}{1-\alpha}+1, &r=1.
     \end{array}\right.
\end{equation*}
\end{theorem}
\begin{proof}
By the recursions \eqref{eqn:iter-ex} and \eqref{eqn:iter-noisy}, the propagation error
$\nu_k=\mathbb{E}[x_k^\delta-x_k]$ satisfies $\nu_1=0$ and $\nu_{k+1} = (I-\eta_kB)\nu_k
+ \eta_k\bar A^t\xi$, with $\xi=y^\delta-y^\dag$. Applying the recursion repeatedly yields
\begin{equation*}
  \nu_{k+1} = \sum_{j=1}^{k} \eta_j\Pi_{j+1}^k(B)\bar A^t \xi.
\end{equation*}
Thus, by the triangle inequality, we have
\begin{align*}
  \|B^s\nu_{k+1}\| \leq \sum_{j=1}^{k} \eta_j\|B^s\Pi_{j+1}^k(B)\bar A^t\|\| \xi\|.
\end{align*}
Since $\|B^s\Pi_{j+1}^k(B)\bar A^t\| = n^{-\frac12}\|\Pi_{j+1}^k(B)B^{s+\frac12}\|$,
by Lemma \ref{lem:kernel},
\begin{align*}
  \|B^s\nu_{k+1}\| &\leq \frac{r^r}{e^r}\sum_{j=1}^{k-1}\frac{\eta_j}{(\sum_{i=j+1}^k\eta_i)^r}\bar\delta+\eta_k\|B^s\bar A^t\|\|\xi\|\\
   & = \Big(\frac{r^r}{e^r}\sum_{j=1}^{k-1}\frac{\eta_j}{(\sum_{i=j+1}^k\eta_i)^r} + k^{-\alpha}c_0\|B\|^r\Big)\bar\delta.
\end{align*}
Under Assumption \ref{ass:eta}, we have $c_0\|B\|^r \leq c_0^{1-r}$. This and Lemma \ref{lem:basicest} complete the proof.
\end{proof}

\begin{remark}
The iterate means $\mathbb{E}[x_k]$ and $\E[x_k^\delta]$ satisfy the recursion for Landweber method {\rm(}LM{\rm)}. Hence, the
proof and error bounds resemble closely that for LM \cite[Chapter 6]{EnglHankeNeubauer:1996}. Taking $s=0$ in
Theorems \ref{thm:err-approx} and \ref{thm:err-prop} yields
\begin{equation*}
  \|\E[x_{k+1}^\delta]-x^\dag\|\leq c_pk^{-p(1-\alpha)} + c_\alpha k^{\frac{1-\alpha}{2}}\bar\delta.
\end{equation*}
By balancing the two terms, one can derive a convergence rate in terms of $\bar\delta$ {\rm(}instead
of $\delta${\rm)}, and this is achieved quickest by $\alpha=0$. Such an estimate is known as weak
error in the literature of stochastic differential equations. By bias variance decomposition, it is
weaker than the mean squared error.
\end{remark}

\subsection{Stochastic error}\label{ssec:residual}

The next result gives a bound on the variance $\mathbb{E}[\|B^s(x_k^\delta - \E[x_k^\delta])\|^2]$.
It arises from the random index $i_k$ in SGD \eqref{eqn:sgd}. Theorem \ref{thm:err-var1}
relates the variance to the past mean squared residuals $\{\E[\|Ax_j^\delta-y^\delta\|^2]\}_{j=1}^k$
and step sizes $\{\eta_j\}_{j=1}^k$. The extra exponent $\frac{1}{2}$ follows from the quadratic
structure of the least-squares functional.

\begin{theorem}\label{thm:err-var1}
For the SGD iteration \eqref{eqn:sgd}, there holds
\begin{equation*}
  \mathbb{E}[\|B^s(x_{k+1}^\delta-\mathbb{E}[x_{k+1}^\delta])\|^2] \leq \sum_{j=1}^{k}  \eta_j^2\|B^{s+\frac{1}{2}}\Pi_{j+1}^k(B)\|^2 \mathbb{E}[\|Ax_j^\delta-y^\delta\|^2].
\end{equation*}
\end{theorem}
\begin{proof}
Let $z_k=x_k^\delta-\mathbb{E}[x_k^\delta]$. By the definition of the iterate $x_k^\delta$ in \eqref{eqn:iter-noisy}, we have
$\E[x_{k+1}^\delta] = \E[x_k^\delta] - \eta_k(B\E[x_k^\delta]-\bar A^ty^\delta)$, and thus $z_k$ satisfies
\begin{equation*}
  z_{k+1} = z_k - \eta_k[((a_{i_k},x_k^\delta)-y_{i_k}^\delta)a_{i_k}-(B\mathbb{E}[x_k^\delta]-\bar A^ty^\delta)],
\end{equation*}
with  $z_1=0$. Upon rewriting, $z_k$ satisfies
\begin{equation}\label{eqn:iter-rand}
  z_{k+1} = (I-\eta_k B)z_k + \eta_k M_k,
\end{equation}
where the iteration noise $M_k$ is defined by
\begin{equation*}
  M_k = (Bx_k^\delta-\bar A^ty^\delta)-((a_{i_k},x_k^\delta)-y_{i_k}^\delta)a_{i_k}.
\end{equation*}
Since $x_j^\delta$ is measurable with respect to $\mathcal{F}_{j-1}$, $\mathbb{E}[M_j|\mathcal{F}_{j-1}] = 0$, and thus $\E[M_j]=0$. Further,
for $j\neq \ell$, $M_j$ and $M_\ell$ satisfy
\begin{equation}\label{eqn:indep}
  \mathbb{E}[(M_j,M_\ell)] = 0,\quad \forall j\neq \ell.
\end{equation}
Indeed, for $j<\ell$, we have $\mathbb{E}[(M_j,M_\ell)|\mathcal{F}_{\ell-1}] = (M_j,\mathbb{E}[M_\ell|
\mathcal{F}_{\ell-1}]) = 0$, since $M_j$ is measurable with respect to $\mathcal{F}_{\ell-1}$. Then
taking full expectation yields \eqref{eqn:indep}. Applying the recursion \eqref{eqn:iter-rand} repeatedly gives
\begin{equation*}
  z_{k+1} = \sum_{j=1}^{k} \eta_j \Pi_{j+1}^k(B) M_j.
\end{equation*}
Then it follows from \eqref{eqn:indep} that
\begin{align*}
 \mathbb{E}[\|B^sz_{k+1}\|^2] & = \sum_{j=1}^{k}\sum_{\ell=1}^{k}  \eta_j\eta_\ell \mathbb{E}[(B^s\Pi_{j+1}^k(B) M_j,B^s\Pi_{\ell+1}^k(B) M_\ell)] = \sum_{j=1}^{k}\eta_j^2\mathbb{E}[\|B^s\Pi_{j+1}^k(B) M_j\|^2].
\end{align*}
Since $a_i = A^te_i$ (with $e_i$ being the $i$th Cartesian vector), we have (with $\bar y^\delta= n^{-1}y^\delta$)
\begin{align*}
    M_j &= A^t(\bar A x_j^\delta-\bar y^\delta)-((a_{i_j},x_j^\delta)-y_{i_j}^\delta)A^{t}e_{i_j}\\
   &= A^t[(\bar Ax_j^\delta-\bar y^\delta)-((a_{i_j},x_j^\delta)-y_{i_j}^\delta)e_{i_j}]:=A^tN_j.
\end{align*}
This and the identity $\|B^s\Pi_{j+1}^k(B)A^t\|^2=n\|B^s\Pi_{j+1}^k(B)B^\frac{1}{2}\|^2$ yield
\begin{align*}
\mathbb{E}[\|B^s\Pi_{j+1}^k(B) M_j\|^2] & \leq \|B^s \Pi_{j+1}^k(B)A^t\|^2 \mathbb{E}[\|N_j\|^2]=\|B^{s+\frac12}\Pi_{j+1}^k(B)\|^2 \mathbb{E}[n\|N_j\|^2].
\end{align*}
By the measurability of $x_j^\delta$ with respect to $\mathcal{F}_{j-1}$,
we can bound $\mathbb{E}[\|N_j\|^2]$ by
\begin{align*}
 \mathbb{E}[\|N_j\|^2|\mathcal{F}_{j-1}] & = \mathbb{E}[\|(\bar Ax_j^\delta-\bar y^\delta)-((a_{i_j},x_j^\delta)-y_{i_j}^\delta)e_{i_j}\|^2|\mathcal{F}_{j-1}]\\
  & \leq \sum_{i=1}^nn^{-1}\|((a_{i},x_j^\delta)-y_i^\delta)e_i\|^2 = n^{-1}\|Ax_j^\delta-y^\delta\|^2,
\end{align*}
where the inequality is due to the identity $\E[((a_{i_j},x_j^\delta)-y_{i_j}^\delta)e_{i_j}|\mathcal{F}_{j-1}]
=\bar A x_j - \bar y^\delta$ and bias-variance decomposition. Thus, by taking full expectation, we obtain
\begin{align*}
   \mathbb{E}[\|N_j\|^2] \leq n^{-1}\mathbb{E}[\|Ax_j^\delta-y^\delta\|^2].
\end{align*}
Combining the preceding bounds yields the desired assertion.
\end{proof}

Last, we state a bound on the mean squared residual $\E[\|Ax_k^\delta-y^\delta\|^2]$. The proof relies
essentially on Theorem \ref{thm:err-var1} with $s=\frac12$ and Lemma \ref{lem:iter-est}. Together with Theorem
\ref{thm:err-var1} with $s=0$, it gives a bound on the stochastic error, which is crucial for
analyzing regularizing property of SGD.
\begin{theorem}\label{thm:res-noisy}
Let Assumption \ref{ass:eta} and condition \eqref{eqn:source} be fulfilled. Then, there holds
\begin{equation}\label{eqn:res-est-1}
  \E[\|Ax_{k+1}^\delta-y^\delta\|^2] \leq c_\alpha k^{-\min(\alpha,\min(1,2p)(1-\alpha))}\ln k + c_\alpha'\delta^2\max(1,\ln k)^2,
\end{equation}
where the constants $c_\alpha$ and $c'_\alpha$ depend on $\alpha$, $p$, $\|w\|$, $\|Ax_1-y^\delta\|$ and $\|A\|$.
\end{theorem}
\begin{proof}
Let $r_k=\E[\|Ax_{k}^\delta-y^\delta\|^2]$ be the mean squared residual at iteration $k$. By bias-variance
decomposition and the triangle inequality, we have
\begin{align*}
     r_{k+1} &= \|A\E[x_{k+1}^\delta]-y^\delta\|^2+\E[\|A(x_{k+1}^\delta-E[x_{k+1}^\delta])\|^2] \\
   &\leq  4\|A(\E[x_{k+1}]-x^\dag)\|^2 + 4\|A\E[x_{k+1}^\delta-x_{k+1}]\|^2+\E[\|A(x_{k+1}^\delta - \E[x_{k+1}^\delta])\|^2]+2\delta^2\\
    & :=4{\rm I}_1+4{\rm I}_2 + {\rm I}_3 + {\rm I}_4.
\end{align*}
With $c_p=(\frac{p(1-\alpha)}{c_0e(2^{1-\alpha}-1)})^{2p}\|A\|^2\|w\|^2$ and $c_\alpha=(\frac{2^\alpha(2-\alpha)}{e(1-\alpha)}+1)^2$,
Theorems \ref{thm:err-approx} and \ref{thm:err-prop} immediately imply
\begin{align*}
{\rm I}_1 \leq c_pk^{-2p(1-\alpha)} \quad \mbox{and}\quad
{\rm I}_2  \leq c_\alpha\delta^2\max(1,\ln k)^2.
\end{align*}
Next, we bound the variance ${\rm I}_3$ by
Theorem \ref{thm:err-var1} with $s=1/2$ and Lemma \ref{lem:kernel}:
\begin{align}
    {\rm I}_3
    & \leq  n\sum_{j=1}^{k} \eta_j^2\|\Pi_{j+1}^k(B)B\|^2 r_j
     \leq c_1\sum_{j=1}^{k-1}\frac{\eta_j^2}{\sum_{i=j+1}^k\eta_i}r_j + c_2 k^{-2\alpha}r_k,\label{eqn:iter-resnorm}
\end{align}
with $c_1=e^{-1}\|A\|^2$ and $c_2=c_0\|A\|^2$. Combining these estimates yields (with
$c_3=4c_p$ and $c_4=4c_\alpha+2$)
\begin{align}\label{eqn:var-recur}
    r_{k+1}  \leq c_1\sum_{j=1}^{k-1}\frac{\eta_j^2}{\sum_{i=j+1}^k\eta_i} r_j+ c_2k^{-2\alpha}r_k + c_3k^{-2p(1-\alpha)} + c_4\delta^2\max(1,\ln k)^2.
\end{align}
This and Lemma \ref{lem:iter-est} imply the desired estimate.
\end{proof}

\begin{remark}\label{rmk:res-exact}
Due to the presence of the factor
$\ln k$ in Theorem \ref{thm:res-noisy}, the upper bound is not uniform in $k$ for noisy data, but the growth is very
mild. For exact data $y^\dag$, there holds:
\begin{equation*}
  \E[\|Ax_{k+1} - y^\dag\|^2] \leq c k^{-\min(\alpha,\min(1,2p)(1-\alpha))}\ln k,
\end{equation*}
where the constant $c$ depends on $\alpha$, $p$, $\|Ax_1-y^\dag\|$ and $\|A\|$. The proof also indicates that the
condition $c_0\max_i\|a_i\|^2\leq 1$ in Assumption \ref{ass:eta} may be replaced with $c_0\|B\|\leq 1$.
\end{remark}

\section{Regularizing property}\label{sec:reg}

In this section, we analyze the regularizing property of SGD with early stopping, and prove convergence rates
under \textit{a priori} stopping rule. First, we show the convergence of the SGD iterate $x_k$ for exact data
to the minimum-norm solution $x^\dag$ defined in \eqref{eqn:sol-min}, for any $\alpha\in(0,1)$.
\begin{theorem}\label{thm:conv-ex}
Let Assumption \ref{ass:eta} be fulfilled. Then the SGD iterate $x_k$
converges to the minimum norm solution $x^\dag$ as $k\to\infty$, i.e.,
\begin{equation*}
\lim_{k\to \infty }\mathbb{E}[\|x_k-x^\dag\|^2] = 0.
\end{equation*}
\end{theorem}
\begin{proof}
The proof employs the decomposition \eqref{eqn:decomp}, and bounds separately the mean and variance.
It follows from \eqref{eqn:iter-mean} that the mean $\E[e_k]$ satisfies $\E[e_{k+1}] = \Pi_1^k(B)e_1.$
The term $\|\Pi_1^k(B)e_1\|$ converges to zero as $k\to\infty$. Specifically, we define
a function $r_k(\lambda): (0,\|B\|]\to [0,1)$ by $r_k(\lambda) =  \prod_{j=1}^k(1-\eta_k\lambda)$.
By Assumption \ref{ass:eta}, $c_0\max_i\|a_i\|^2\leq 1$, $r_k(\lambda)$ is
uniformly bounded. By the inequality $1-x\leq e^{-x}$ for $x\geq0$,
$r_k(\lambda)\leq e^{-\lambda\sum_{j=1}^k\eta_j}$. This and the identity $\lim_{k\to\infty}
\sum_{j=1}^k\eta_j=\infty$ imply that for any $\lambda>0$, $\lim_{k\to\infty}
r_k(\lambda)=0$. Hence, $r_k(\lambda)$ converges to zero pointwise, and
the argument for Theorem 4.1 of \cite{EnglHankeNeubauer:1996} yields
$\lim_{k\to\infty}\|\E[e_{k}]\|=0.$
Next, we bound the variance $\E[\|x_{k+1}-\E[x_{k+1}]\|^2]$. By Theorem
\ref{thm:err-var1} (with $s=0$) and Lemma \ref{lem:kernel} (with $p=\frac{1}{2}$),
\begin{align*}
  \mathbb{E}[\|x_{k+1}-\mathbb{E}[x_{k+1}]\|^2] & \leq \sum_{j=1}^k \eta_j^2\|\Pi_{j+1}^k(B)B^\frac{1}{2}\|^2 \mathbb{E}[\|A(x_j-x^\dag)\|^2]\\
    & \leq \sup_j\E[\|A(x_j-x^\dag)\|^2] \Big((2e)^{-1}\sum_{j=1}^{k-1} \frac{\eta_j^{2}}{\sum_{i=j+1}^k\eta_i} + c_0k^{-2\alpha}\Big).
\end{align*}
By Theorem \ref{thm:res-noisy} (and Remark \ref{rmk:res-exact}), the sequence $\{\E[\|A(x_j-x^\dag)\|^2]\}_{j=1}^\infty$ is
uniformly bounded. Then Lemma \ref{lem:basicest2} implies
\begin{equation*}
  \lim_{k\to \infty}\E[\|x_{k}-\E[x_k]\|^2]=0.
\end{equation*}
The desired assertion follows from bias variance decomposition by
\begin{equation*}
  \lim_{k\to\infty}\E[\|x_k-x^\dag\|^2] \leq \lim_{k\to\infty}\|\E[x_k]-x^\dag\|^2 + \lim_{k\to\infty}\E[\|x_k-\E[x_k]\|^2]=0.
\end{equation*}
It is well known that the minimum norm solution is characterized by $x^\dag-x_1\in \mathrm{range}(A^t)$.
By the construction of the SGD iterate $x_k$, $x_k-x_1$ always belongs to ${\rm range}(A^t)$,
and thus the limit is the unique minimum-norm solution $x^\dag$.
\end{proof}

Next we analyze the convergence of the SGD iterate $x_{k}^\delta$ for noisy data $y^\delta$ as $\delta\to0$.
To this end, we need a bound on the variance $\E[\|x_{k}^\delta-\E[x_{k}^\delta]\|^2]$ of the iterate $x_k$.
\begin{lemma}\label{lem:bdd-var}
Let Assumption \ref{ass:eta} be fulfilled.
Under the source condition \eqref{eqn:source}, there holds
\begin{align*}
  \mathbb{E}[\|x_{k+1}^\delta-\mathbb{E}[x_{k+1}^\delta]\|^2] \leq ck^{-\min(1-\alpha,\alpha+2p(1-\alpha),2\alpha)}\ln^2 k + c'\delta^2,
\end{align*}
where the constants $c$ and $c'$ depend on $\alpha$, $p$, $\|w\|$, $\|Ax_1-y^\delta\|$ and $\|A\|$.
\end{lemma}
\begin{proof}
Let $r_k=\mathbb{E}[\|Ax_k^\delta-y^\delta\|^2]$ be the expected squared residual at the $k$th iteration. Then
Theorem \ref{thm:err-var1} with $s=0$ and Lemma \ref{lem:kernel} with $p=\frac12$ imply (with $c_1=(2e)^{-1}$)
\begin{align*}
  \mathbb{E}[\|x_{k+1}^\delta-\mathbb{E}[x_{k+1}^\delta]\|^2] & \leq \sum_{j=1}^{k-1} \eta_j^2\|\Pi_{j+1}^k(B)B^\frac{1}{2}\|^2 r_j +\eta_k^2\|B^\frac{1}{2}\|^2r_k\\
   & 
     \leq c_1\sum_{j=1}^{k-1}\frac{\eta_j^2}{\sum_{i=j+1}^k\eta_i}r_j
    +c_0k^{-2\alpha}r_k.
\end{align*}
where the last step is due to $c_0\|B\|\leq 1$ from Assumption \ref{ass:eta}. Now
Theorem \ref{thm:res-noisy} gives
\begin{equation*}
  r_{k+1} \leq c_\alpha k^{-\min(\alpha,\min(1,2p)(1-\alpha))}\ln k + c_\alpha'\delta^2 \max(\ln k,1)^2.
\end{equation*}
The last two inequalities and Lemma \ref{lem:basicest2} imply the desired bound.
\end{proof}

Now we can prove the regularizing property of SGD in Theorem \ref{thm:consistency}.
\begin{proof}[Proof of Theorem \ref{thm:consistency}]
We appeal to the bias-variance decomposition \eqref{eqn:decomp}:
\begin{align*}
  \E[\|x_{k(\delta)}^\delta - x^\dag\|^2] \leq 2\|\E[x_{k(\delta)}^\delta-x_{k(\delta)}]\|^2 + 2\|\E[x_{k(\delta)}]-x^\dag\|^2 + \E[\|x_{k(\delta)}^\delta - \E[x_{k(\delta)}^\delta]\|^2].
\end{align*}
By the proof of Theorem \ref{thm:conv-ex} and condition \eqref{eqn:cond-k}, we have
\begin{equation*}
\lim_{\delta\to0^+}\|\E[x_{k(\delta)}]-x^\dag\|=\lim_{k\to\infty}\|\E[x_k]-x^\dag\|=0.
\end{equation*}
Thus, it suffices to analyze the errors $\|\E[x_{k(\delta)}^\delta-x_{k(\delta)}]\|^2$ and
$\E[\|x_{k(\delta)}^\delta-\E[x_{k(\delta)}^\delta]\|^2]$. By Theorem
\ref{thm:err-prop} and the choice of $k(\delta)$ in condition \eqref{eqn:cond-k}, there holds
\begin{equation*}
  \lim_{\delta\to0^+} \|\E[x_{k(\delta)}-x_{k(\delta)}^\delta]\| =0.
\end{equation*}
Last, by Lemma \ref{lem:bdd-var} and condition \eqref{eqn:cond-k}, we can
bound the variance $\E[\|x_{k(\delta)}^\delta-\E[x_{k(\delta)}^\delta]\|^2]$ by
\begin{equation*}
  \lim_{\delta\to 0^+}\E[\|x_{k(\delta)}^\delta-\E[x_{k(\delta)}^\delta]\|^2] = 0.
\end{equation*}
Combining the last three estimates completes the proof.
\end{proof}

\begin{remark}
The consistency condition \eqref{eqn:cond-k} in Theorem \ref{thm:consistency} requires $\alpha\in (0,1)$.
The constant step size, i.e., $\alpha=0$, is not covered by the theory, for which the bootstrapping
argument does not work.
\end{remark}

Last, we give the proof of Theorem \ref{thm:err-total} on the convergence rate of SGD under a priori stopping rule.
\begin{proof}[Proof of Theorem \ref{thm:err-total}]
By bias-variance decomposition, we have
\begin{align*}
  \mathbb{E}[\|x_{k+1}^\delta - x^\dag\|^2] = \E[\|x_{k+1}^\delta - \E [x_{k+1}^\delta]\|^2] + \|\E[x_{k+1}^\delta]-x^\dag\|^2.
\end{align*}
It follows from Lemma \ref{lem:bdd-var} that
\begin{align*}
  \mathbb{E}[\|x_{k+1}^\delta-\mathbb{E}[x_{k+1}^\delta]\|^2] \leq ck^{-\min(1-\alpha,2p(1-\alpha)+\alpha,2\alpha)}\ln^2 k + c'\delta^2.
\end{align*}
Meanwhile, by the triangle inequality and Theorems \ref{thm:err-approx} and \ref{thm:err-prop},
\begin{align*}
  \|\E[x_{k+1}^\delta]-x^\dag\|^2 & \leq 2c_p^2k^{-2p(1-\alpha)} + 2c_\alpha^2 k^{1-\alpha}\bar\delta^2.
\end{align*}
These  two estimates together give the desired rate.
\end{proof}

\begin{remark}
The \textit{a priori} parameter choice in Theorem \ref{thm:err-total} requires a knowledge of the regularity index $p$,
and thus is infeasible in practice. The popular discrepancy principle also does not work
directly due to expensive residual evaluation, and further, it induces complex dependence between
the iterates, which requires different techniques for the analysis. Thus, it is of much
interest to develop purely data-driven rules without residual evaluation while automatically adapting to
the unknown solution regularity, e.g., quasi-optimality criterion and balancing principle \cite{JinLorenz:2010,
PereverzevSchock:2005}.
\end{remark}

\section{Preasymptotic convergence}\label{app:preasympt}

In this part, we present the proofs of Theorems \ref{thm:asympt-weak-noisy} and \ref{thm:asympt-noisy}
on the preasymptotic weak and strong convergence, respectively. First, we briefly discuss the low-frequency
dominance on the initial error $e_1$ under the source condition \eqref{eqn:source}: if the singular
values $\sigma_i$ of $n^{-\frac12}A$ decay fast, $e_1$ is indeed dominated by $P_\mathcal{L}e_1$, i.e., $\|P_\mathcal{L}e_1\|
\gg\|P_\mathcal{H}e_1\|$. We illustrate this with a simple probabilistic model: the
sourcewise representer $w\in\mathbb{R}^m$ follows the standard Gaussian distribution $\mathcal{N}(0,I_m)$.

\begin{prop}\label{prop:source}
In Condition \eqref{eqn:source}, if $w\sim \mathcal{N}(0,I_m)$, then there hold
\begin{equation*}
  \mathbb{E}[\|P_\mathcal{L}e_1\|^2] = \sum_{i=1}^L\sigma_i^{4p}\quad \mbox{and}\quad \mathbb{E}[\|P_\mathcal{H}e_1\|^2] = \sum_{i=L+1}^r\sigma_i^{4p}.
\end{equation*}
\end{prop}
\begin{proof}
Under Condition \eqref{eqn:source}, we have $e_1 = B^p w = V\Sigma^{2p}V^tw$. Thus, we have
\begin{align*}
  \|P_\mathcal{L}e_1\|^2 &= \|\sum_{i=1}^L V_i\sigma_i^{2p}(V^tw)_i\|^2 = \sum_{i=1}^L \sigma_i^{4p}(V^tw)_i^2.
\end{align*}
Since $w\sim \mathcal{N}(0,I_m)$ and the matrix $V$ is orthonormal, $(V^tw)_i\sim \mathcal{N}(0,1)$, and
$\mathbb{E}[(V^tw)_i^2]=1$, from which the assertion on $\mathbb{E}[\|P_\mathcal{L}e_1\|^2]$ follows,
and the other estimate follows similarly.
\end{proof}

\begin{remark}
For polynomially decaying singular values $\sigma_i$, i.e., $\sigma_i=ci^{-\beta}$, $\beta>0$, and if $4p\beta>1$,
simple computation shows that $\mathbb{E}[\|P_{\mathcal{L}}e_1\|^2]  \geq c^4(4p\beta-1)^{-1}(1-(L+1)^{1-4p\beta})$ and
$\mathbb{E}[\|P_{\mathcal{H}}e_1\|^2]  \leq c^4(4p\beta-1)^{-1}(L^{1-4p\beta}-m^{1-4p\beta})$, and thus
\begin{equation*}
  \frac{\mathbb{E}[\|P_{\mathcal{L}}e_1\|^2]}{\mathbb{E}[\|P_{\mathcal{H}}e_1\|^2]}\geq \frac{1-(L+1)^{1-4p\beta}}{L^{1-4p\beta}-m^{1-4p\beta}}.
\end{equation*}
Hence, for a truncation level $L\ll m$ and $4p\beta\gg1$, $\mathbb{E}[\|P_\mathcal{L}e_1\|^2]$ is dominating.
The condition $4p\beta\gg1$ holds for either severely ill-posed problems {\rm(}large $\beta${\rm)} or highly
regular solution {\rm(}large $p${\rm)}.
\end{remark}

Now we give the proof of the preasymptotic weak convergence in Theorem \ref{thm:asympt-weak-noisy}.
\begin{proof}[Proof of Theorem \ref{thm:asympt-weak-noisy}]
By applying $P_\mathcal{L}$ to the SGD iteration \eqref{eqn:iter-noisy}, we have
\begin{equation*}
  P_\mathcal{L}e_{k+1}^\delta = P_\mathcal{L}e_k^\delta -\eta_k(a_{i_k},e_k^\delta)P_\mathcal{L}a_{i_k}+\eta_k\xi_{i_k}P_\mathcal{L}a_{i_k}.
\end{equation*}
By taking conditional expectation with respect to $\mathcal{F}_{k-1}$, since $e_k^\delta = P_\mathcal{L}e_k^\delta
+ P_\mathcal{H}e_k^\delta$, we obtain
\begin{align*}
  \mathbb{E}[P_\mathcal{L}e_{k+1}^\delta|\mathcal{F}_{k-1}] & =  P_\mathcal{L}e_k^\delta - \eta_kn^{-1}\sum_{i=1}^n(a_{i},e_k^\delta)P_\mathcal{L}a_{i}+ \eta_kn^{-1}\sum_{i=1}^n\xi_iP_\mathcal{L}a_i\\
   &= P_\mathcal{L}e_k^\delta-\eta_kP_\mathcal{L}Be_k^\delta +\eta_kP_\mathcal{L}\bar A^t\xi\\
   & = (I-\eta_kP_\mathcal{L}BP_\mathcal{L}) P_\mathcal{L}e_k^\delta - \eta_kP_\mathcal{L}B P_He_k^\delta+\eta_kP_\mathcal{L}\bar A^t \xi.
\end{align*}
By the construction of $P_\mathcal{L}$ and $P_\mathcal{H}$, $P_\mathcal{L}BP_He_k^\delta=0$, and then taking full expectation yields
\begin{align*}
  \mathbb{E}[P_\mathcal{L}e_{k+1}^\delta]
   &= (I-\eta_kP_\mathcal{L}BP_\mathcal{L}) \mathbb{E}[P_\mathcal{L}e_k^\delta]+\eta_kP_\mathcal{L}\bar A^t \xi.
\end{align*}
Then the first assertion follows since $\|\bar A^t\| = n^{-\frac12}\|B\|^\frac{1}{2}\leq n^{-\frac12}c_0^{-\frac12}$,
$\|P_\mathcal{L}\bar A^t\xi\| \leq c_0^{-\frac{1}2}\bar\delta$, and
$\|(I-\eta_kP_\mathcal{L}BP_\mathcal{L})P_\mathcal{L}e_k\|\geq (1-\eta_k\sigma_L^2)\|P_\mathcal{L}e_k\|$.
Next, appealing again to the SGD iteration \eqref{eqn:iter-noisy} gives
\begin{equation*}
  P_\mathcal{H}e_{k+1}^\delta = P_\mathcal{H}e_k^\delta -\eta_k(a_{i_k},e_k^\delta)P_\mathcal{H}a_{i_k} + \eta_k \xi_{i_k}P_\mathcal{H}a_{i_k}.
\end{equation*}
Thus the conditional expectation $\mathbb{E}[P_\mathcal{H}e_{k+1}|\mathcal{F}_{k-1}]$ is given by
\begin{align*}
  \mathbb{E}[P_\mathcal{H}e_{k+1}^\delta|\mathcal{F}_{k-1}] & = P_\mathcal{H}e_k^\delta - \eta_k n^{-1}
    \sum_{i=1}^n(a_i,e_k^\delta)P_\mathcal{H}a_i +\eta_kn^{-1}\sum_{i=1}^n\xi_iP_\mathcal{H}a_i\\
    & = (I-\eta_kP_\mathcal{H}BP_\mathcal{H})P_\mathcal{H}e_k^\delta +\eta_kP_\mathcal{H}\bar A^t\xi.
\end{align*}
Then, taking full expectation and appealing to the triangle inequality yield the second estimate.
\end{proof}

\begin{remark}
For exact data $y^\dag$, we obtain the following simplified expressions:
\begin{align*}
  \|\mathbb{E}[P_\mathcal{L}e_{k+1}]\| \leq (1-\eta_k\sigma_L^2) \|\mathbb{E}[P_\mathcal{L}e_k]\|\quad\mbox{and}\quad
  \|\mathbb{E}[P_\mathcal{H}e_{k+1}]\| \leq \|\mathbb{E}[P_\mathcal{H}e_k]\|.
\end{align*}
Thus the low-frequency error always decreases faster than the high-frequency one in the weak sense. Further,
there is no interaction between the low- and high-frequency errors in the weak error.
\end{remark}

Next we analyze preasymptotic strong convergence of SGD. We first analyze exact data $y^\dag$. The argument
is needed for the proof of Theorem \ref{thm:asympt-noisy}.
\begin{lemma}\label{thm:asympt-exact}
If $\eta_k\leq c_0 $ such that $c_0\max_i\|a_i\|\leq 1$, then with $c_1=\sigma_L^2$ and
$c_2=\sum_{i=L+1}^r\sigma_i^2$, there hold
\begin{align*}
  \mathbb{E}[\|P_\mathcal{L}e_{k+1}\|^2 |\mathcal{F}_{k-1}] &\leq (1-\eta_kc_1) \|P_\mathcal{L}e_k\|^2 + c_2c_0^{-1}\eta_k^2\|P_\mathcal{H}e_k\|^2,\\
  \mathbb{E}[\|P_\mathcal{H}e_{k+1}\|^2|\mathcal{F}_{k-1}] &\leq c_2c_0^{-1}\eta_k^2\|P_\mathcal{L}e_k\|^2 +(1+c_2c_0^{-1}\eta_k^2)\|P_\mathcal{H}e_k\|^2.
\end{align*}
\end{lemma}
\begin{proof}
It follows from the SGD iteration \eqref{eqn:iter-ex} that
$
  P_\mathcal{L}e_{k+1} = P_\mathcal{L}e_k -\eta_k(a_{i_k},e_k)P_\mathcal{L}a_{i_k}.
$
This and the condition $c_0\max_i\|a_i\|^2\leq 1$, imply
\begin{align*}
  \|P_\mathcal{L}e_{k+1}\|^2 & = \|P_\mathcal{L}e_k\|^2 - 2\eta_k(a_{i_k},e_k)(P_\mathcal{L}e_k,P_\mathcal{L}a_{i_k}) + \eta_k^2 (e_k,a_{i_k})^2\|P_\mathcal{L}a_{i_k}\|^2\\
   & \leq \|P_\mathcal{L}e_k\|^2 - 2\eta_k(a_{i_k},e_k)(P_\mathcal{L}e_k,P_\mathcal{L}a_{i_k}) + c_0^{-1}\eta_k^2(e_k,a_{i_k})^2.
\end{align*}
The conditional expectation with respect to $\mathcal{F}_{k-1}$ is given by
\begin{align*}
  \mathbb{E}[\|P_\mathcal{L}e_{k+1}\|^2|\mathcal{F}_{k-1}] & \leq \|P_\mathcal{L}e_k\|^2 - 2\eta_kn^{-1}\sum_{i=1}^n(a_i,e_k)(P_\mathcal{L}e_k,P_\mathcal{L}a_i)
  + c_0^{-1}\eta_k^2n^{-1}\sum_{i=1}^n(e_k,a_i)^2\\
    &= \|P_\mathcal{L}e_k\|^2 - 2\eta_k (P_\mathcal{L}e_k,P_\mathcal{L}Be_k) + c_0^{-1}\eta_k^2(e_k,Be_k).
\end{align*}
With the splitting $e_k=P_\mathcal{L}e_k+P_\mathcal{H}e_k$ and the construction of $P_\mathcal{L}$
and $P_\mathcal{H}$, we obtain
\begin{align*}
  (P_\mathcal{L}e_k,P_\mathcal{L}Be_k) & = (P_\mathcal{L}e_k,P_\mathcal{L}B P_\mathcal{L}e_k),\\
  (e_k,B e_k)& = (P_\mathcal{L}e_k,P_\mathcal{L}BP_\mathcal{L}e_k) + (P_\mathcal{H}e_k,P_\mathcal{H}BP_\mathcal{H}e_k).
\end{align*}
Substituting the last two identities leads to
\begin{align*}
  \mathbb{E}[\|P_\mathcal{L}e_{k+1}\|^2|\mathcal{F}_{k-1}] & \leq \|P_\mathcal{L}e_k\|^2 - \eta_k(P_\mathcal{L}e_k,P_\mathcal{L}BP_\mathcal{L}e_k)+
  c_0^{-1}\eta_k^2(P_\mathcal{H}e_k,P_\mathcal{H}BP_\mathcal{H}e_k)\\
    & \leq (1-\eta_k\sigma_L^2)\|P_\mathcal{L}e_k\|^2 + c_0^{-1}\eta_k^2\sigma_{L+1}^2\|P_\mathcal{H}e_k\|^2\\
    &\leq (1-c_1\eta_k)\|P_\mathcal{L}e_k\|^2 + c_2c_0^{-1}\eta_k^2\|P_\mathcal{H}e_k\|^2.
\end{align*}
This shows the first estimate. Next, appealing again to the SGD iteration \eqref{eqn:iter-ex}, we obtain
\begin{equation*}
  P_\mathcal{H}e_{k+1} = P_\mathcal{H}e_k -\eta_k(a_{i_k},e_k)P_\mathcal{H}a_{i_k},
\end{equation*}
which together with the condition $c_0\max_i\|a_i\|^2\leq 1$, and the Cauchy-Schwarz inequality, implies
\begin{align*}
  \|P_\mathcal{H}e_{k+1}\|^2 & = \|P_\mathcal{H}e_k\|^2 - 2\eta_k(a_{i_k},e_k)(P_\mathcal{H}e_k,P_\mathcal{H}a_{i_k}) + \eta_k^2 (e_k,a_{i_k})^2\|P_\mathcal{H}a_{i_k}\|^2\\
   & \leq \|P_\mathcal{H}e_k\|^2 - 2\eta_k(a_{i_k},e_k)(P_\mathcal{H}e_k,P_\mathcal{H}a_{i_k}) + c_0^{-1}\eta_k^2\|e_k\|^2\|P_\mathcal{H}a_{i_k}\|^2.
\end{align*}
Thus the conditional expectation $\mathbb{E}[\|P_\mathcal{H}e_{k+1}\|^2|\mathcal{F}_{k-1}]$ is given by
\begin{align*}
  \mathbb{E}[\|P_\mathcal{H}e_{k+1}\|^2|\mathcal{F}_{k-1}] & \leq \|P_\mathcal{H}e_k\|^2 - 2\eta_kn^{-1}\sum_{i=1}^n(a_i,e_k)
  (P_\mathcal{H}a_i,P_\mathcal{H}e_k) + c_0^{-1}\eta_k^2n^{-1}\|e_k\|^2\sum_{i=1}^n\|P_\mathcal{H}a_i\|_F^2\\
    &= \|P_\mathcal{H}e_k\|^2 - 2\eta_k (P_\mathcal{H}e_k,P_\mathcal{H}Be_k) + c_0^{-1}\eta_k^2\|e_k\|^2\|P_\mathcal{H}B^\frac{1}{2}\|^2_F.
\end{align*}
Upon observing the identity $\|P_\mathcal{H}B^\frac12\|_F^2 = \sum_{i=L+1}^r\sigma_i^2\equiv c_2$ \cite[Lemma 3.2]{JiaoJinLu:2017}, we deduce
\begin{align*}
  \mathbb{E}[\|P_\mathcal{H}e_{k+1}\|^2|\mathcal{F}_{k-1}] & \leq \|P_\mathcal{H}e_k\|^2 - 2\eta_k\|B^\frac12P_\mathcal{H}e_k\|^2+ c_2c_0^{-1}\eta_k^2\|e_k\|^2\\
    & \leq \|P_\mathcal{H}e_k\|^2 + c_2c_0^{-1}\eta_k^2(\|P_\mathcal{L}e_k\|^2 + \|P_\mathcal{H}e_k\|^2).
\end{align*}
This proves the second estimate and completes the proof of the lemma.
\end{proof}

\begin{remark}
The proof gives a slightly sharper estimate on the low-frequency error:
\begin{align*}
  \mathbb{E}[\|P_\mathcal{L}e_{k+1}\|^2|\mathcal{F}_{k-1}]
    & \leq (1-\eta_k\sigma_L^2)\|P_\mathcal{L}e_k\|^2 + c_0^{-1}\eta_k^2\sigma_{L+1}^2\|P_\mathcal{H}e_k\|^2.
\end{align*}
\end{remark}

Now we can present the proof of Theorem \ref{thm:asympt-noisy} on preasymptotic strong convergence.
\begin{proof}[Proof of Theorem \ref{thm:asympt-noisy}]
It follows from the SGD iteration \eqref{eqn:iter-noisy} that
\begin{equation*}
  P_\mathcal{L}e_{k+1}^\delta = P_\mathcal{L}e_k^\delta -\eta_k(a_{i_k},e_k^\delta)P_\mathcal{L}a_{i_k} + \eta_k\xi_{i_k}P_\mathcal{L}a_{i_k},
\end{equation*}
and upon expansion, we obtain
\begin{align*}
  \mathbb{E}[\|P_\mathcal{L}e_{k+1}^\delta\|^2|\mathcal{F}_{k-1}] = \mathbb{E}[ \|P_\mathcal{L}e_k^\delta& -\eta_k(a_{i_k},e_k^\delta)P_\mathcal{L}a_{i_k}\|^2|\mathcal{F}_{k-1}]+  \eta_k^2\mathbb{E}[\xi_{i_k}^2\|P_\mathcal{L}a_{i_k}\|^2|\mathcal{F}_{k-1}] \\
  & +2\mathbb{E}[(P_\mathcal{L}e_k^\delta -\eta_k(a_{i_k},e_k^\delta)P_\mathcal{L}a_{i_k}, \eta_k\xi_{i_k}P_\mathcal{L}a_{i_k})|\mathcal{F}_{k-1}]
   := {\rm I}_1 + {\rm I}_2 + {\rm I}_3.
\end{align*}
It suffices to bound the three terms ${\rm I}_i$. The term ${\rm I}_1$ can be bounded by the
argument in Lemma \ref{thm:asympt-exact} as
\begin{equation}\label{eqn:est1}
  {\rm I}_1 \leq (1-\eta_kc_1) \|P_\mathcal{L}e_k^\delta\|^2 + c_2c_0^{-1}\eta_k^2\|P_\mathcal{H}e_k^\delta\|^2.
\end{equation}
For the term ${\rm I}_2$, by Assumption \ref{ass:eta}, there holds
${\rm I}_2 \leq \eta_k^2n^{-1}\max_i\|P_\mathcal{L}a_i\|^2\sum_{i=1}^n\xi_i^2 \leq c_0^{-1}\eta_k^2 \bar\delta^2$.
For the third term ${\rm I}_3$, by the identity $(a_i,e_k^\delta)=(P_\mathcal{L}a_i,P_\mathcal{L}e_k^\delta)
+(P_\mathcal{H}a_i,P_\mathcal{H}e_k^\delta)$, we have
\begin{align*}
  {\rm I}_3=2 n^{-1}\eta_k\sum_{i=1}^n\xi_i[(P_\mathcal{L}a_i,P_\mathcal{L}e_k^\delta) -\eta_k(a_{i},e_k^\delta)
  \|P_\mathcal{L}a_{i}\|^2] = 2 n^{-1}\eta_k\sum_{i=1}^n\xi_i {\rm I}_{3,i},
\end{align*}
with
${\rm I}_{3,i} =(1-\eta_k\|P_\mathcal{L}a_i\|^2)(P_\mathcal{L}a_i,P_\mathcal{L}e_k^\delta) - \eta_k(P_\mathcal{H}a_{i},P_\mathcal{H}e_k^\delta)
\|P_\mathcal{L}a_{i}\|^2$.
It suffices to bound ${\rm I}_{3,i}$. By the condition on $\eta_k$, we deduce
\begin{align*}
  {\rm I}_{3,i}^2 & = 2(1-\eta_k\|P_\mathcal{L}a_i\|^2)^2(P_\mathcal{L}a_i,P_\mathcal{L}e_k^\delta)^2 + 2\eta_k^2\|P_\mathcal{L}a_i\|^4(P_\mathcal{H}a_{i},P_\mathcal{H}e_k^\delta)^2\\
    & \leq 2(P_\mathcal{L}a_i,P_\mathcal{L}e_k^\delta)^2 + 2(P_\mathcal{H}a_{i},P_\mathcal{H}e_k^\delta)^2,
\end{align*}
and consequently,
\begin{align*}
  \sum_{i=1}^n{\rm I}_{3,i}^2 &\leq 2\sum_{i=1}^n\left((P_\mathcal{L}a_i,P_\mathcal{L}e_k^\delta)^2 + (P_\mathcal{H}a_{i},P_\mathcal{H}e_k^\delta)^2 \right)= 2\| A^t e_k^\delta\|^2 \leq 2n \|B\|\|e_k^\delta\|^2.
\end{align*}
Combining these two estimates with the Cauchy-Schwarz inequality leads to $|{\rm I}_3|\leq 2\sqrt{2}\bar\delta\eta_k\sigma_1\|e_k^\delta\|$.
The bounds on ${\rm I}_1$, ${\rm I}_2$ and ${\rm I}_3$ together show the first assertion. For the high-frequency part $P_He_k^\delta$, we have
\begin{equation*}
  P_\mathcal{H}e_{k+1}^\delta = P_\mathcal{H}e_k^\delta -\eta_k(a_{i_k},e_k^\delta)P_\mathcal{H}a_{i_k} +\eta_k\xi_{i_k}P_\mathcal{H}a_{i_k},
\end{equation*}
and upon expansion, we obtain
\begin{align*}
  \mathbb{E}[\|P_\mathcal{H}e_{k+1}^\delta\|^2|\mathcal{F}_{k-1}] &= \mathbb{E}[\|P_\mathcal{H}e_k^\delta -\eta_k(a_{i_k},e_k^\delta)P_\mathcal{H}a_{i_k}\|^2|\mathcal{F}_{k-1}]
  +  \eta_k^2\mathbb{E}[\xi_{i_k}^2\|P_\mathcal{H}a_{i_k}\|^2|\mathcal{F}_{k-1}] \\
  &\quad+2\mathbb{E}[(P_\mathcal{H}e_k^\delta -\eta_k(a_{i_k},e_k^\delta)P_\mathcal{H}a_{i_k}, \eta_k\xi_{i_k}P_\mathcal{H}a_{i_k})|\mathcal{F}_{k-1}]
:= {\rm I}_4 + {\rm I}_5 + {\rm I}_6.
\end{align*}
The term ${\rm I}_4$ can be bounded by the argument in Lemma \ref{thm:asympt-exact} as
\begin{equation*}
  {\rm I}_4 \leq c_2c_0^{-1}\eta_k^2\|P_\mathcal{L}e_k^\delta\|^2 +(1+c_2c_0^{-1}\eta_k^2)\|P_\mathcal{H}e_k^\delta\|^2.
\end{equation*}
 Clearly, ${\rm I}_5 \leq c_0^{-1}\eta_k^2\bar\delta^2$.
For  ${\rm I}_6$, simple computation yields
\begin{equation*}
  {\rm I}_6 = 2n^{-1}\eta_k\sum_{i=1}^n \xi_i[(P_\mathcal{H}a_i,P_\mathcal{H}e_k^\delta)-\eta_k(a_i,e_k)\|P_\mathcal{H}a_i\|^2] :=2n^{-1}\eta_k\sum_{i=1}^n\xi_i{\rm I}_{6,i},
\end{equation*}
with ${\rm I}_{6,i}$ given by ${\rm I}_{6,i} = (P_\mathcal{H}a_i,P_\mathcal{H}e_k^\delta)
-\eta_k(a_i,e_k)\|P_\mathcal{H}a_i\|^2$. Simple computation shows
\begin{align*}
  \sum_{i=1}^n{\rm I}_{6,i}^2 & \leq 2 \sum_{i=1}^n\Big((P_\mathcal{H}a_i,P_\mathcal{H}e_k^\delta)^2+\eta_k^2(a_i,e_k)^2\|P_\mathcal{H}a_i\|^4\Big)\\
    & \leq \big(2\|P_\mathcal{H}e_k^\delta\|^2 + 2\eta_k^2\max_{i}\|a_i\|^4\|e_k^\delta\|^2\big) \sum_{i=1}^n\|P_\mathcal{H}a_i\|^2\\
    & \leq 2c_2n(\|P_\mathcal{H}e_k^\delta\|^2 + c_0^{-2}\eta_k^2\|e_k^\delta\|^2),
\end{align*}
where the last line is due to the identity $\|P_\mathcal{H}B^\frac12\|_F^2 = \sum_{i=L+1}^r\sigma_i^2\equiv c_2$ \cite[Lemma 3.2]{JiaoJinLu:2017}.
This estimate together with the Cauchy-Schwarz inequality gives
\begin{equation*}
    |{\rm I}_6| \leq 2\sqrt{2}c_2^\frac{1}{2}\eta_k \bar\delta\Big(\|P_\mathcal{H}e_k^\delta\|^2 + c_0^{-2}\eta_k^2\|e_k^\delta\|^2\Big)^\frac{1}{2}.
\end{equation*}
These estimates together show the second assertion, and complete the proof.
\end{proof}

\section{Numerical experiments}\label{sec:numer}

Now we present numerical experiments to complement the theoretical study. All the numerical
examples, i.e., \texttt{phillips}, \texttt{gravity} and \texttt{shaw}, are taken from the public domain
\texttt{MATLAB} package \textbf{Regutools}\footnote{Available from \url{http://www.imm.dtu.dk/~pcha/Regutools/},
last accessed on January 8, 2018}. They are Fredholm integral equations of the first kind, with the first
example being mildly ill-posed, and the other two severely ill-posed. Unless otherwise stated,
the examples are discretized with a dimension $n=m=1000$. The noisy data $y^\delta$ is generated from the
exact data $y^\dag$ as
\begin{equation*}
  y^\delta_i = y_i^\dag + \delta \max_{j}(|y_j^\dag|)\xi_i,\quad i =1,\ldots,n,
\end{equation*}
where $\delta$ is the relative noise level, and the random variables $\xi_i$s follow the standard Gaussian
distribution. The initial guess $x_1$ is fixed at
$x_1=0$. We present the mean squared error $e_k$ and/or residual $r_k$, i.e.,
\begin{equation}\label{eqn:err-res}
  e_k =\mathbb{E}[\|x^\dag-x_k\|^2]\quad \mbox{and}\quad r_k = \mathbb{E}[\|Ax_k-y^\delta\|^2].
\end{equation}
The expectation $\mathbb{E}[\cdot]$ with respect to the random index $i_k$ is approximated
by the average of 100 independent runs. The constant $c_0$ in the step size schedule is always
taken to be $c_0=1/\max_i\|a_i\|^2$, and the exponent $\alpha$ is taken to be $\alpha=0.1$, unless otherwise
stated. All the computations were carried out on a personal laptop
with 2.50 GHz CPU and 8.00G RAM by \texttt{MATLAB} 2015b.

\subsection{The role of the exponent $\alpha$}

The convergence of SGD depends essentially on the parameter $\alpha$. To examine its role, we present in
Figs. \ref{fig:phil-sgd-al}, \ref{fig:grav-sgd-al} and \ref{fig:shaw-sgd-al} the numerical results for the
examples with different noise levels, computed using different $\alpha$ values. The smaller the $\alpha$
value is, the quicker the algorithm reaches the convergence and the iterate diverges for noisy data.
This agrees with the analysis in Section \ref{sec:reg}. Hence, a smaller $\alpha$ value is desirable
for convergence. However, in the presence of large noise, a too small $\alpha$ value may sacrifice the
attainable accuracy; see Figs. \ref{fig:phil-sgd-al}(c) and \ref{fig:grav-sgd-al}(c) for illustrations;
and also the oscillation magnitudes of the iterates and the residual tend to be larger. This is possibly
due to the intrinsic variance for large step sizes, and it would be interesting to precisely characterize
the dynamics, e.g., with stochastic differential equations \cite{LiTaiE:2017}. In practice, the
variations may cause problems with a proper stopping rule (especially with only one single trajectory).

\begin{figure}[hbt!]
  \centering
  \setlength{\tabcolsep}{0pt}
  \begin{tabular}{ccc}
   \includegraphics[trim={2cm 0 2cm 0},clip,width=.3\textwidth]{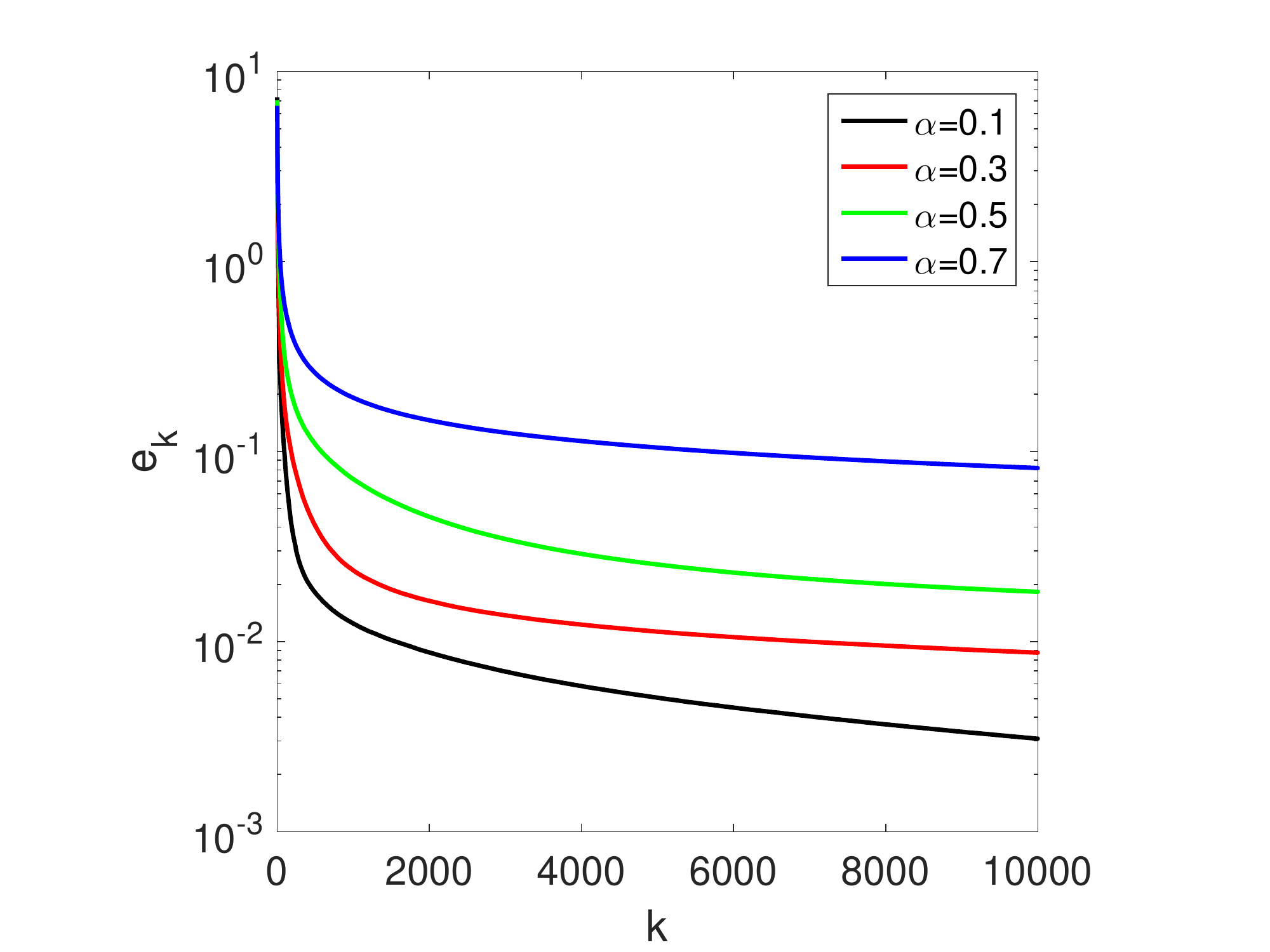} & \includegraphics[trim={2cm 0 2cm 0},clip,width=.30\textwidth]{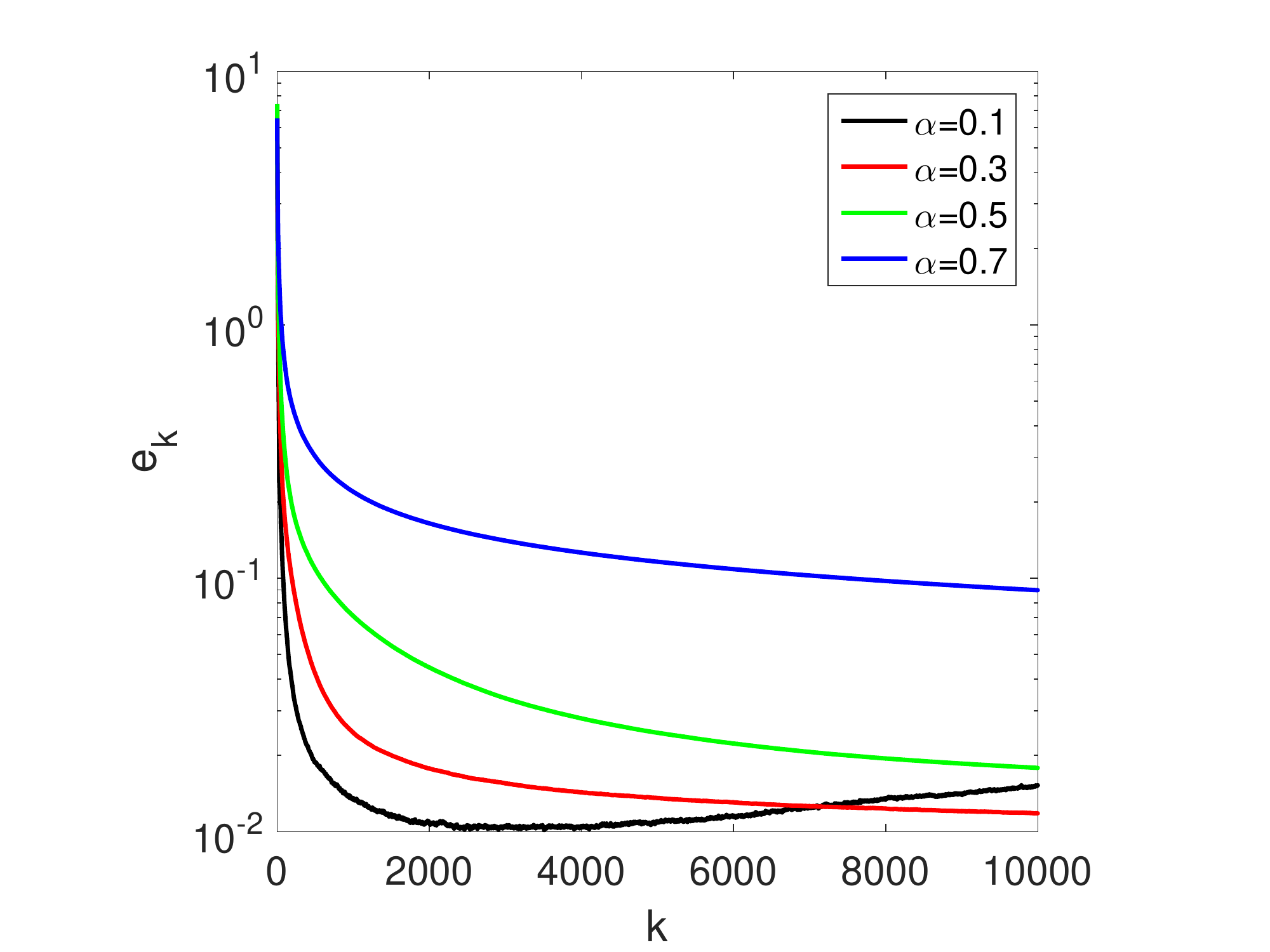} &\includegraphics[trim={2cm 0 2cm 0},clip,width=.30\textwidth]{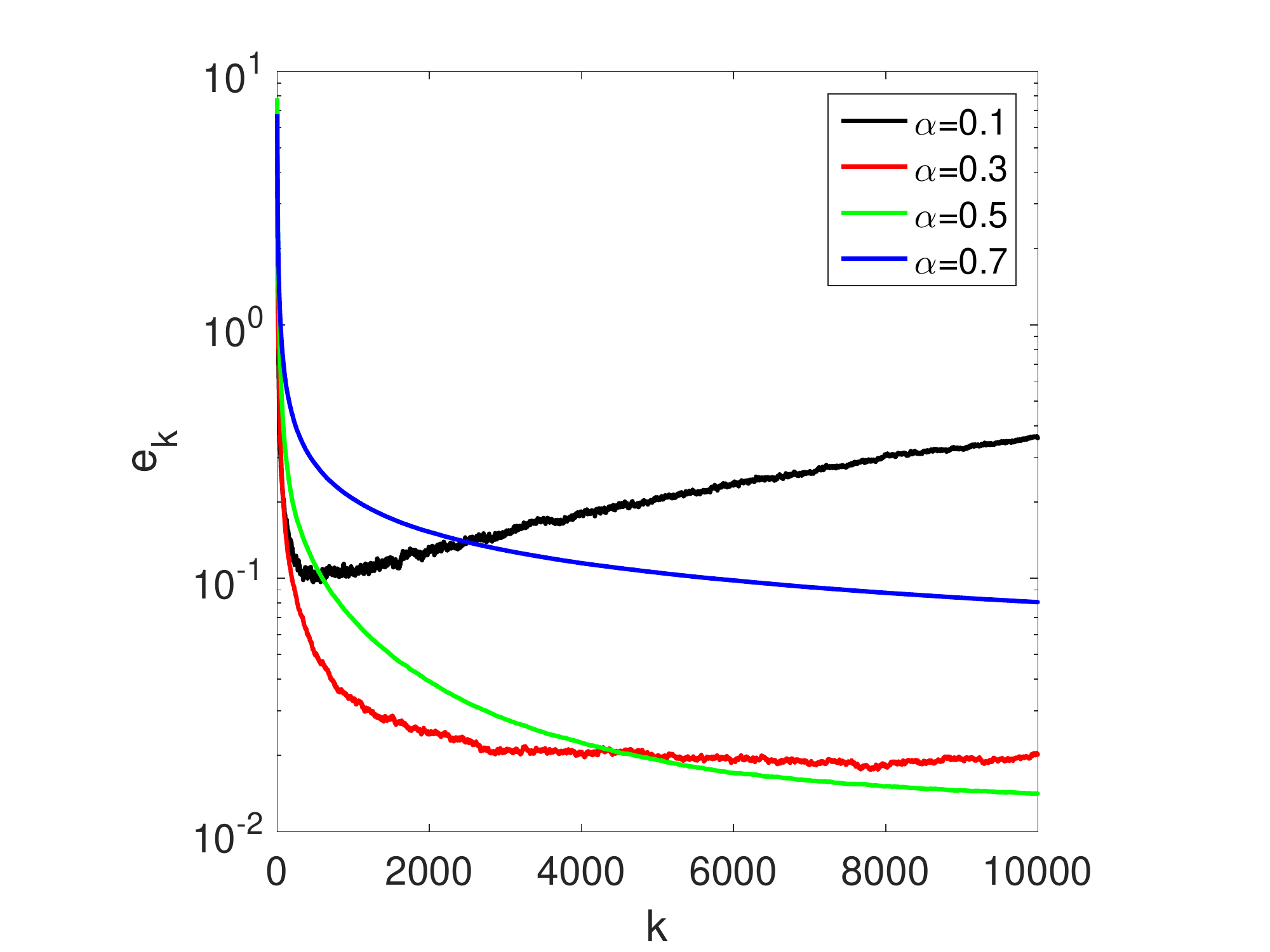}\\
  \includegraphics[trim={2cm 0 2cm 0},clip,width=.30\textwidth]{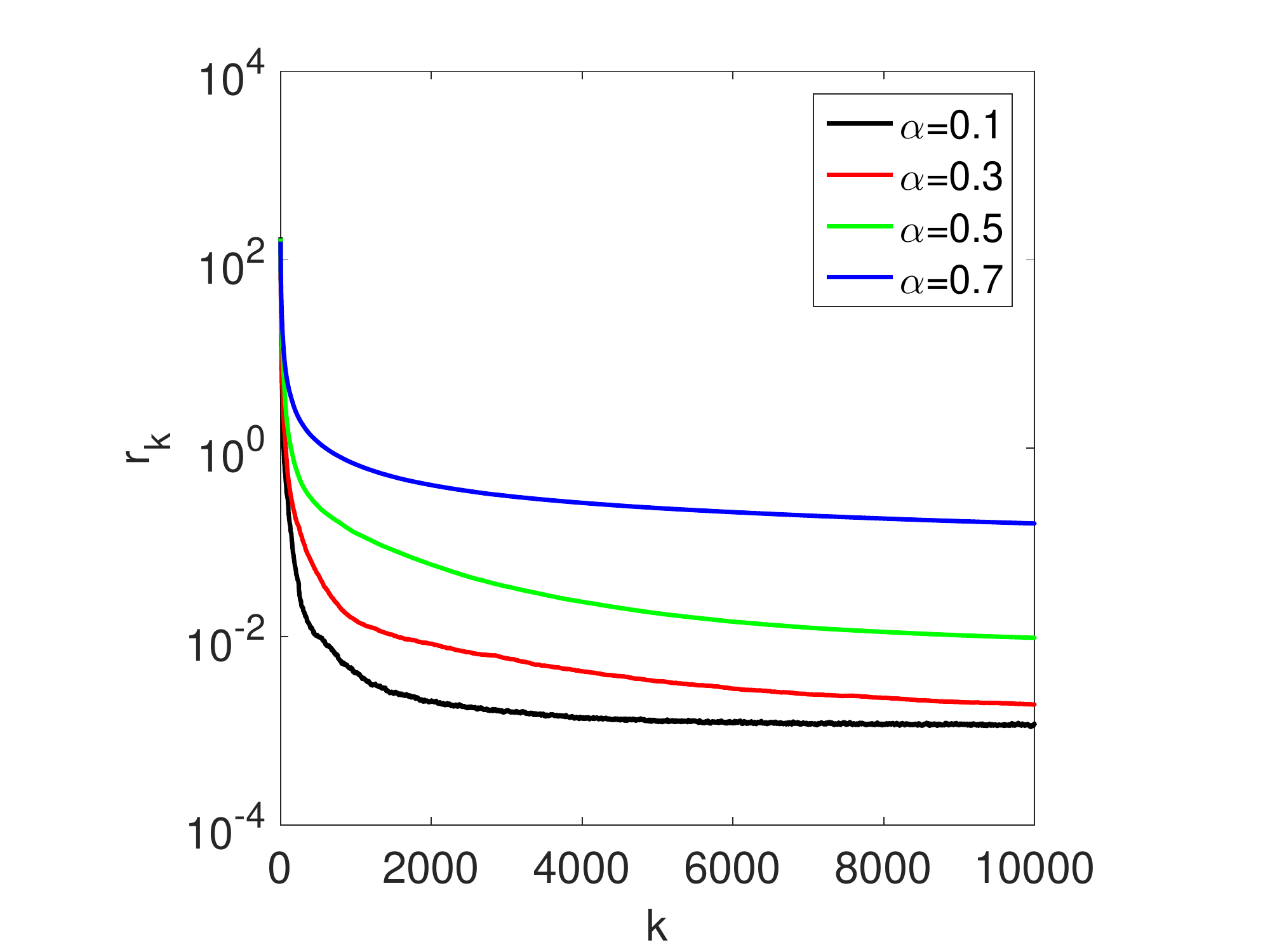} & \includegraphics[trim={2cm 0 2cm 0},clip,width=.30\textwidth]{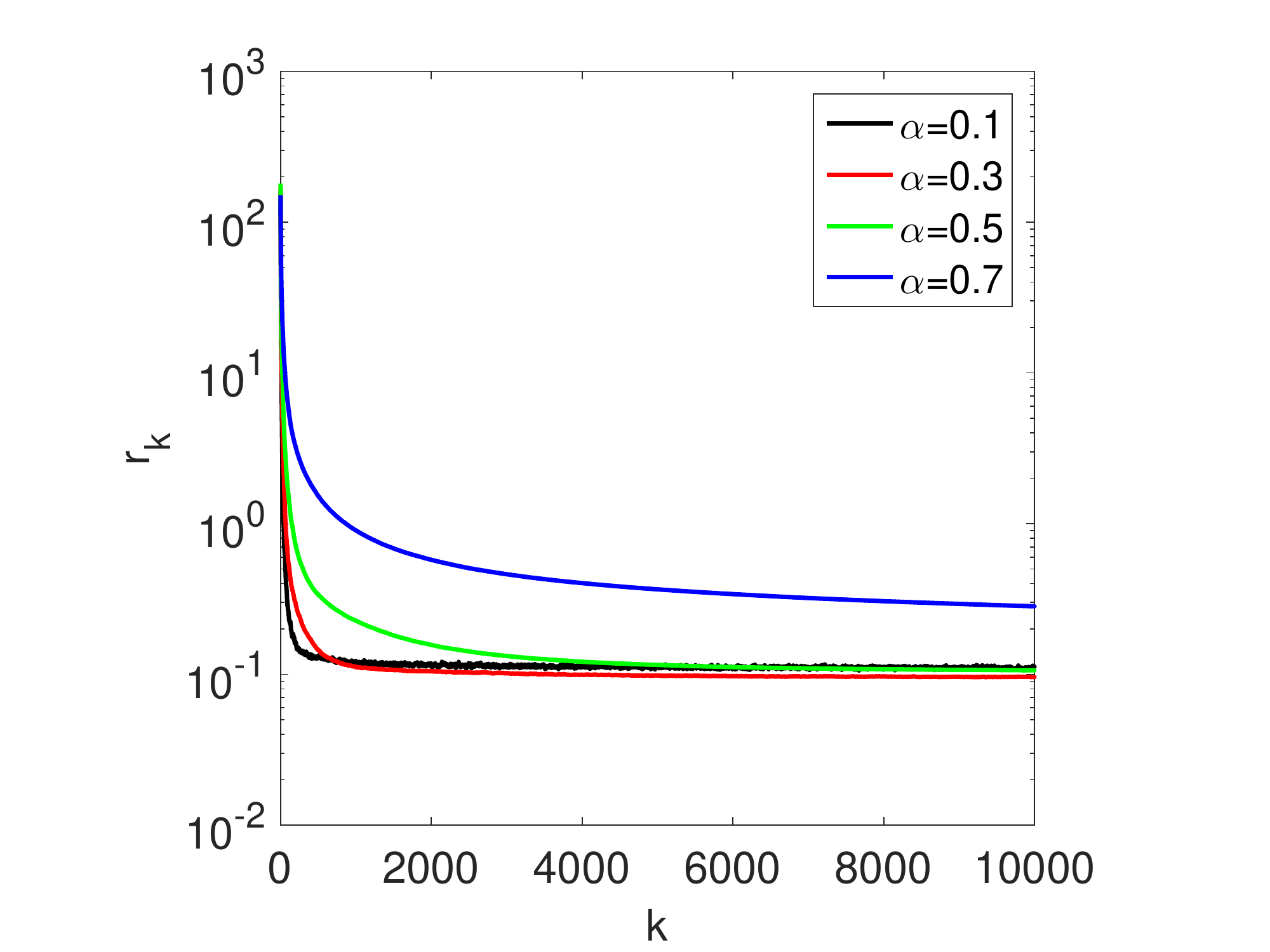} &\includegraphics[trim={2cm 0 2cm 0},clip,width=.30\textwidth]{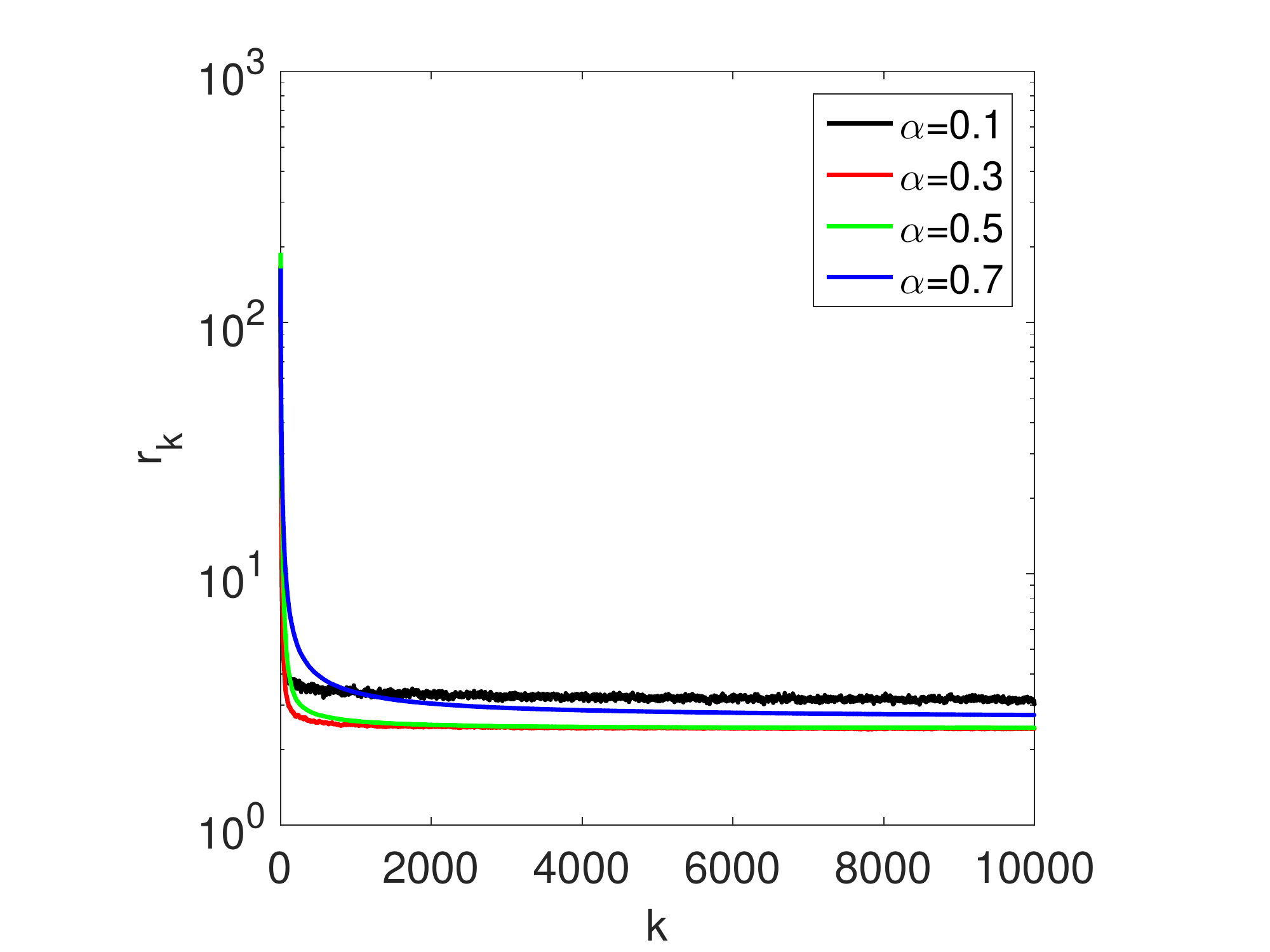}\\
  (a) $\delta=\text{1e-3}$ & (b) $\delta=\text{1e-2}$ & (c) $\delta = \text{5e-2}$
  \end{tabular}
  \caption{Numerical results for \texttt{phillips} with different noise levels by SGD (with various $\alpha$).\label{fig:phil-sgd-al}}
\end{figure}

\begin{figure}[hbt!]
  \centering
  \setlength{\tabcolsep}{0pt}
  \begin{tabular}{ccc}
   \includegraphics[trim={2cm 0 2cm 0},clip,width=.30\textwidth]{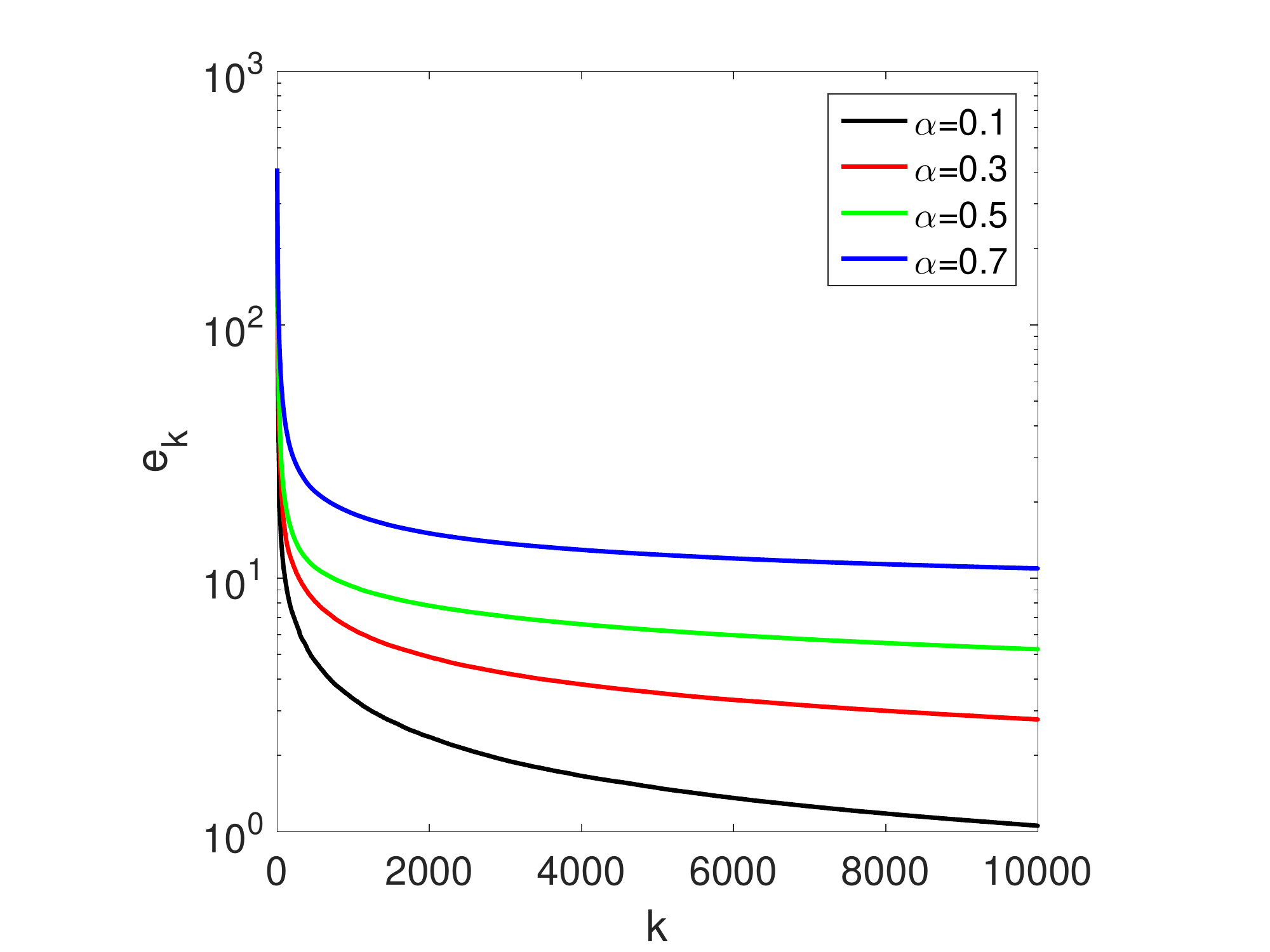} & \includegraphics[trim={2cm 0 2cm 0},clip,width=.30\textwidth]{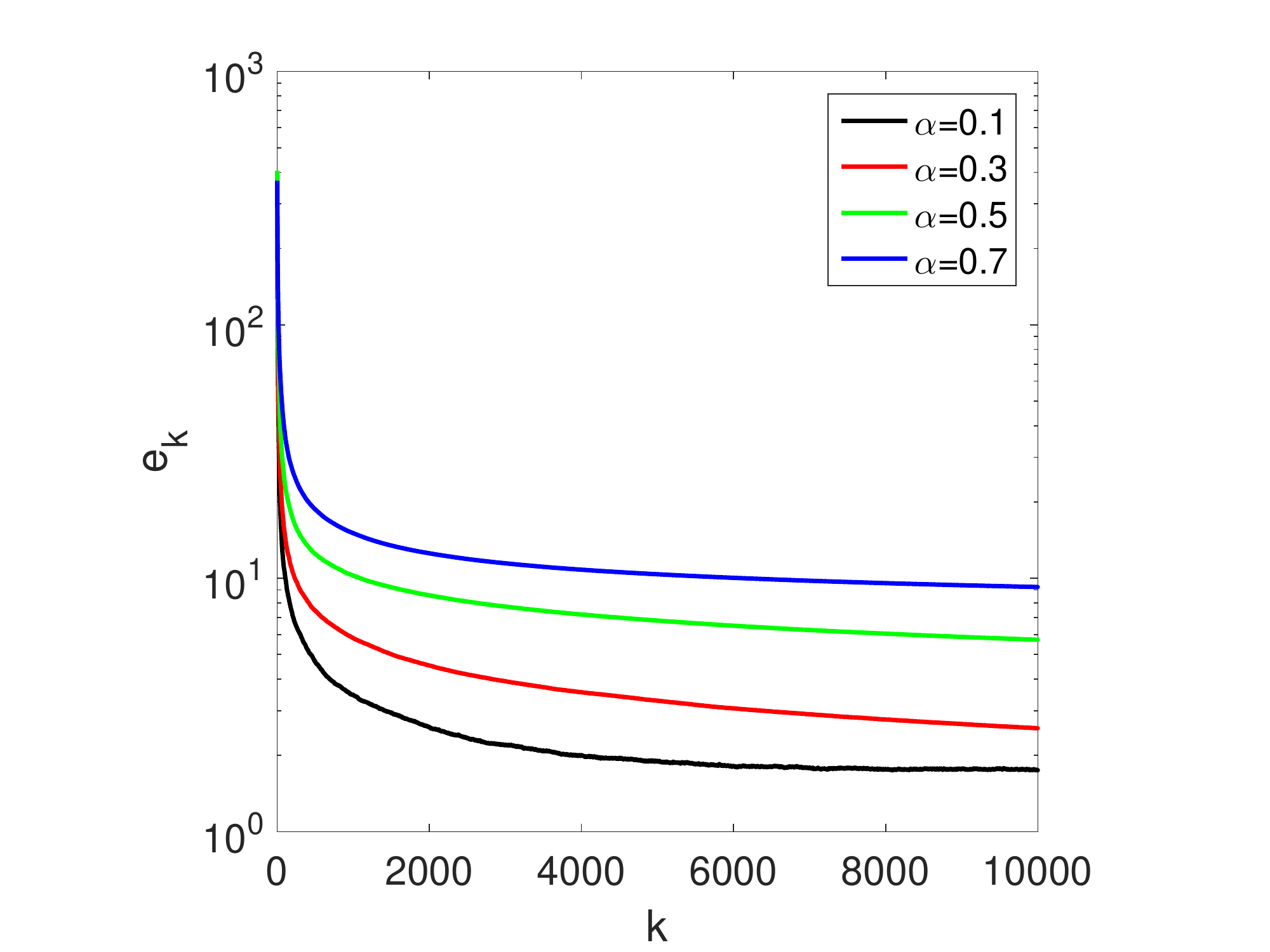} &\includegraphics[trim={2cm 0 2cm 0},clip,width=.30\textwidth]{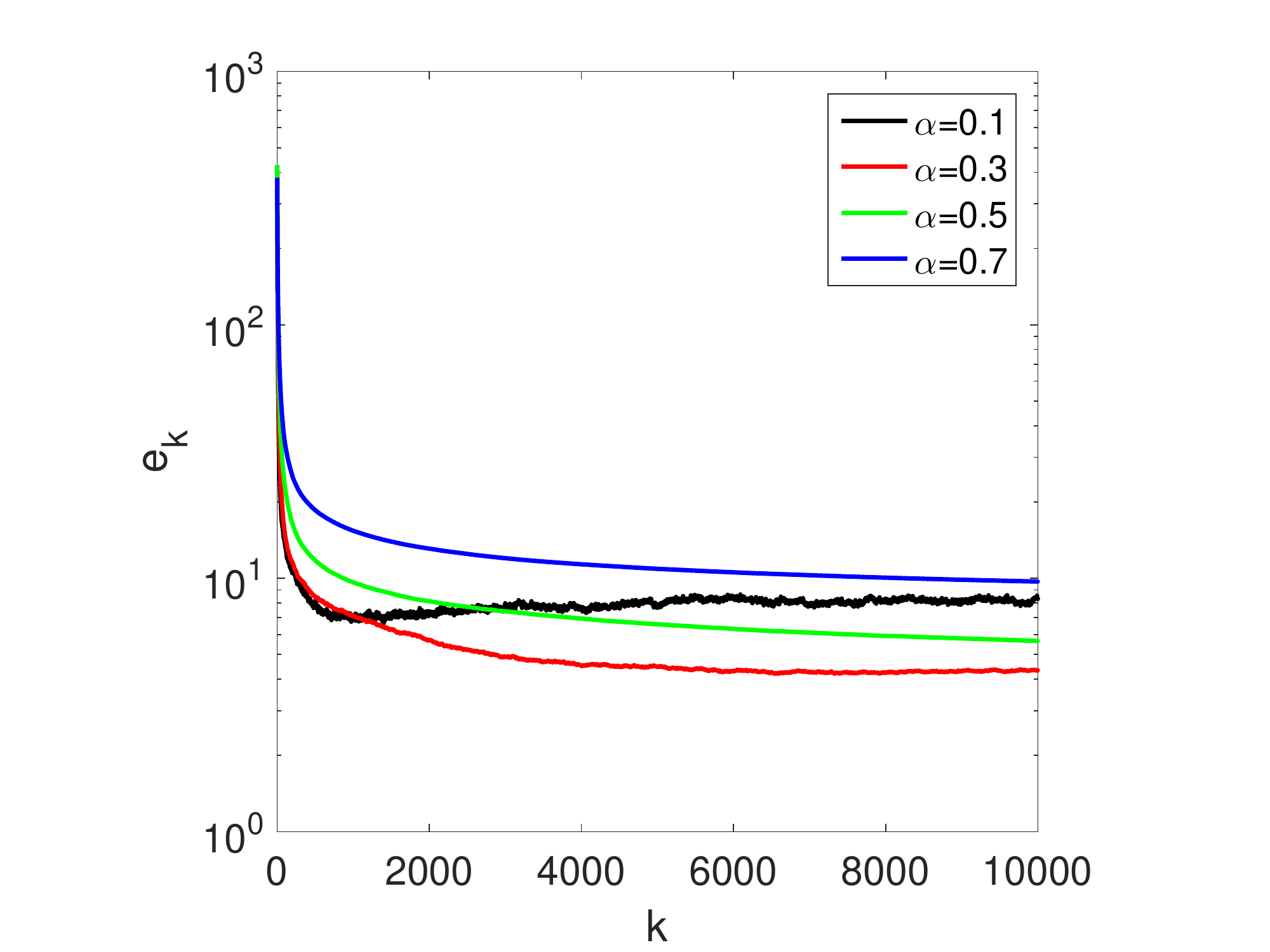}\\
  \includegraphics[trim={2cm 0 2cm 0},clip,width=.30\textwidth]{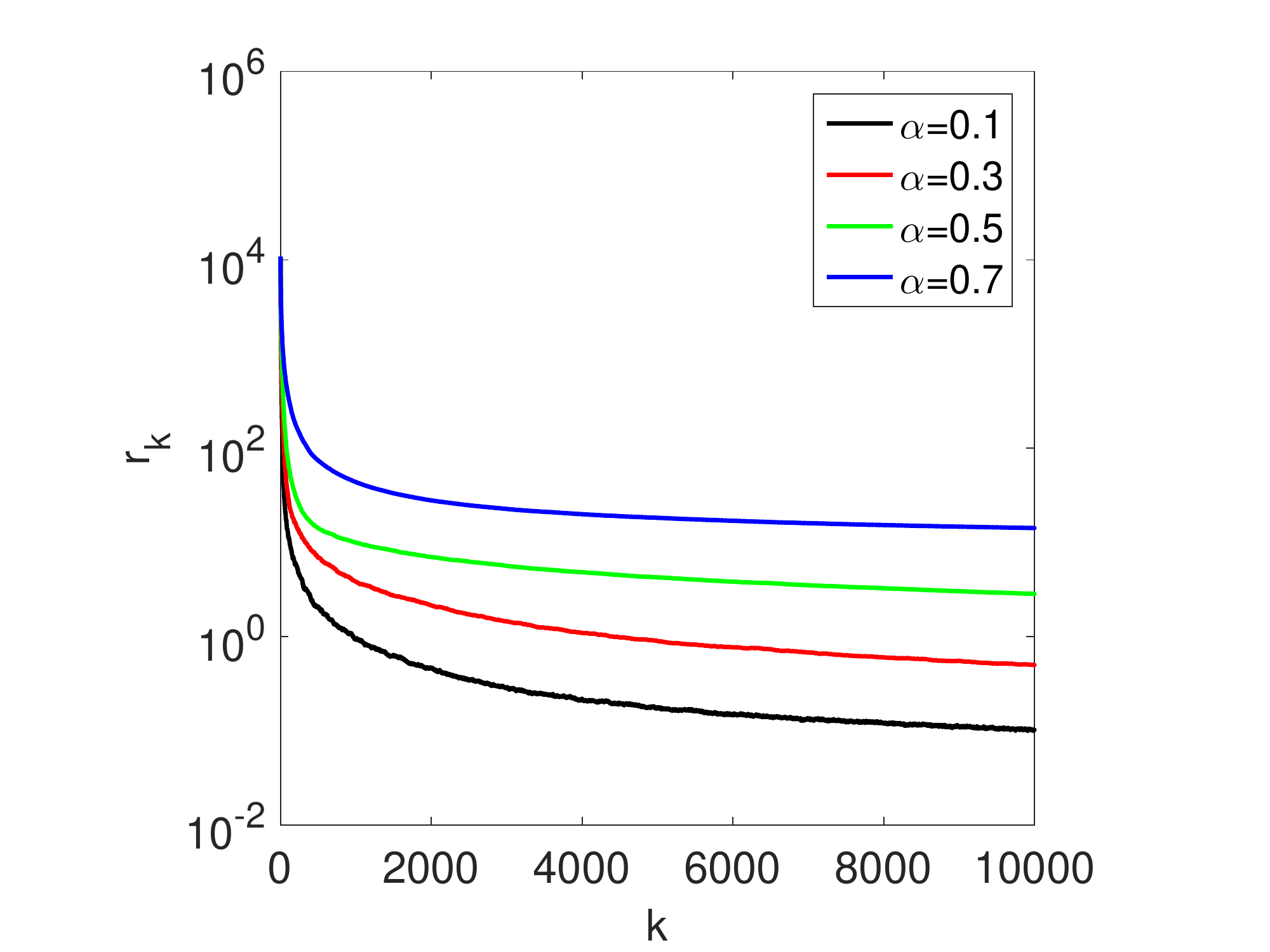} & \includegraphics[trim={2cm 0 2cm 0},clip,width=.30\textwidth]{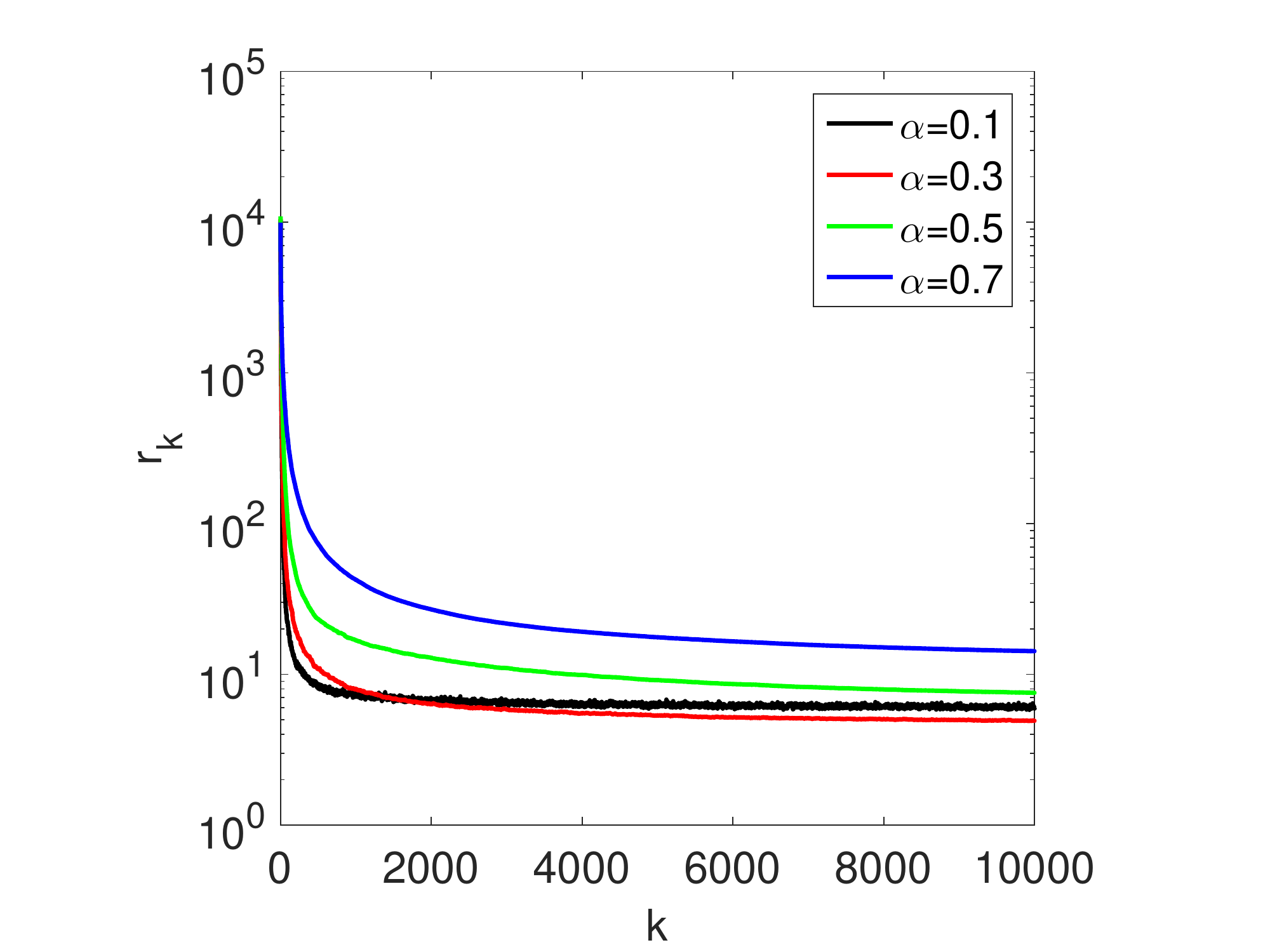} &\includegraphics[trim={2cm 0 2cm 0},clip,width=.30\textwidth]{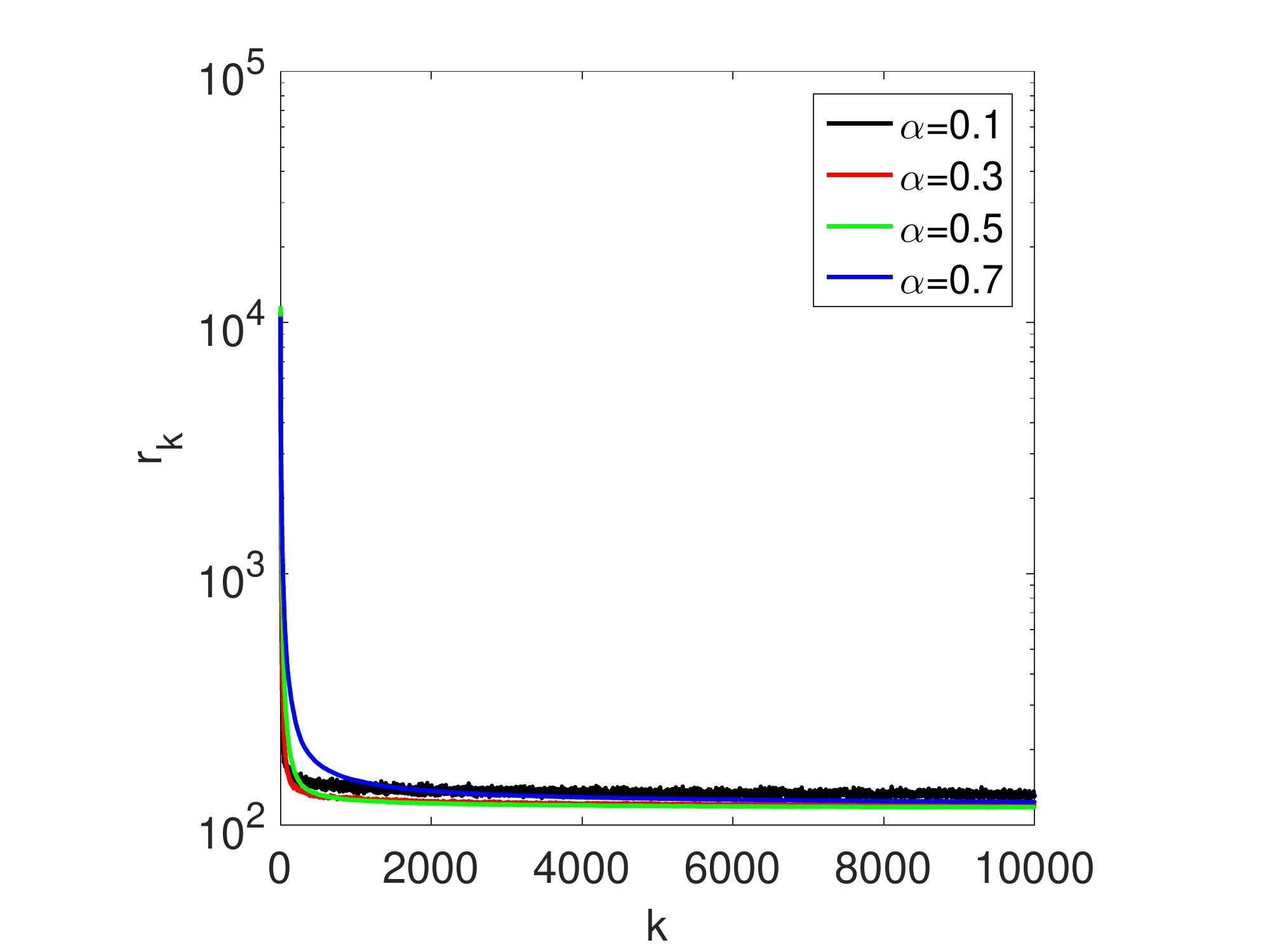}\\
  (a) $\delta=\text{1e-3}$ & (b) $\delta=\text{1e-2}$ & (c) $\delta = \text{5e-2}$
  \end{tabular}
  \caption{Numerical results for \texttt{gravity} with different noise levels by SGD (with various $\alpha$).\label{fig:grav-sgd-al}}
\end{figure}

\begin{figure}[hbt!]
  \centering
  \setlength{\tabcolsep}{0pt}
  \begin{tabular}{ccc}
   \includegraphics[trim={2cm 0 2cm 0},clip,width=.30\textwidth]{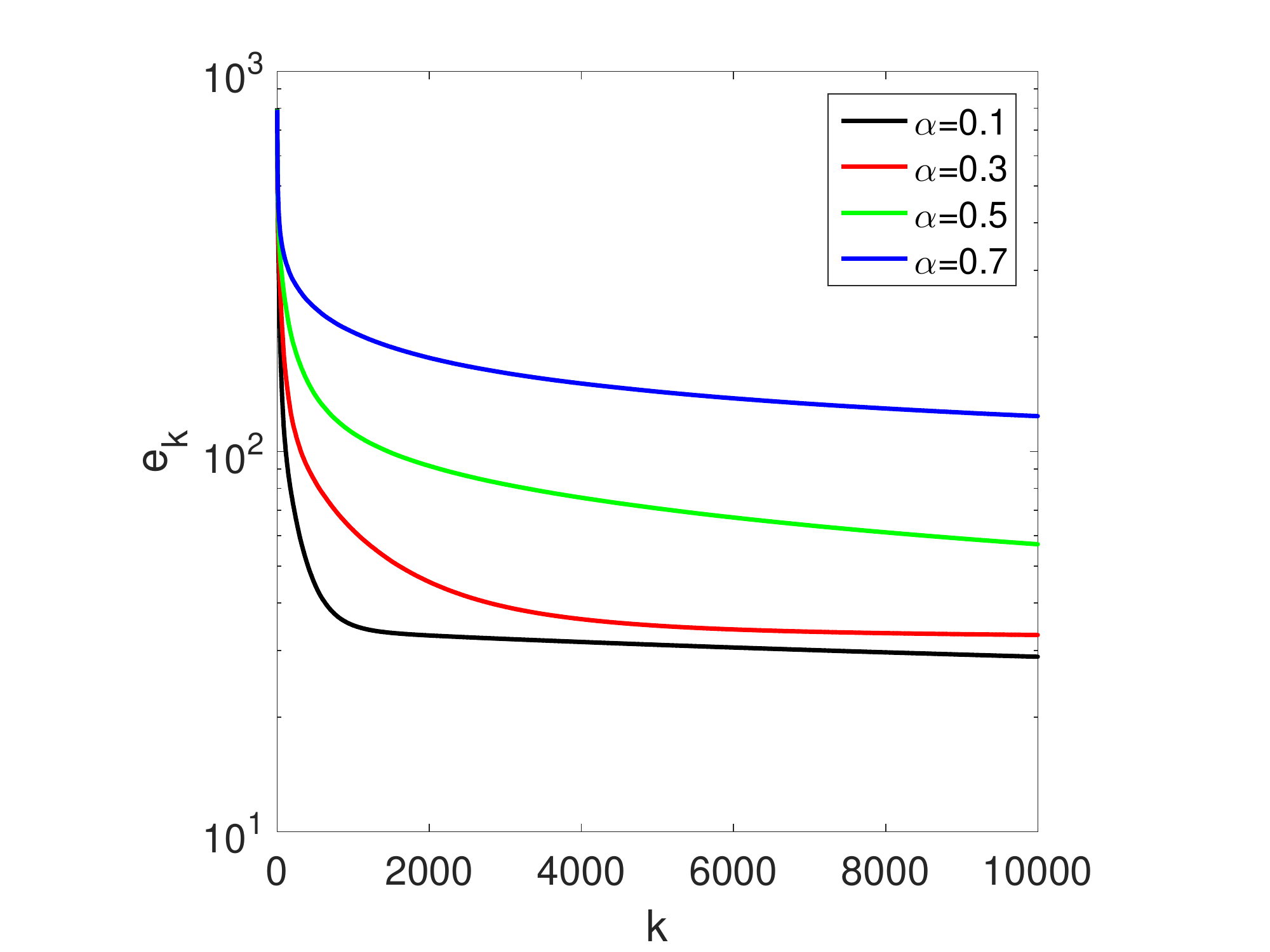} & \includegraphics[trim={2cm 0 2cm 0},clip,width=.30\textwidth]{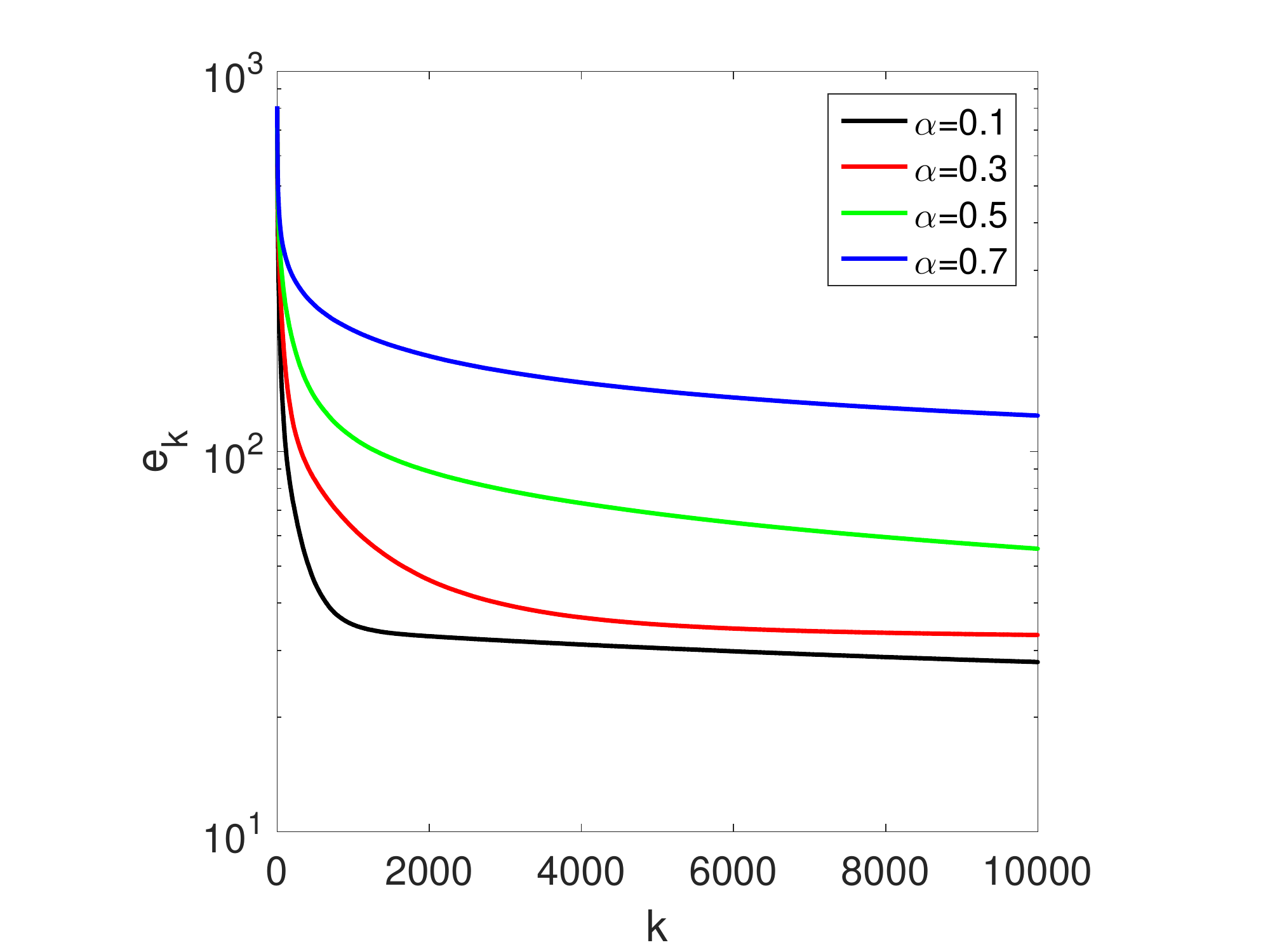} &\includegraphics[trim={2cm 0 2cm 0},clip,width=.30\textwidth]{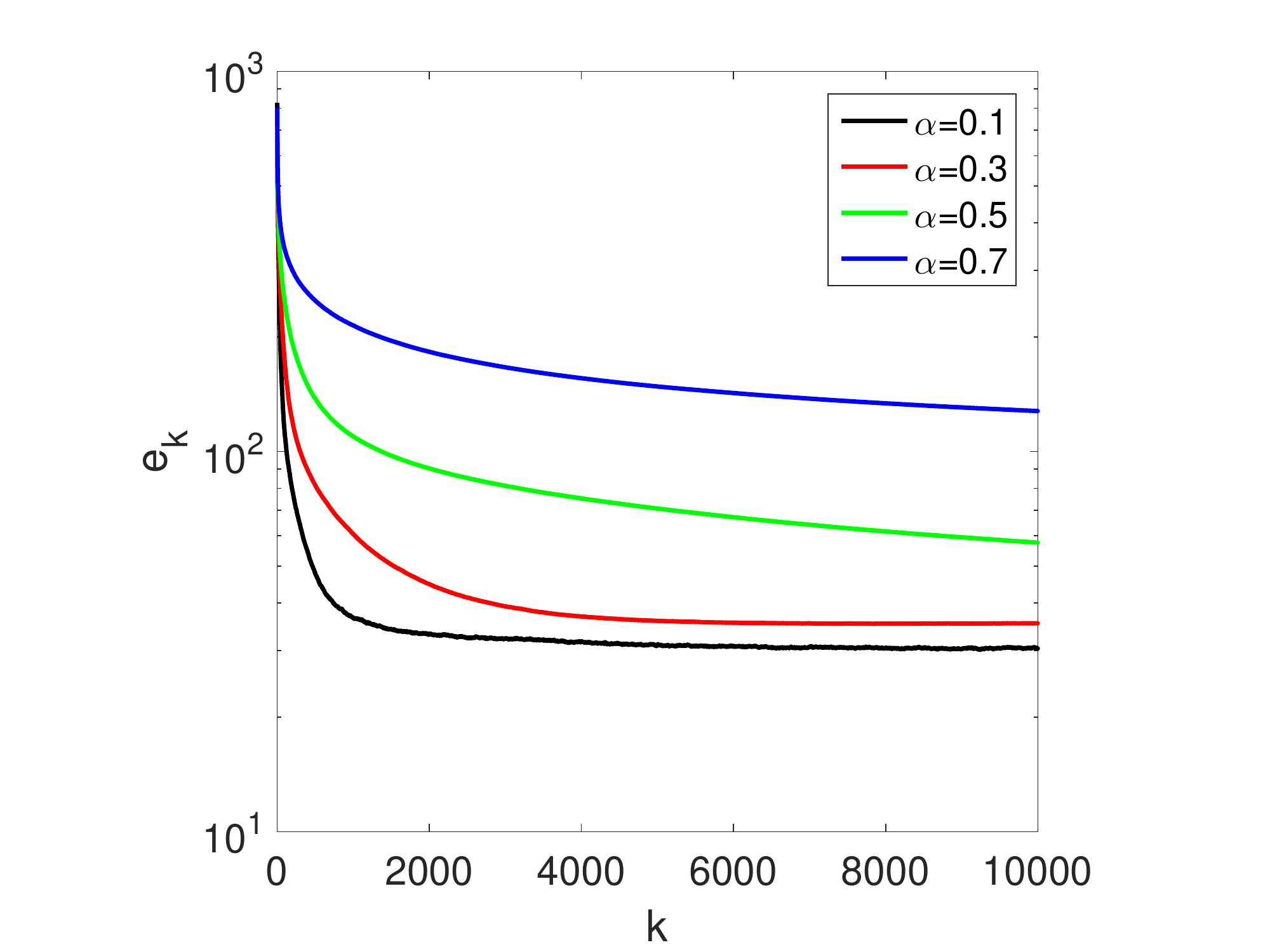}\\
  \includegraphics[trim={2cm 0 2cm 0},clip,width=.30\textwidth]{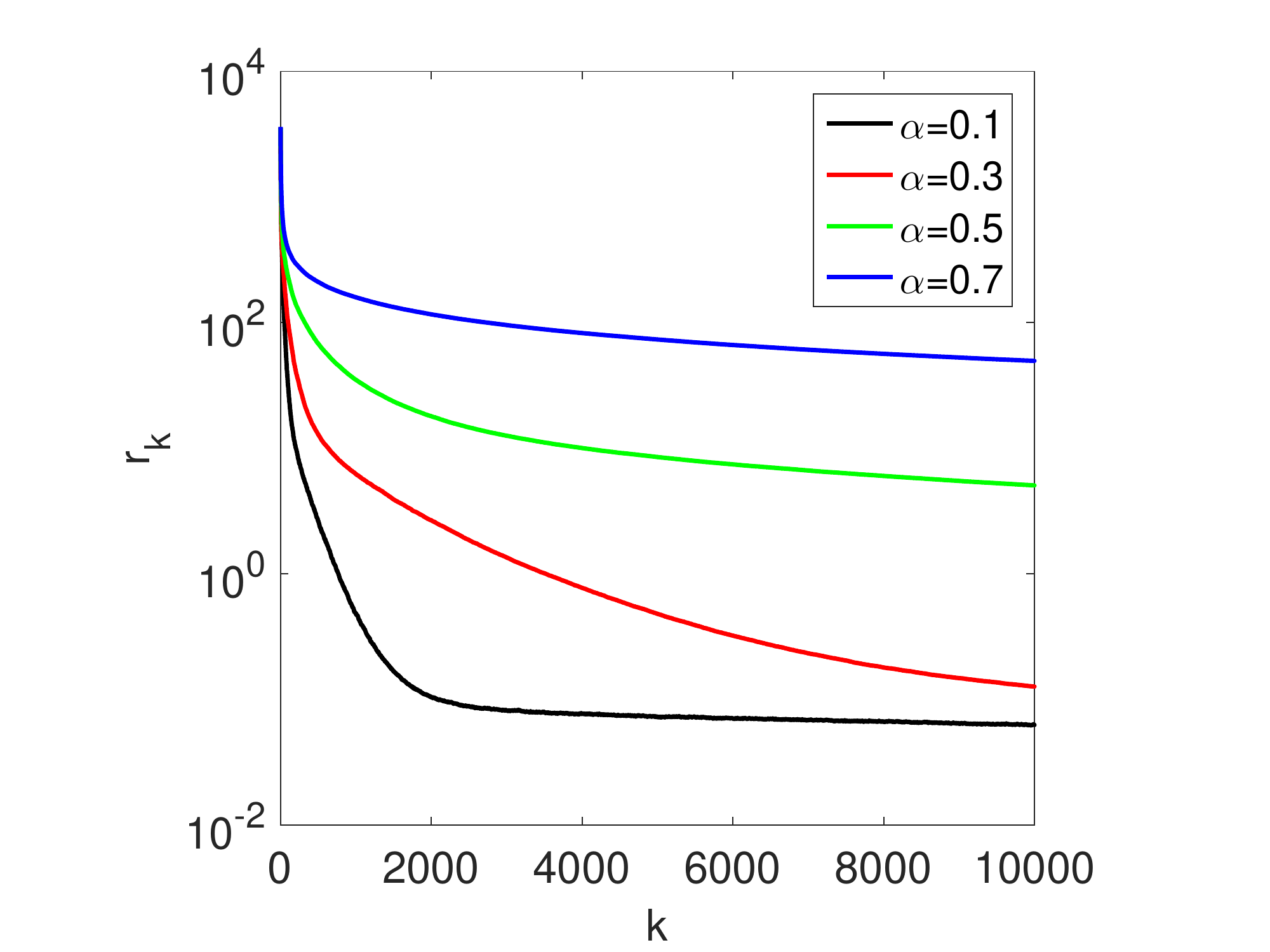} & \includegraphics[trim={2cm 0 2cm 0},clip,width=.30\textwidth]{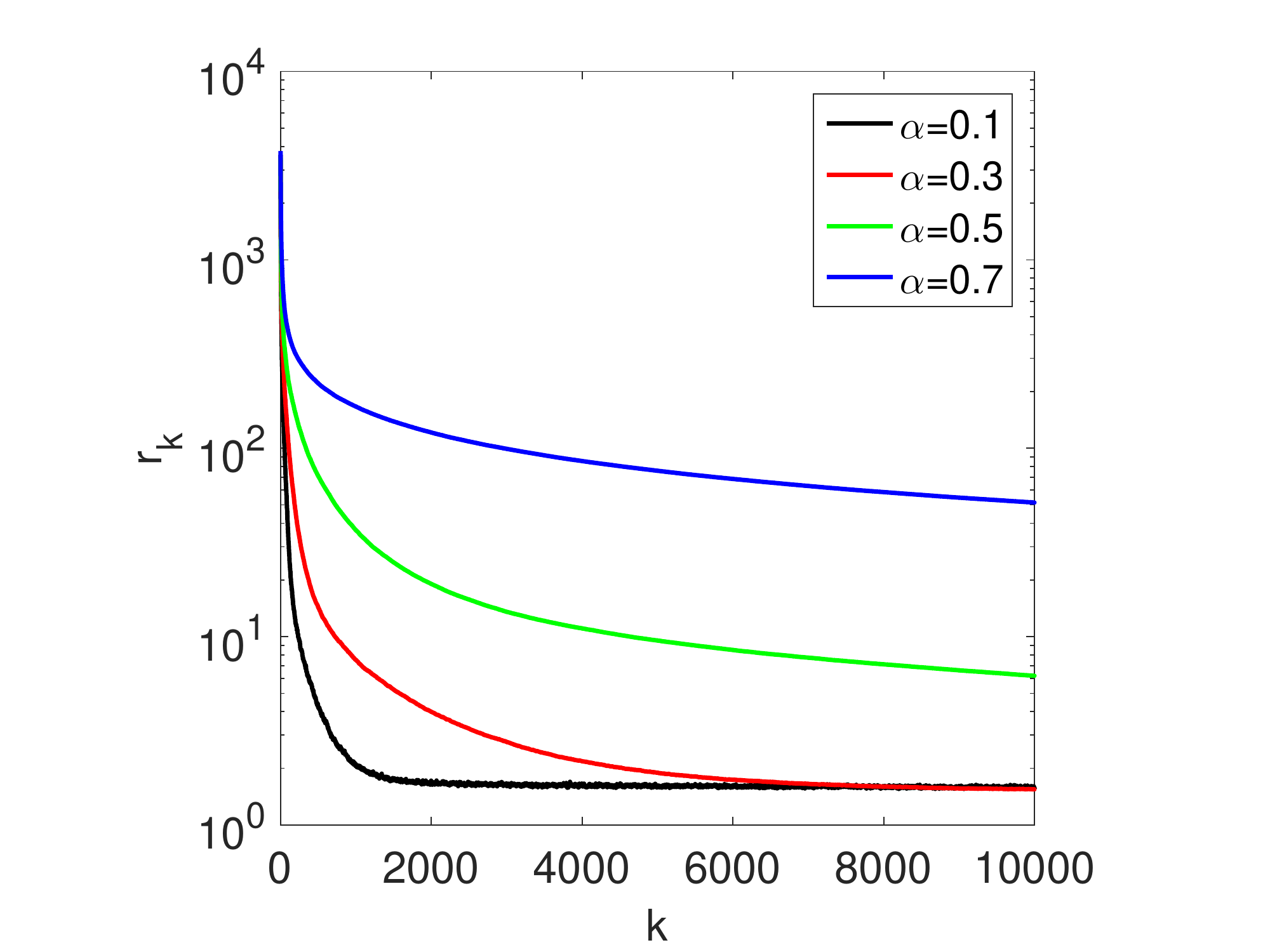} &\includegraphics[trim={2cm 0 2cm 0},clip,width=.30\textwidth]{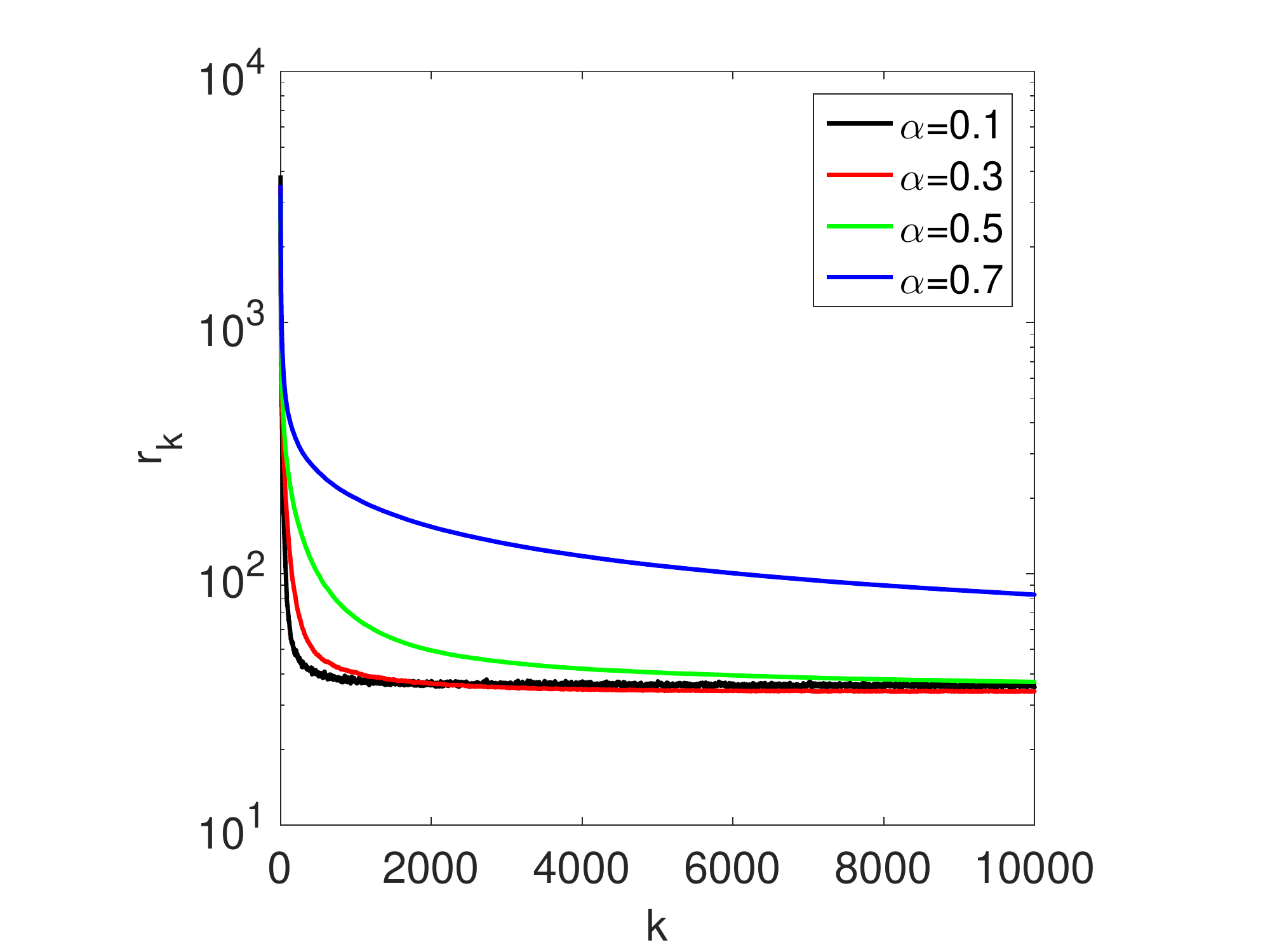}\\
  (a) $\delta=\text{1e-3}$ & (b) $\delta=\text{1e-2}$ & (c) $\delta = \text{5e-2}$
  \end{tabular}
  \caption{Numerical results for \texttt{shaw} with different noise levels by SGD (with various $\alpha$).\label{fig:shaw-sgd-al}}
\end{figure}

\subsection{Comparison with Landweber method}

Since SGD is a randomized version of the classical Landweber method, in
Fig. \ref{fig:land}, we compare their performance. To compare the iteration complexity only,
we count one Landweber iteration as $n$ SGD iterations, and the full gradient evaluation is
indicated by flat segments in the plots. For all examples, the error $e_k$ and residual $r_k$ first
experience fast reduction, and then the error starts to increase,
 which is especially pronounced at $\delta = 5\times 10^{-2}$, exhibiting the typical semiconvergence
behavior. During the initial stage, SGD is much more effective than SGD: indeed
one single loop over all the data can already significantly reduce the error $e_k$ and produce
an acceptable approximation. The precise
mechanism for this interesting observation will be further examined below. However, the nonvanishing
variance of the stochastic gradient slows down the asymptotic convergence of SGD, and
the error $e_k$ and the residual $r_k$ eventually tend to oscillate for noisy data, before finally diverge.

\begin{figure}[hbt!]
  \centering
  \setlength{\tabcolsep}{0pt}
  \begin{tabular}{cc}
   \includegraphics[trim={2cm 0 2cm 0},clip,width=.25\textwidth]{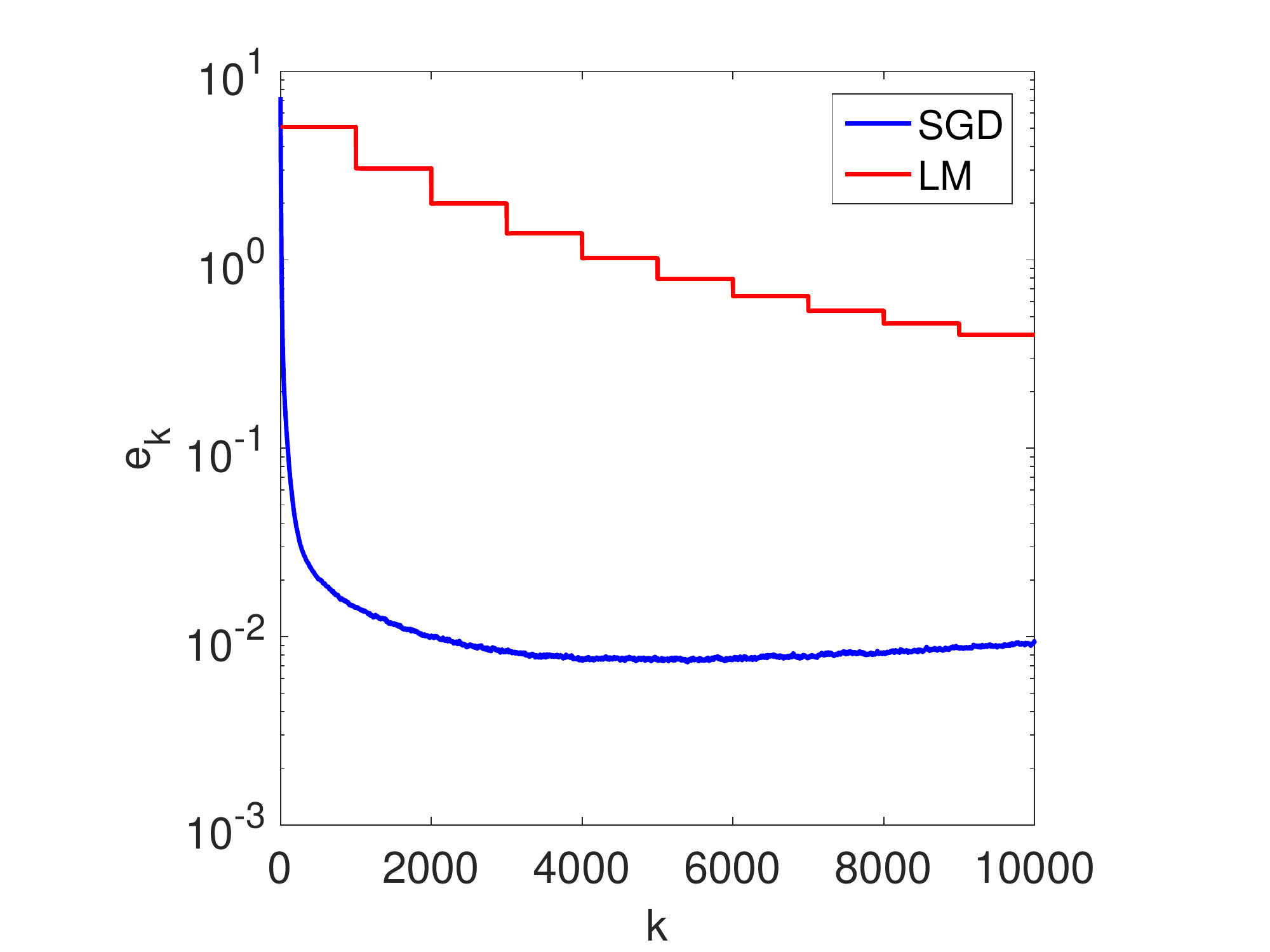} \includegraphics[trim={2cm 0 2cm 0},clip,width=.25\textwidth]{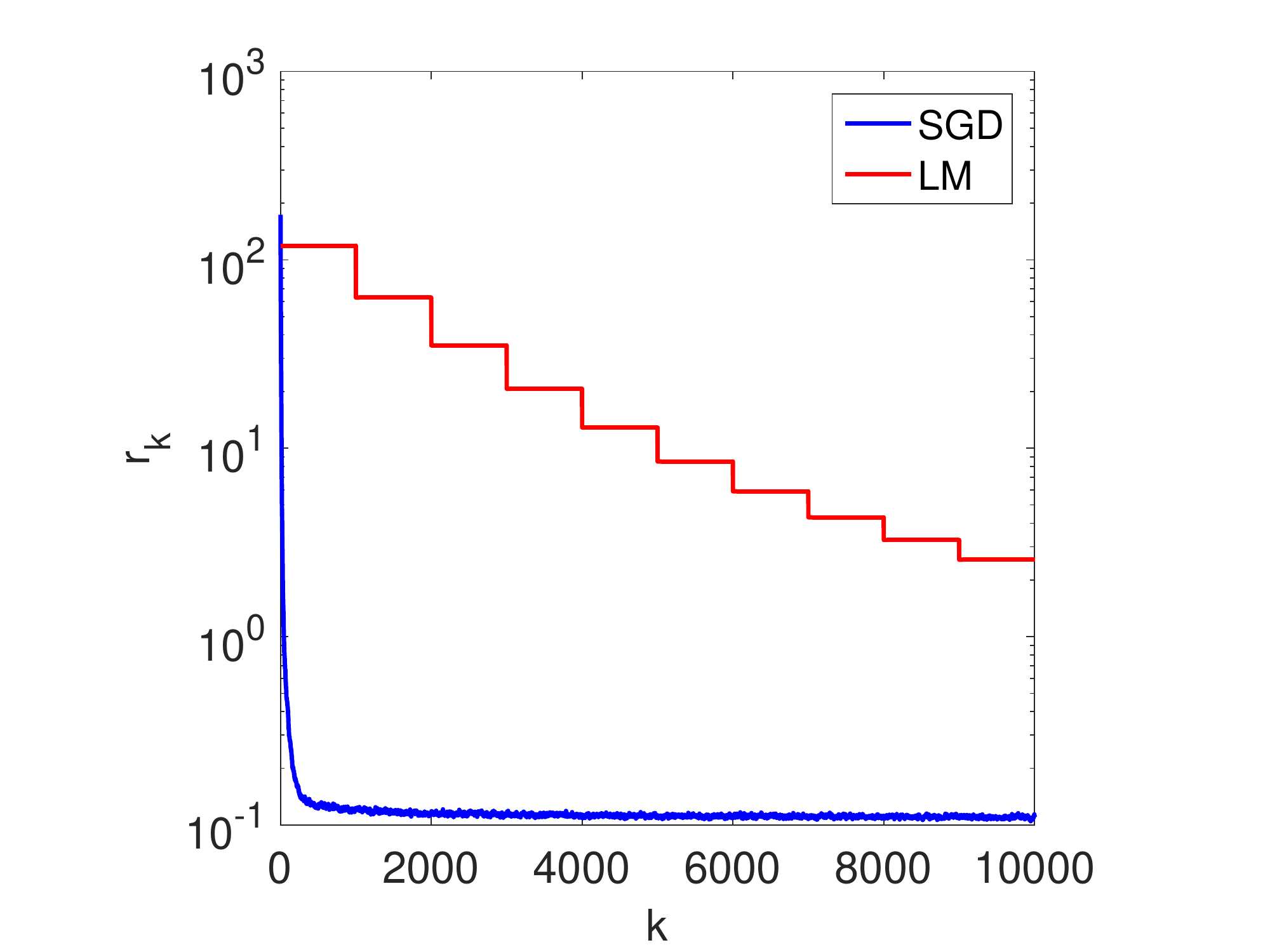}& \includegraphics[trim={2cm 0 2cm 0},clip,width=.25\textwidth]{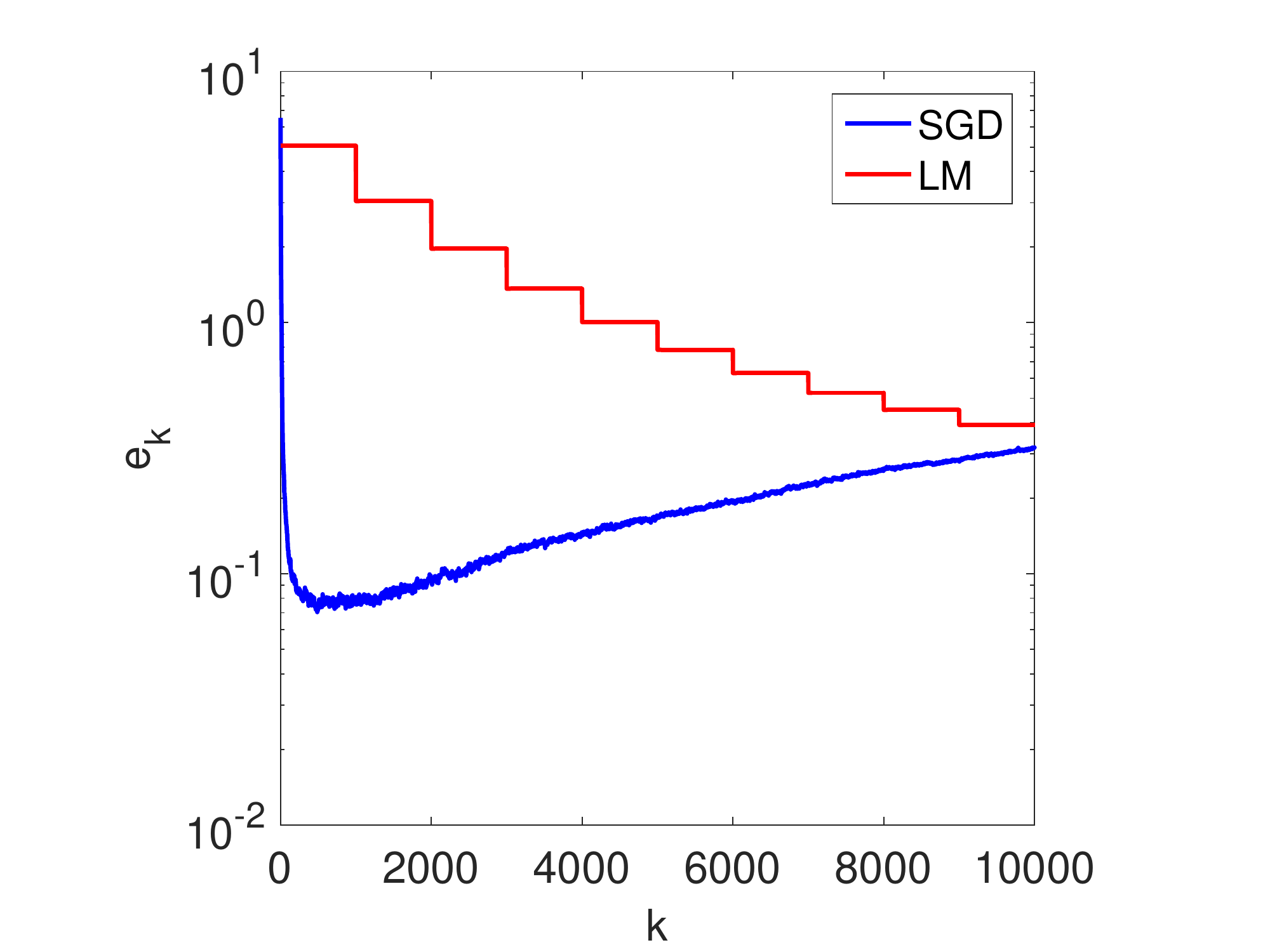} \includegraphics[trim={2cm 0 2cm 0},clip,width=.25\textwidth]{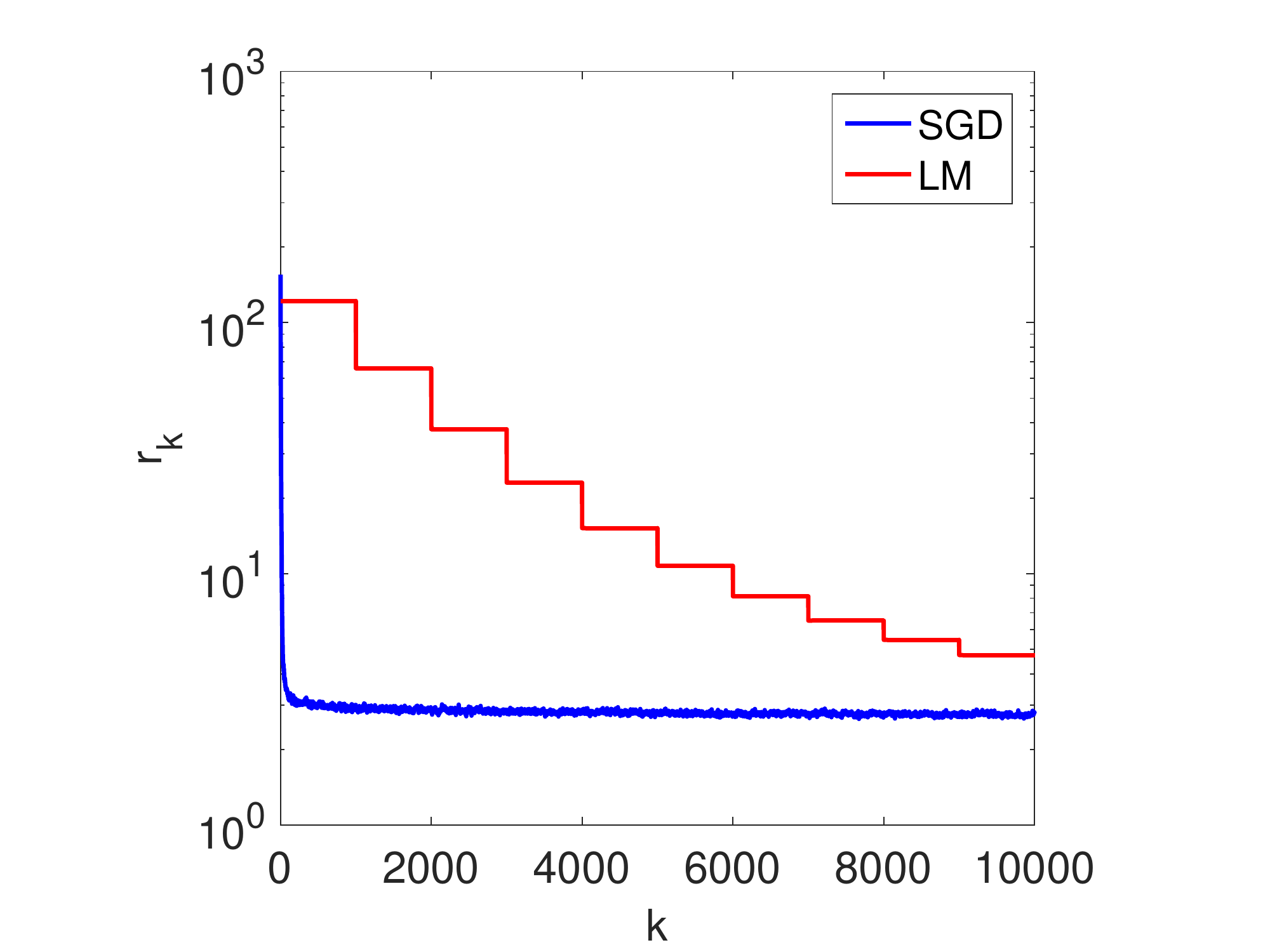}\\
   (a) \texttt{phillips}, $\delta=10^{-2}$ & (b) \texttt{phillips}, $\delta=5\times10^{-2}$ \\
   \includegraphics[trim={2cm 0 2cm 0},clip,width=.25\textwidth]{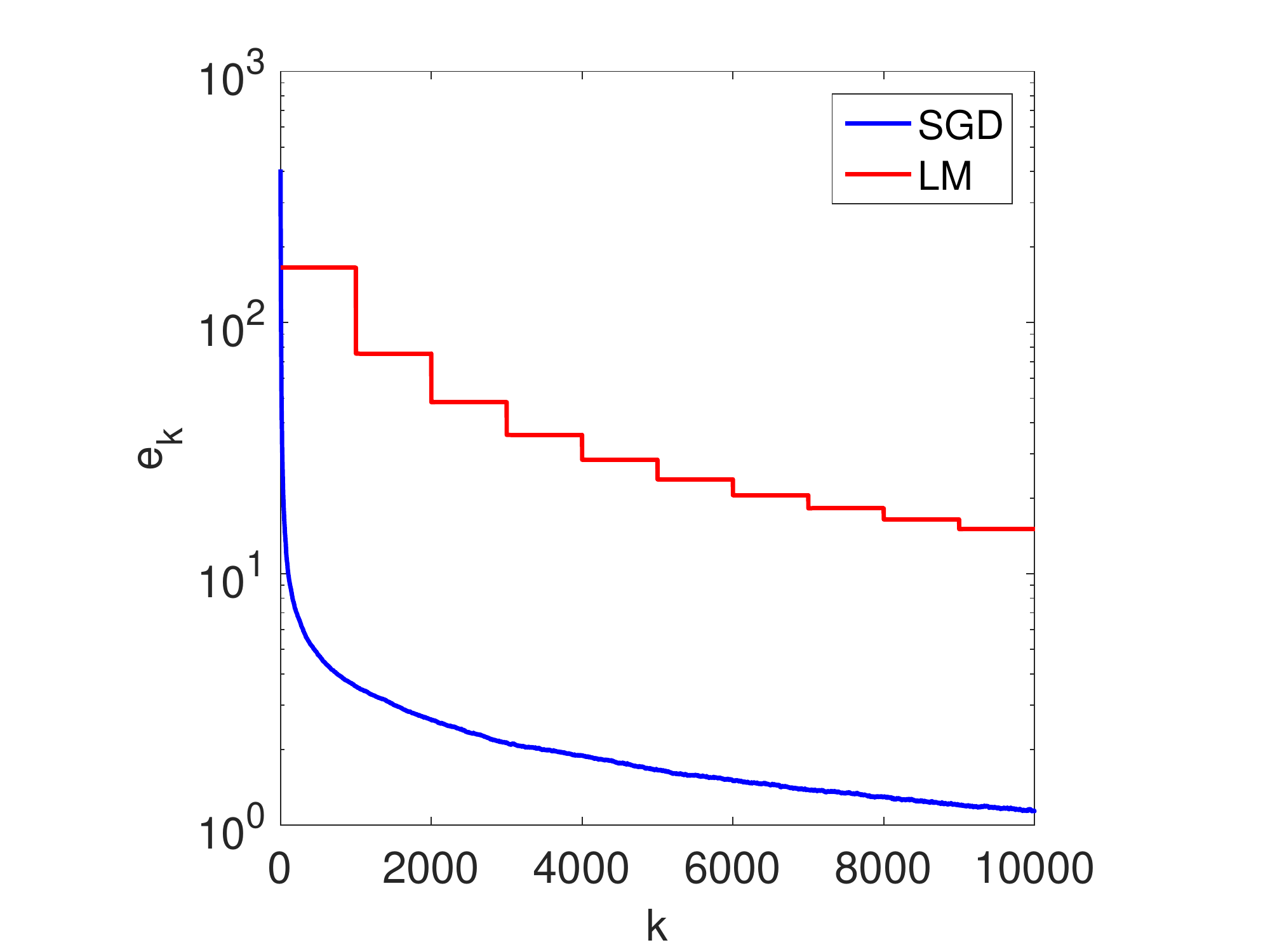} \includegraphics[trim={2cm 0 2cm 0},clip,width=.25\textwidth]{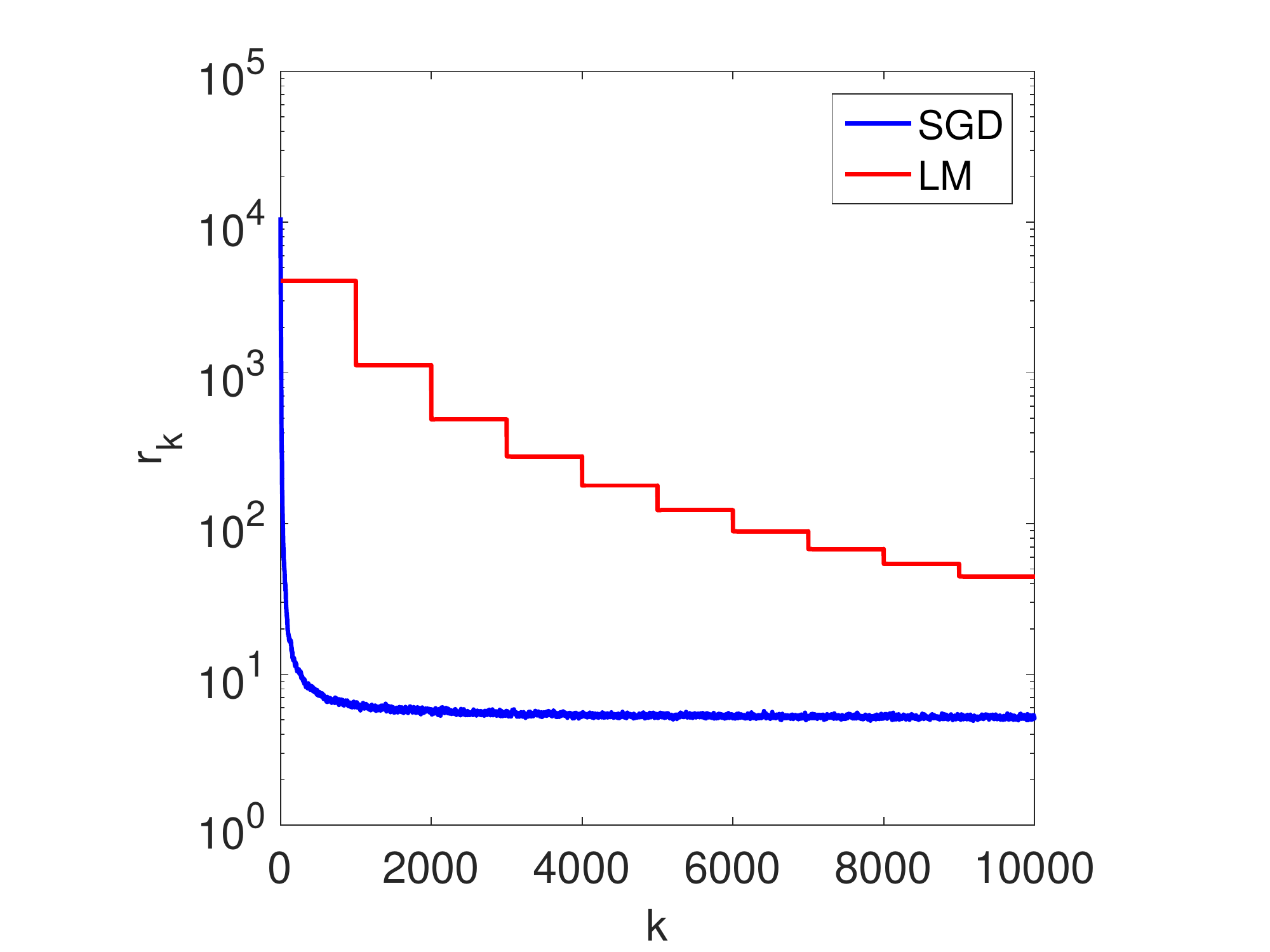}& \includegraphics[trim={2cm 0 2cm 0},clip,width=.25\textwidth]{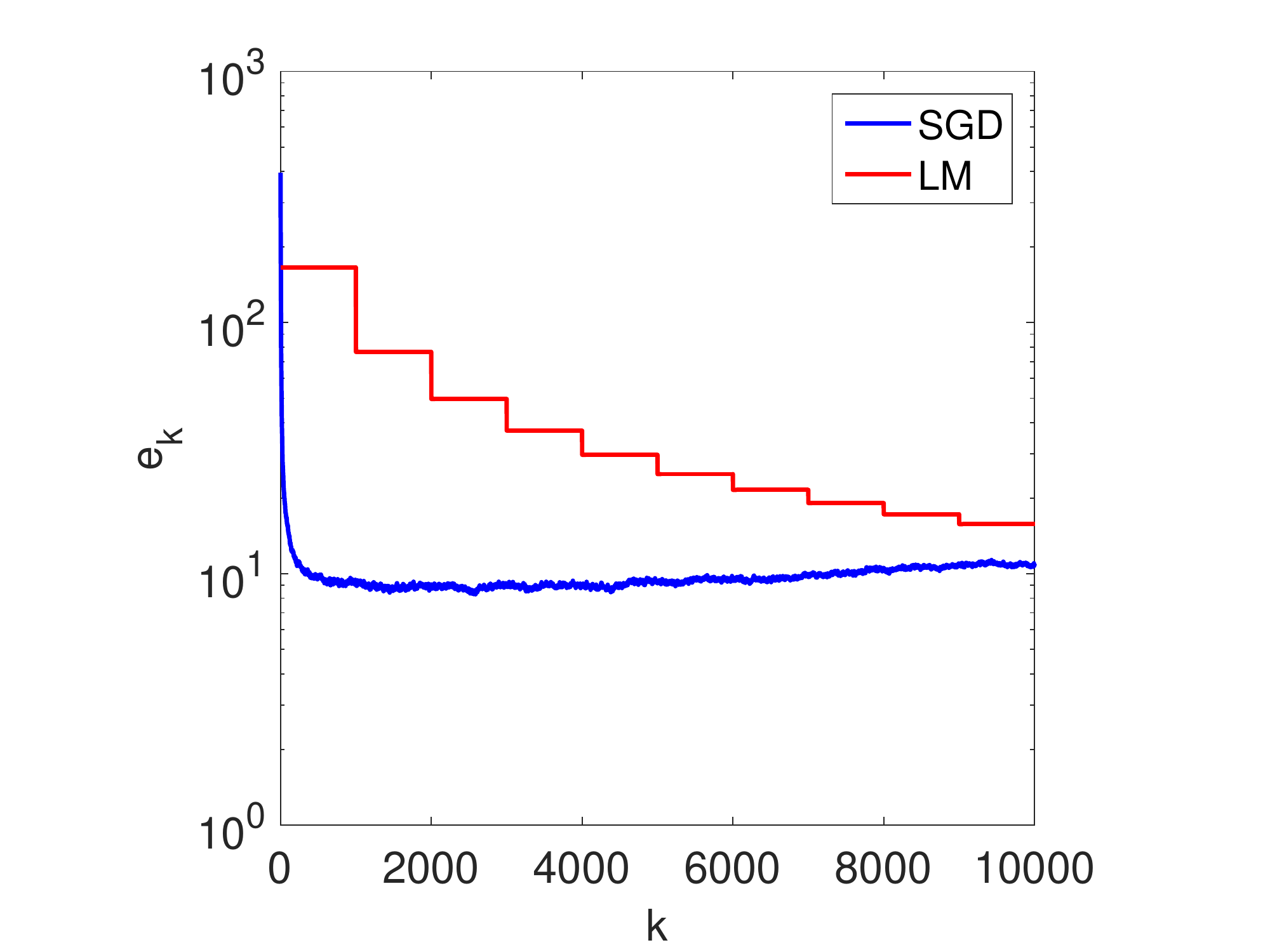} \includegraphics[trim={2cm 0 2cm 0},clip,width=.25\textwidth]{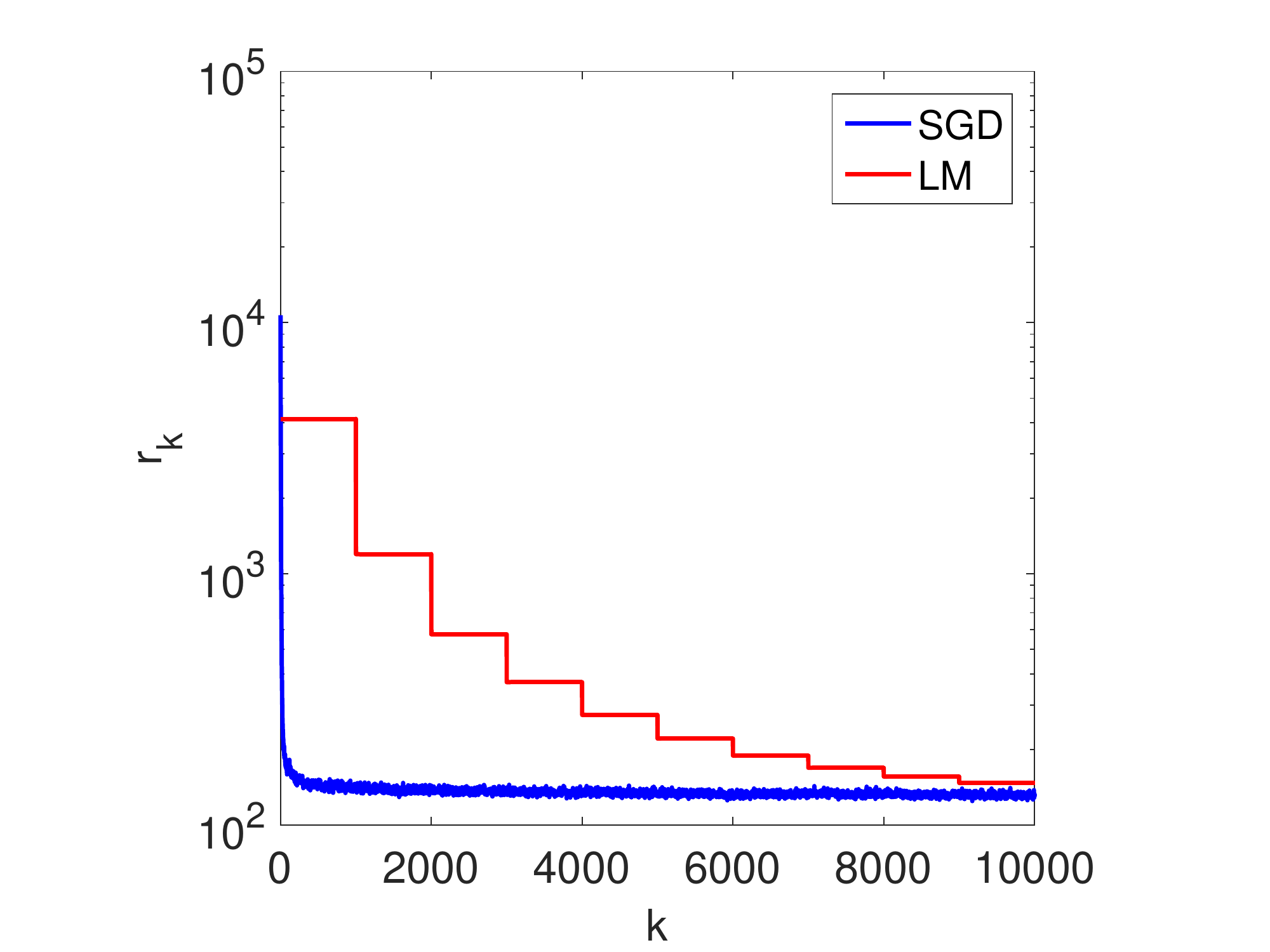}\\
   (c) \texttt{gravity}, $\delta=10^{-2}$ & (d) \texttt{gravity}, $\delta=5\times10^{-2}$  \\
   \includegraphics[trim={2cm 0 2cm 0},clip,width=.25\textwidth]{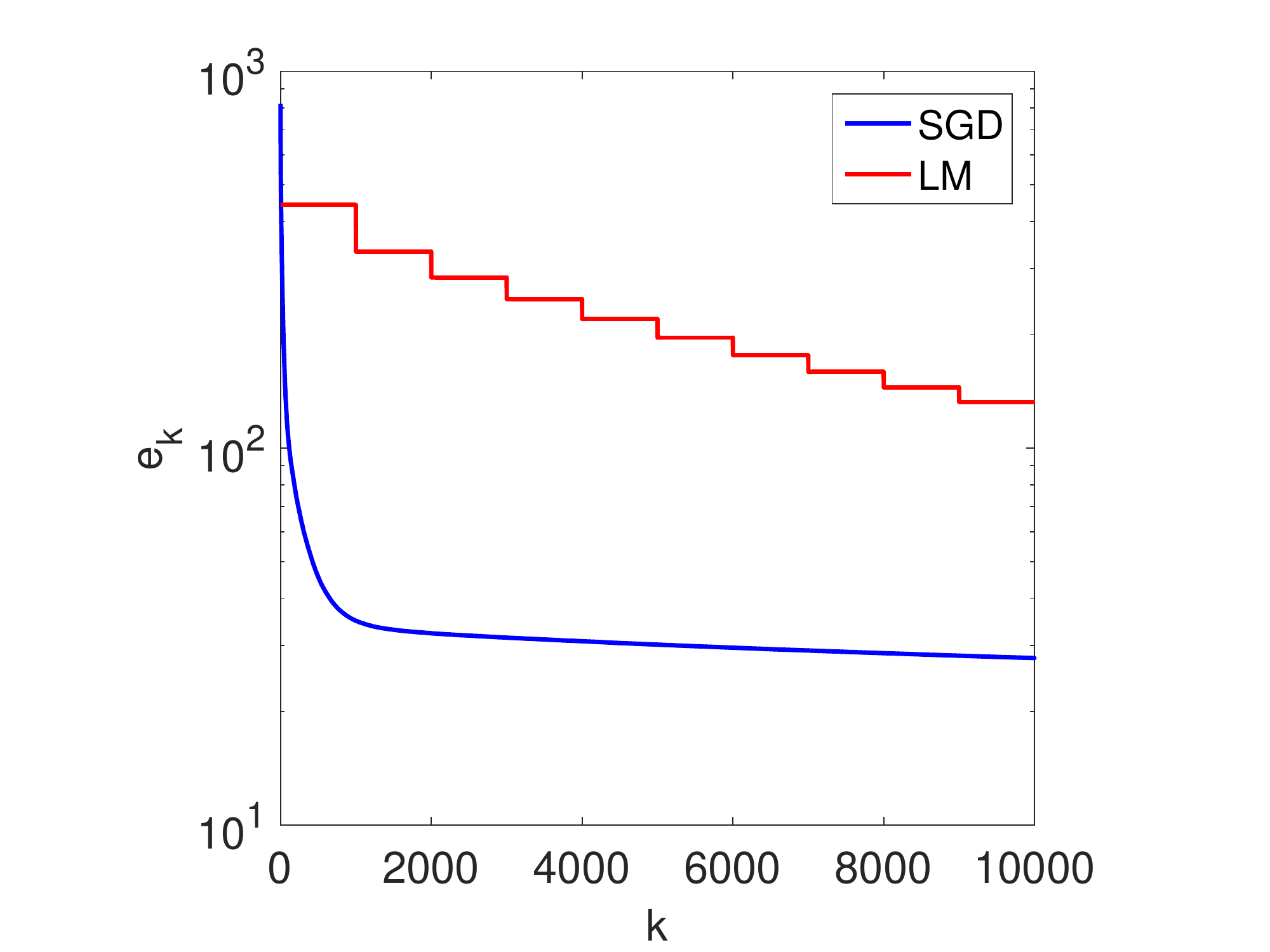} \includegraphics[trim={2cm 0 2cm 0},clip,width=.25\textwidth]{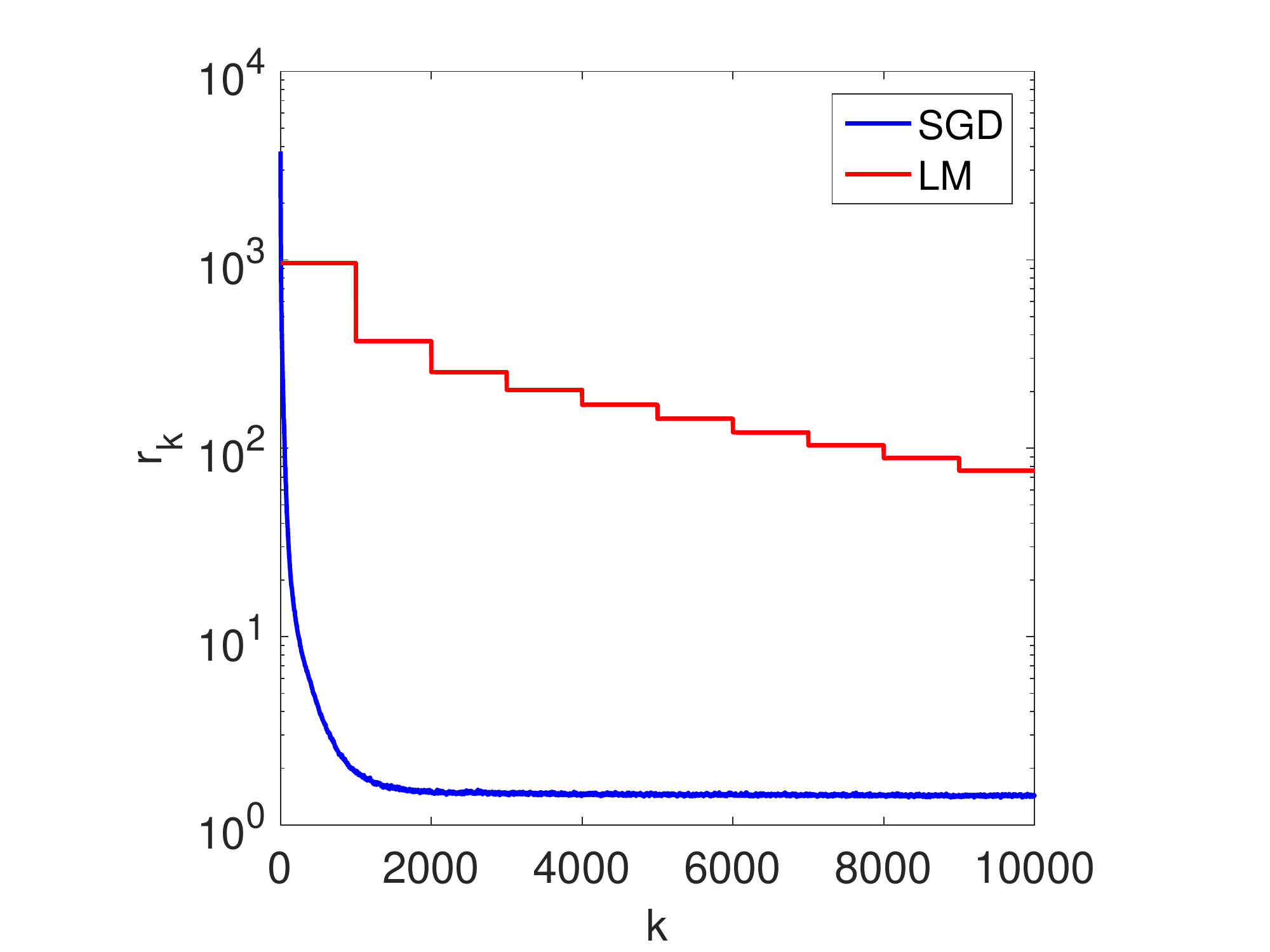} & \includegraphics[trim={2cm 0 2cm 0},clip,width=.25\textwidth]{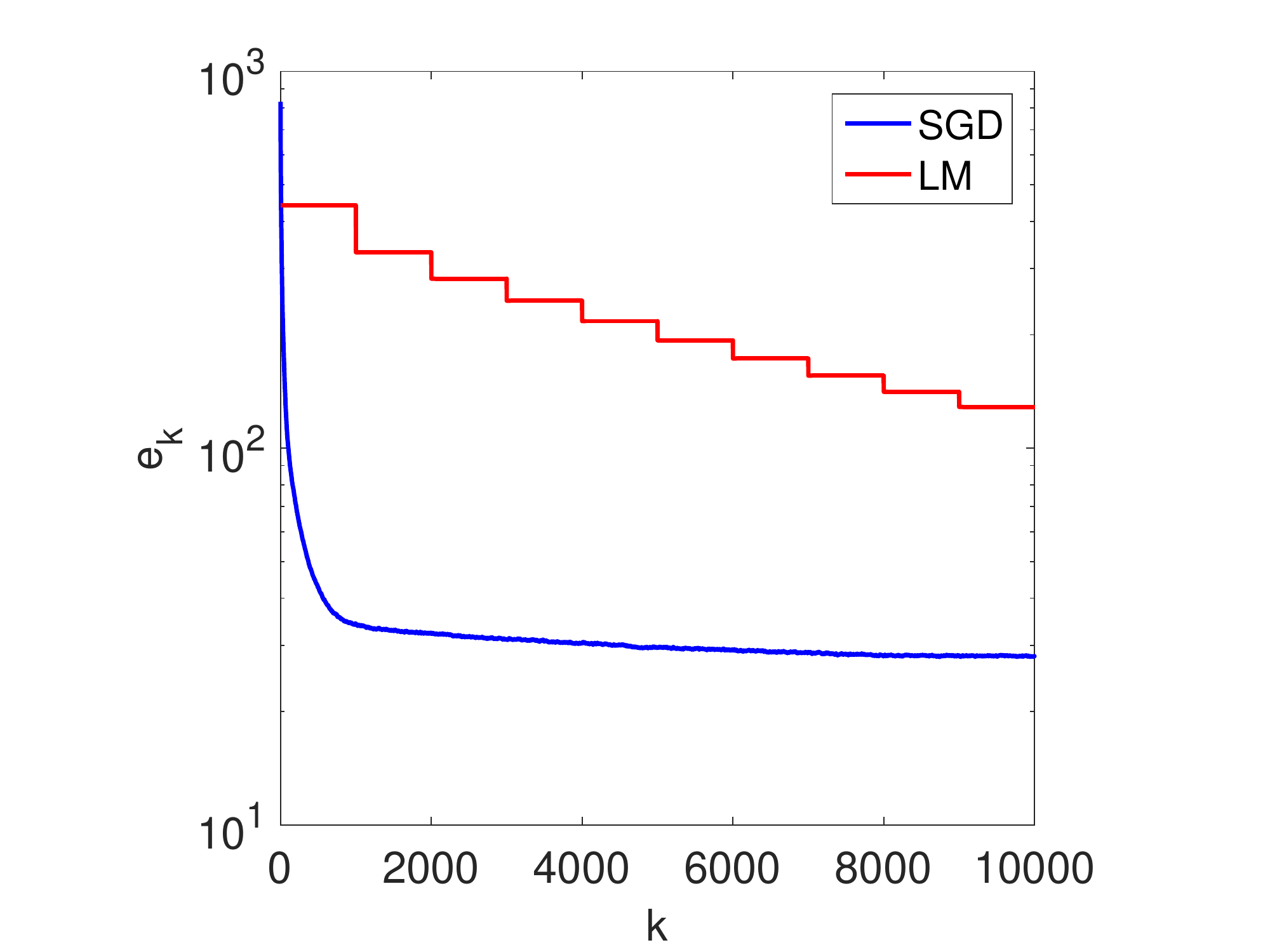} \includegraphics[trim={2cm 0 2cm 0},clip,width=.25\textwidth]{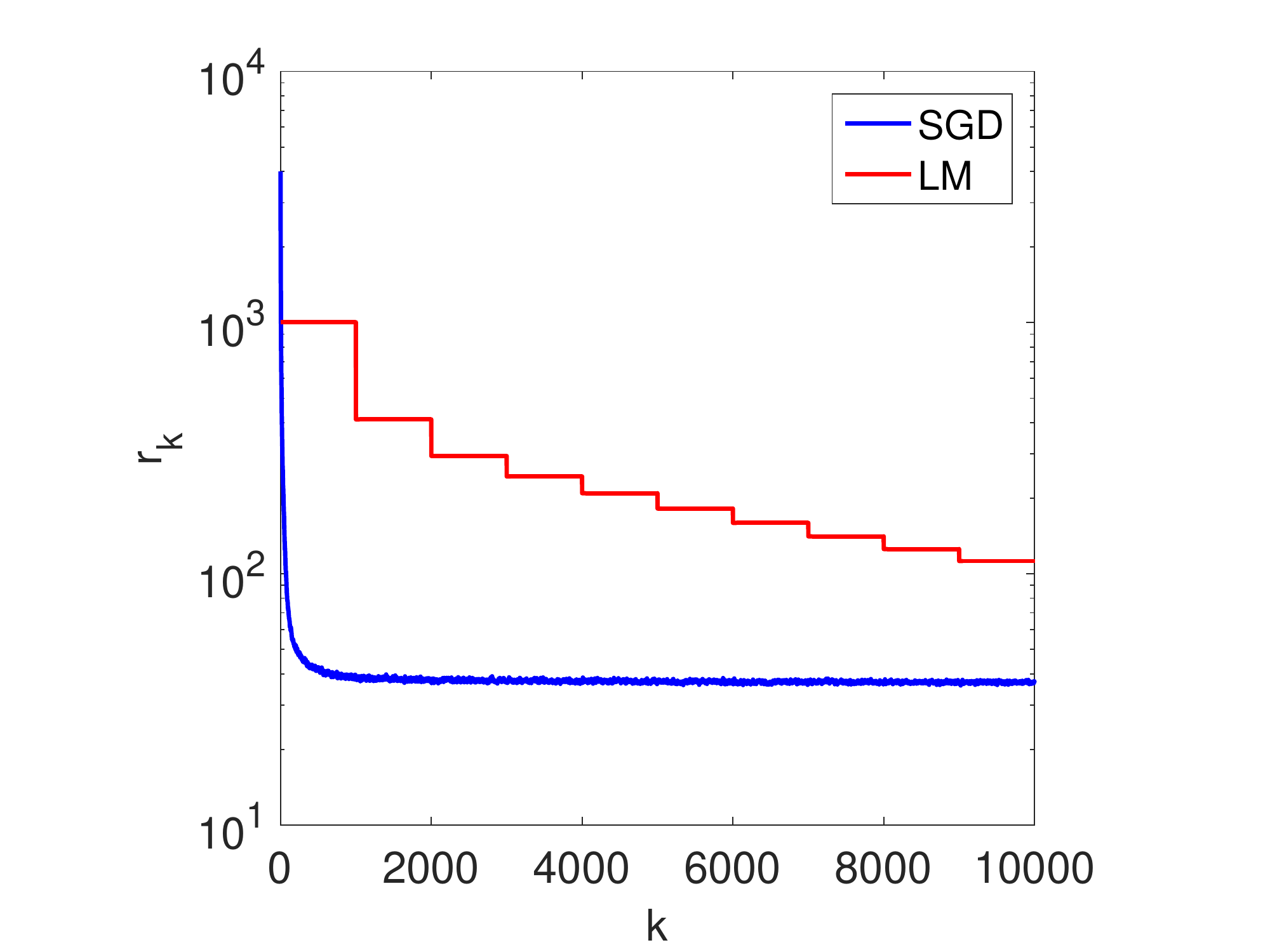}\\
   (e) \texttt{shaw}, $\delta=10^{-2}$ & (f) \texttt{shaw}, $\delta=5\times10^{-2}$
  \end{tabular}
  \caption{Numerical results for the examples by SGD (with $\alpha=0.1$) and LM.\label{fig:land}}
\end{figure}

\subsection{Preasymptotic convergence}

Now we examine the preasymptotic strong convergence of SGD (note that the weak error satisfies a
Landweber type iteration). Theorem \ref{thm:asympt-noisy} (and Lemma \ref{thm:asympt-exact}) predicts that
during first iterations, the low-frequency error $e_L:=\E[\|P_\mathcal{L}e_k\|^2]$ decreases rapidly, but the
high-frequency error $e_H:=\E[\|P_\mathcal{H}e_k\|^2]$ can at best decay mildly. For all examples, the first five
singular vectors can well capture the total energy of the initial error $e_1=x^*-x_1$, which suggests
a truncation level $L=5$ for the numerical illustration. We plot the low- and high-frequency errors $e_L$ and
$e_H$ and the total error $e=\E[\|e_k\|^2]$ in Fig. \ref{fig:decom}. The low-frequency error $e_L$ decays
much more rapidly during the initial iterations, and since under the source condition \eqref{eqn:source},
$e_L$ is indeed dominant, the total error $e$ also enjoys a fast initial decay. Intuitively, this
behavior may be explained as follows. The rows of the matrix ${A}$ mainly contain low-frequency modes,
and thus each SGD iteration tends to mostly remove the low-frequency component $e_L$ of the initial error
$x^*-x_1$. The high-frequency component $e_H$ experiences a similar but much slower decay. Eventually, both components
level off and oscillate, due to the deleterious effect of noise. These observations confirm the preasymptotic analysis in Section \ref{app:preasympt}. For
noisy data, the error $e_k$ can be highly oscillating, so is the residual $r_k$. The larger the noise
level $\delta$ is, the larger the oscillation magnitude becomes.

\begin{figure}[hbt!]
  \centering
  \setlength{\tabcolsep}{0pt}
  \begin{tabular}{ccc}
    \includegraphics[trim={2cm 0 2cm 0},clip,width=.33\textwidth]{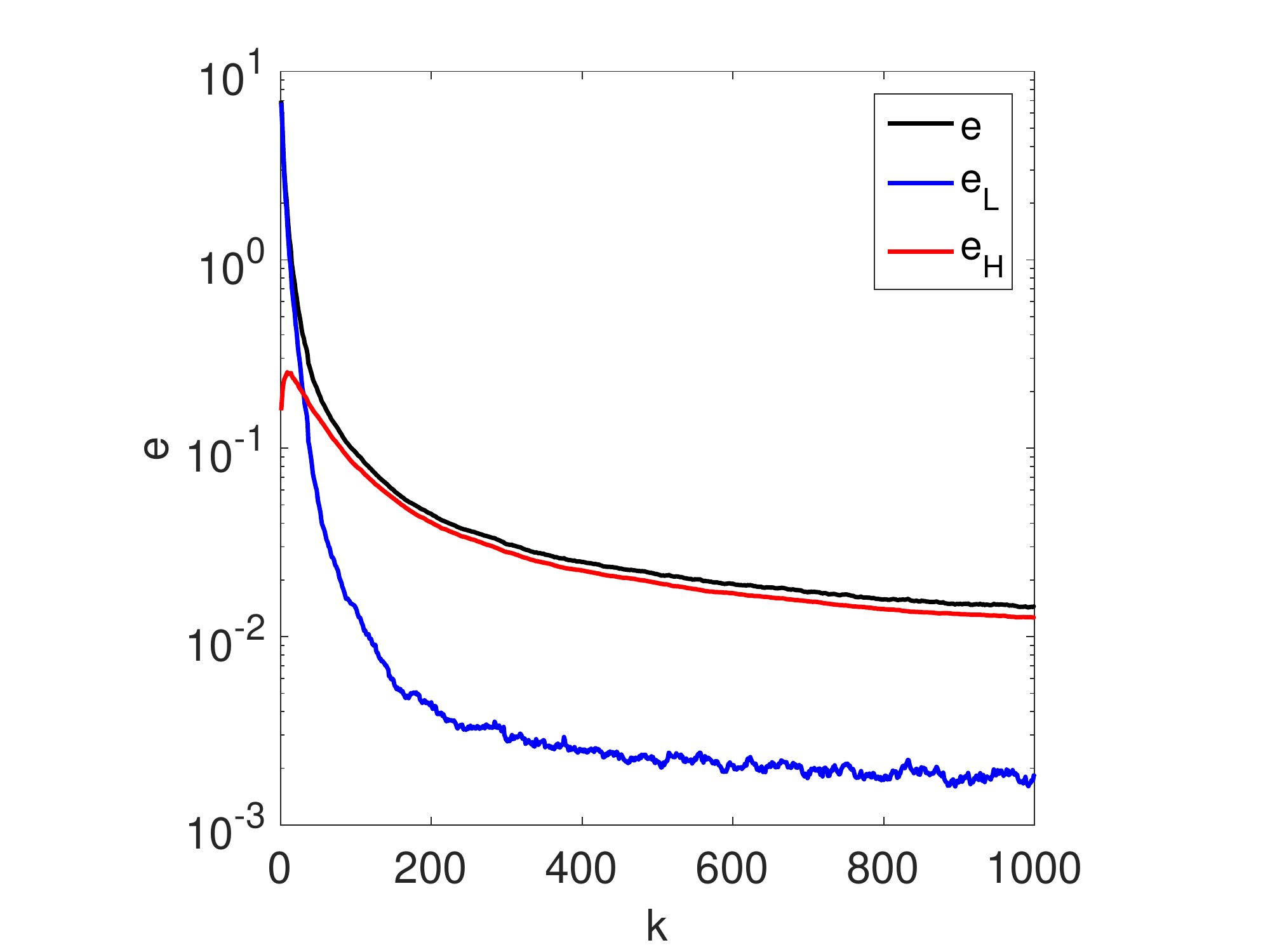} & \includegraphics[trim={2cm 0 2cm 0},clip,width=.33\textwidth]{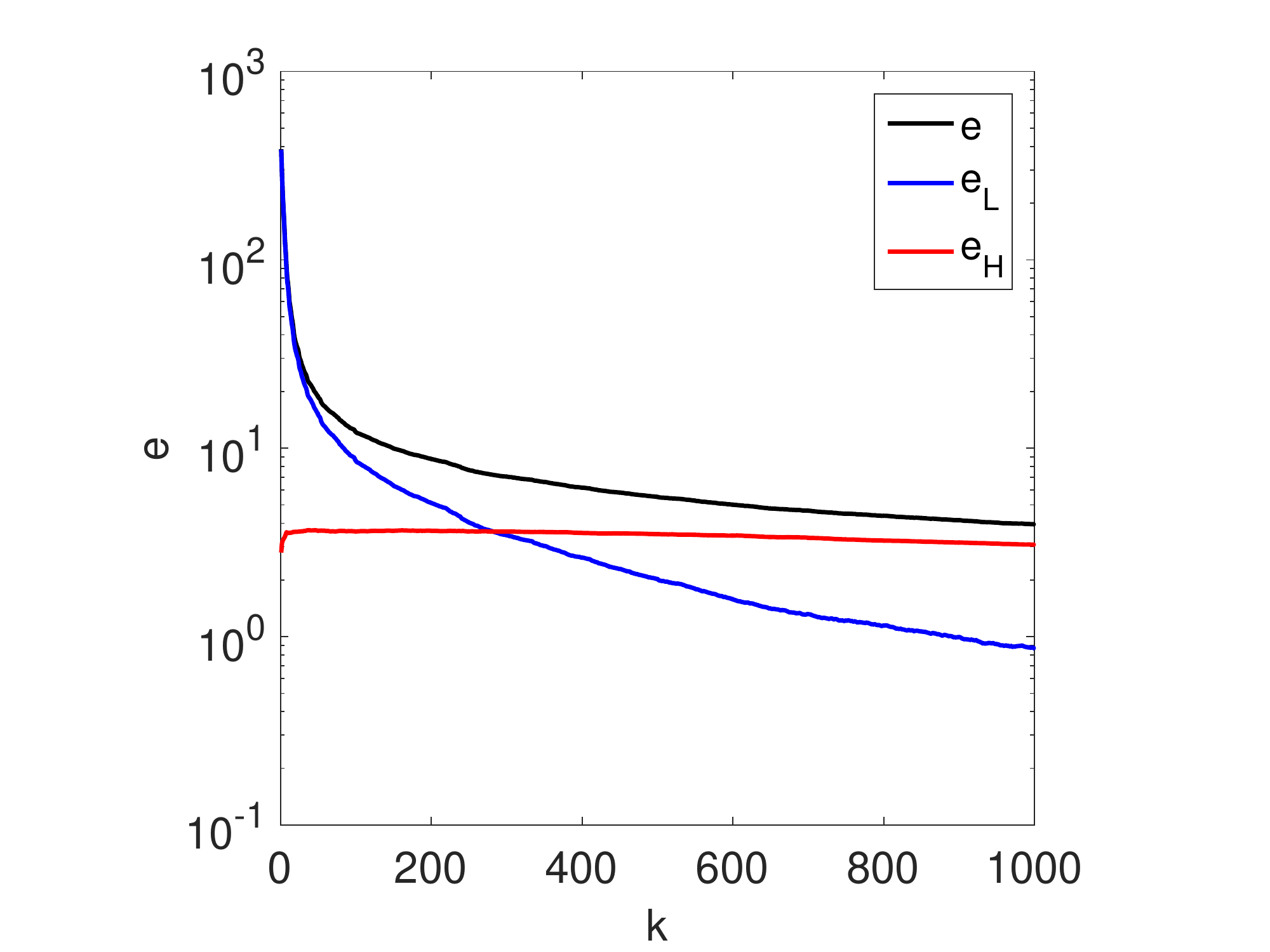}
   &\includegraphics[trim={2cm 0 2cm 0},clip,width=.33\textwidth]{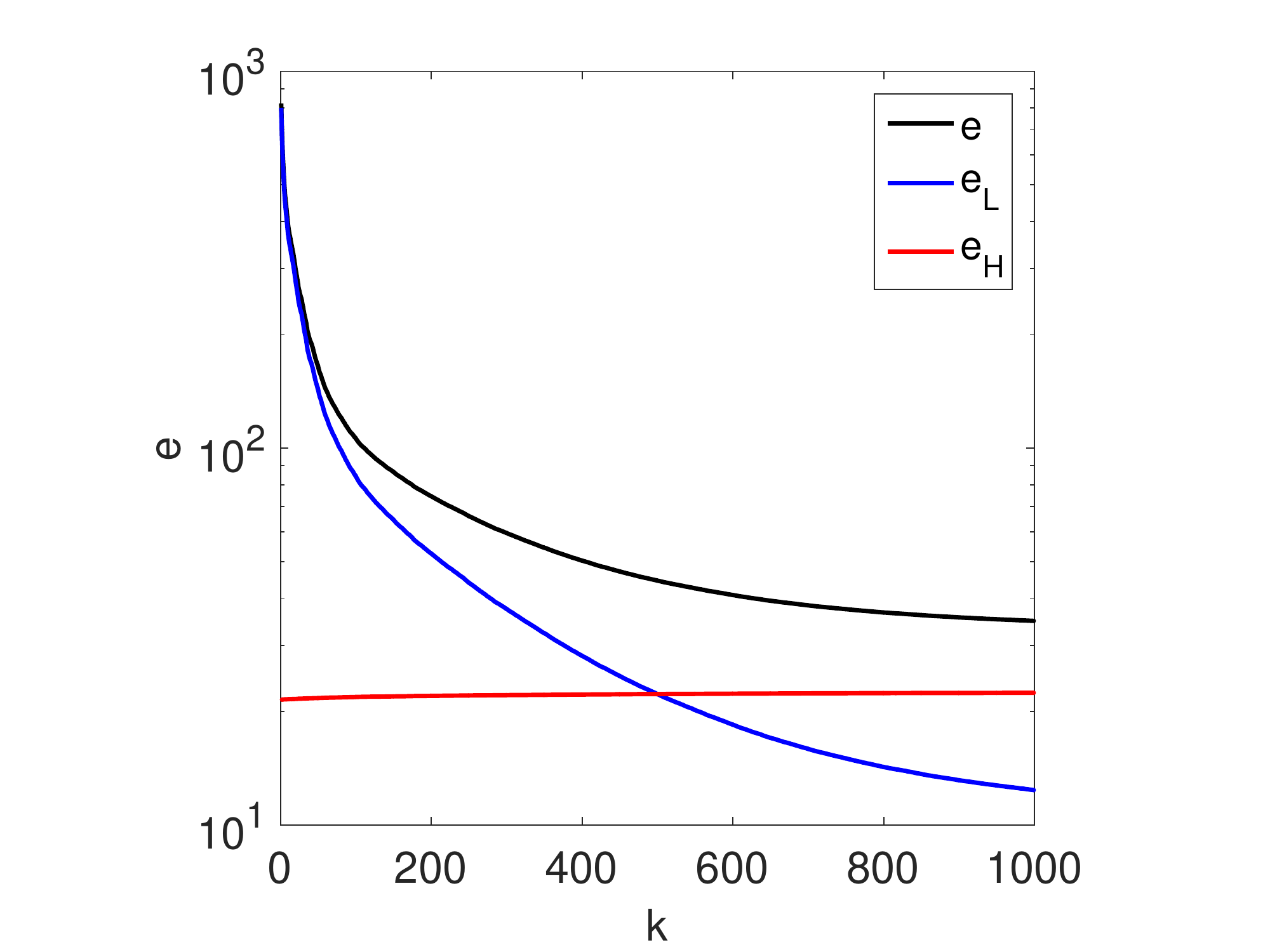}\\
   \includegraphics[trim={2cm 0 2cm 0},clip,width=.33\textwidth]{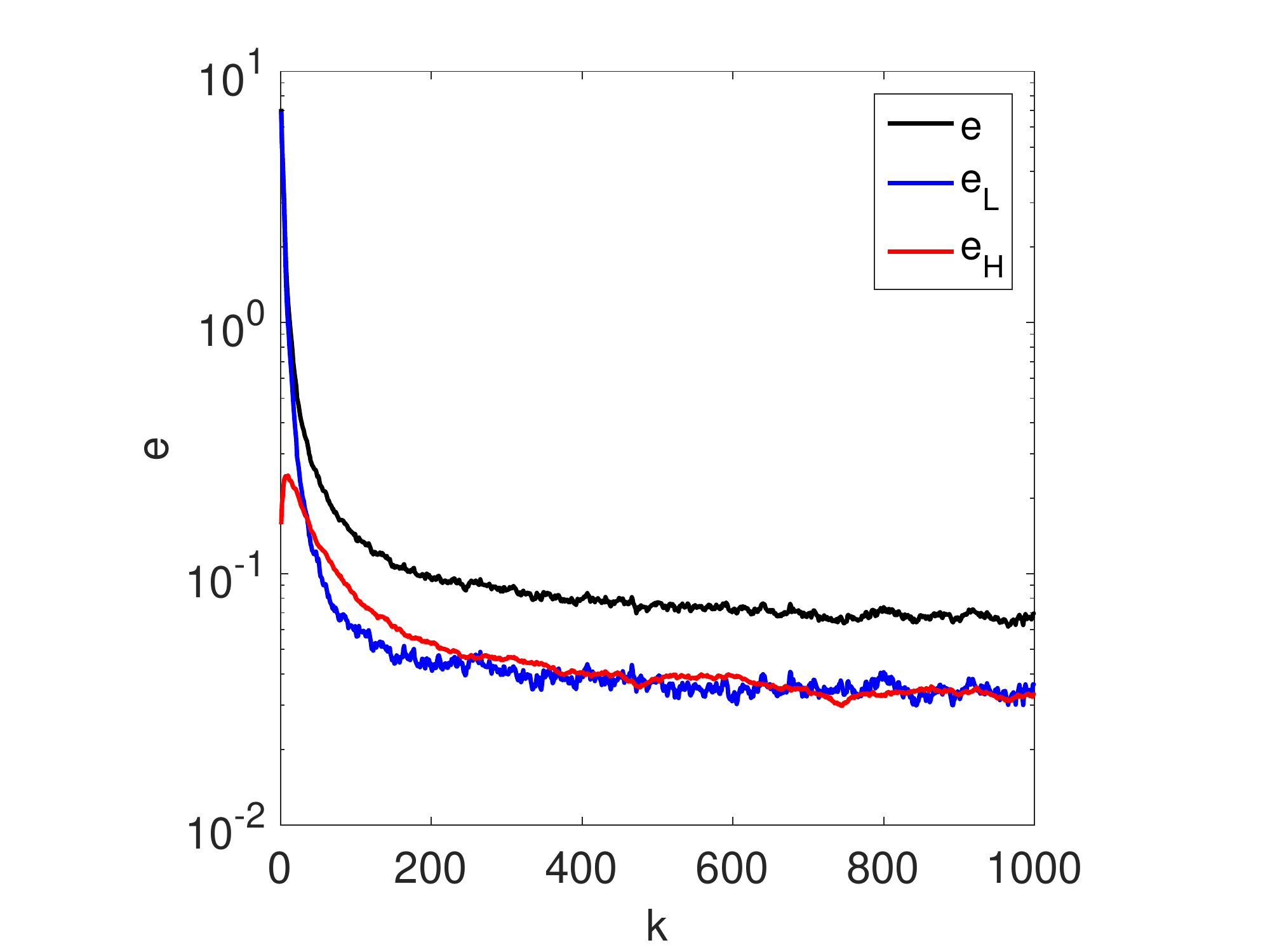} & \includegraphics[trim={2cm 0 2cm 0},clip,width=.33\textwidth]{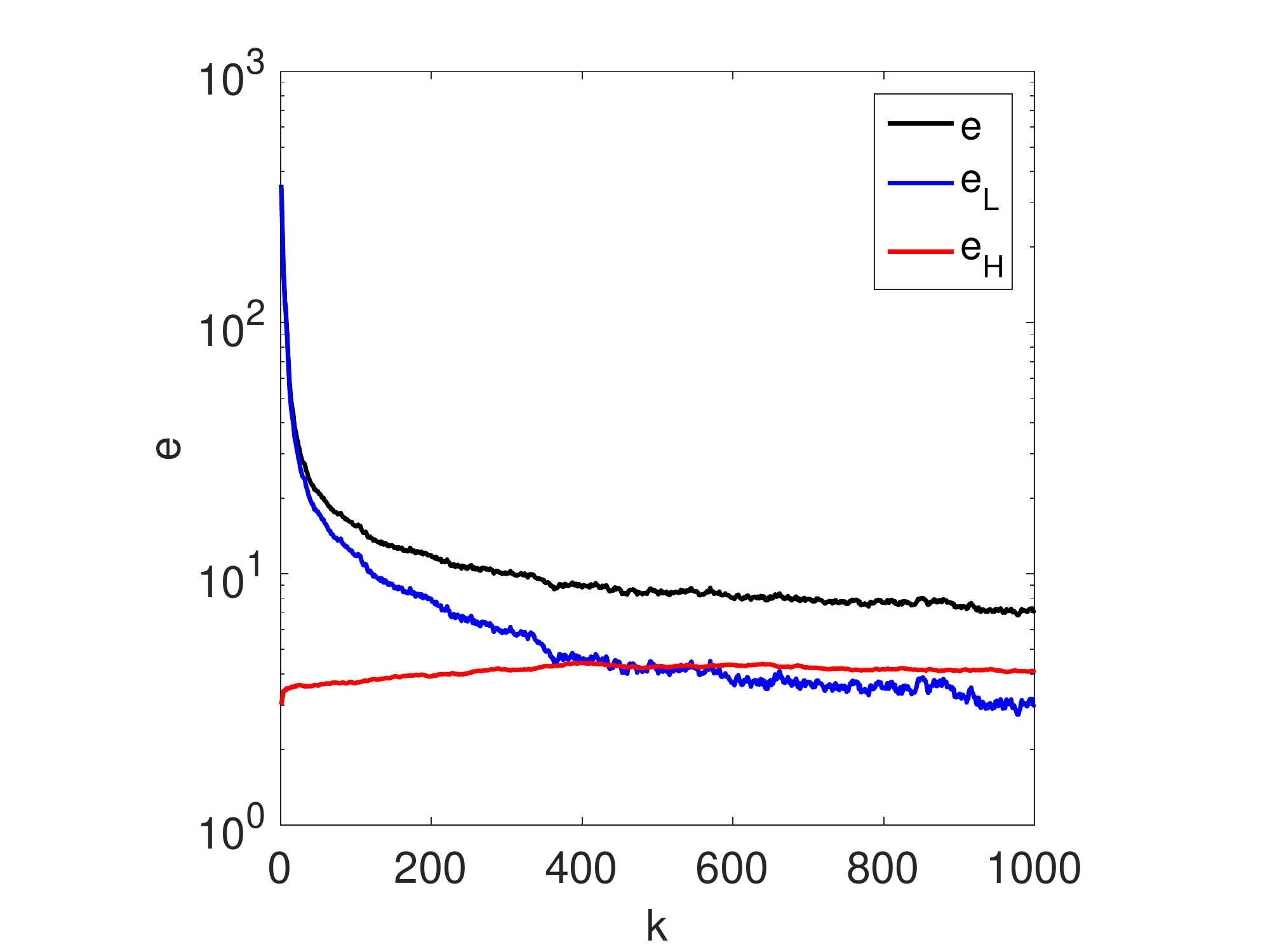}
   &\includegraphics[trim={2cm 0 2cm 0},clip,width=.33\textwidth]{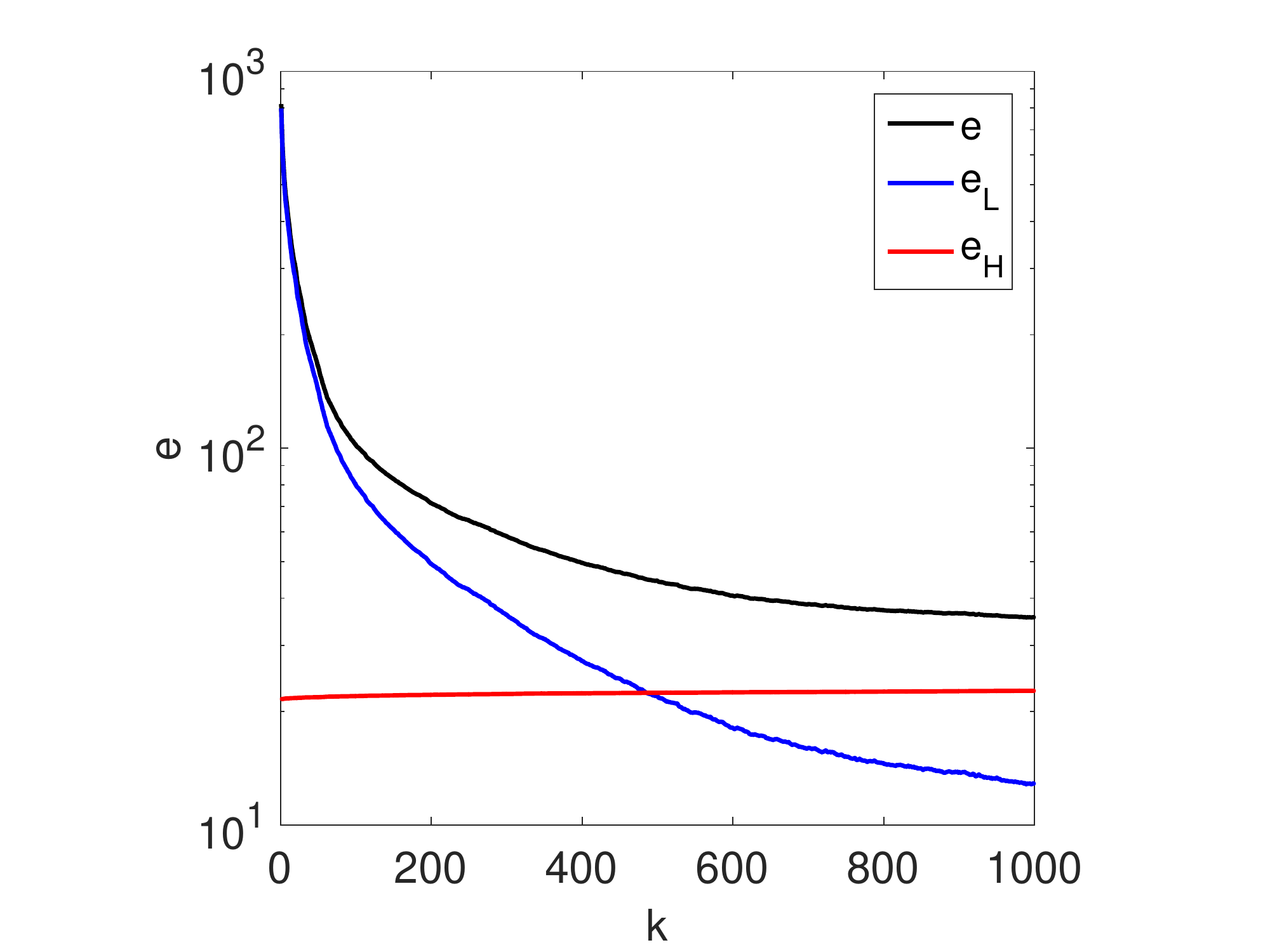}\\
      (a) \texttt{phillips} & (b) \texttt{gravity} & (c) \texttt{shaw}
  \end{tabular}
  \caption{The error decay for the examples with two noise levels: $\delta=10^{-2}$ (top) and $\delta = 5\times 10^{-2}$ (bottom), with a truncation level $L=5$.
  \label{fig:decom}}
\end{figure}

\subsection{Asymptotic convergence}

To examine the asymptotic convergence (with respect to the noise level $\delta$), in Table \ref{tab:err}, we present the smallest
error $e$ along the trajectory and the number of iterations to reach the error $e$ for several different
noise levels. It is observed that for all three examples, the minimal error $e$ increases steadily with
the noise level $\delta$, whereas also the required number of iterations decreases dramatically, which
qualitatively agrees well with Remark \ref{rmk:rate}. Thus, SGD is especially efficient in the regime
of high noise level, for which one or two epochs can already give very good approximations, due to the fast
preasymptotic convergence. This agrees with the common belief that SGD is most effective for finding an
approximate solution that is not highly accurate. At low noise levels, \texttt{shaw} takes far more iterations
to reach the smallest error. This might be attributed to the fact that the exponent $p$ in the source condition
\eqref{eqn:source} for \texttt{shaw} is much smaller than that for \texttt{phillips} or \texttt{gravity}, since
the low-frequency modes are less dominating, as roughly indicated by the red curves in Fig. \ref{fig:decom}. Interestingly, for
all examples, the error $e$ undergoes sudden change when the noise level $\delta$ increases from 1e-2 to 3e-2.
This might be related to the exponent $\alpha$ in the step size schedule, which probably should
be adapted to the noise level $\delta$ in order to achieve optimal balance between the computational efficiency
and statistical errors.

\begin{table}[hbt!]
\centering
\caption{The (minimal) expected error $e$ for the examples. \label{tab:err}}
\begin{tabular}{c|cccccc}
\hline
 $\delta$ & \texttt{phillips} & \texttt{gravity} & \texttt{shaw}\\
\hline
 1e-3 & (1.09e-3,7.92e4) & (3.22e-1,4.55e5) & (2.92e0,3.55e6)\\
 5e-3 & (3.23e-3,1.83e4) & (5.65e-1,6.19e4) & (3.21e0,1.95e6)\\
 1e-2 & (6.85e-3,3.09e3) & (6.21e-1,4.60e4) & (6.75e0,1.15e6)\\
 3e-2 & (4.74e-2,4.20e2) & (2.60e0, 6.50e3) & (3.50e1,7.80e3) \\
 5e-2 & (6.71e-2,1.09e3) & (6.32e0, 2.55e3) & (3.70e1,1.28e3)\\
\hline
\end{tabular}
\end{table}
\section{Concluding remarks}

In this work, we have analyzed the regularizing property of SGD for solving linear
inverse problems, by extending properly deterministic inversion theory. The study indicates that with proper
early stopping and suitable step size schedule, it is regularizing in the sense that iterates converge to
the exact solution in the mean squared norm as the noise level tends to zero. Further, under the canonical
source condition, we prove error estimates, which depend on the noise level and the schedule of step sizes.
Further we analyzed the preasymptotic convergence behavior of SGD, and proved that the
low-frequency error can decay much faster than high-frequency error. This allows explaining the fast initial
convergence of SGD typically observed in practice. The findings are complemented by extensive numerical experiments.

There are many interesting questions related to stochastic iteration algorithms that deserve further research. One outstanding issue
is stopping criterion, and rigorous yet computationally efficient criteria have to be developed. In practice, the performance of SGD
can be sensitive to the exponent $\alpha$ in the step size schedule \cite{NemirovskiJuditskyLan:2008}. Promising strategies for overcoming
the drawback include averaging \cite{PolyakJuditsky:1992} and variance reduction \cite{JohnsonZhang:2013}. It is
of much interest to analyze such schemes in the context of inverse problems, including nonlinear inverse
problems and penalized variants.

\bibliographystyle{abbrv}
\bibliography{sgd}

\appendix
\section{Elementary inequalities}

In this appendix, we collect some useful inequalities. We begin with an estimate on the operator
norm. This estimate is well known (see, e.g., \cite{LinRosasco:2017}).
\begin{lemma}\label{lem:kernel}
For $j<k$, and any symmetric and positive semidefinite matrix $S$ and step sizes $\eta_j\in (0,\|S\|^{-1}]$ and $p\geq 0$, there holds
\begin{equation*}
  \|\prod_{i=j}^k(I-\eta_iS)S^p\|\leq \frac{p^p}{e^p(\sum_{i=j}^k\eta_i)^p}.
\end{equation*}
\end{lemma}

Next we derive basic estimates on finite sums involving $\eta_j=c_0j^{-\alpha}$, with $c_0>0$
and $\alpha\in[0,1)$.
\begin{lemma}\label{lem:basicest}
For the choice $\eta_j=c_0j^{-\alpha}$, $\alpha\in[0,1)$ and $r\in[0,1]$, for any $1\leq j<k$, there holds
\begin{align}
  \sum_{i=1}^k \eta_i & \geq (2^{1-\alpha}-1)(1-\alpha)^{-1}c_0 k^{1-\alpha},\label{eqn:basicest1}\\
   \sum_{j=1}^{k-1} \frac{\eta_j}{(\sum_{i=j+1}^k\eta_i)^{r}} & \leq\left\{\begin{array}{ll}
       c_0^{1-r}B(1-\alpha,1-r) k^{(1-r)(1-\alpha)}, & r\in [0,1),\\
       2^\alpha((1-\alpha)^{-1}+\ln k) & r=1,
     \end{array}\right.\label{eqn:basicest2}
\end{align}
where $B(\cdot,\cdot)$ is the Beta function defined by $B(a,b)=\int_0^1s^{a-1}(1-s)^{b-1}\d s$ for any $a,b>0$.
\end{lemma}
\begin{proof}
Since $\alpha\in[0,1)$, we have $c_0^{-1}\sum_{i=1}^k \eta_i \geq \int_{1}^{k+1} s^{-\alpha}
\d s = (1-\alpha)^{-1}((k+1)^{1-\alpha}-1)\geq (1-\alpha)^{-1}(2^{1-\alpha}-1)k^{1-\alpha}$.
This shows the estimate \eqref{eqn:basicest1}. Next, since $\eta_i\geq c_0k^{-\alpha}$, for any $i=j+1,\ldots,k$, we have
\begin{align}\label{eqn:est-sum}
  c_0^{-1}\sum_{i=j+1}^k \eta_i \ge k^{-\alpha}(k-j).
\end{align}
If $r\in [0,1)$, by changing variables and by the definition of the Beta function $B(\cdot,\cdot)$, we have
\begin{align*}
   &c_0^{r-1}\sum_{j=1}^{k-1} \frac{\eta_j}{(\sum_{i=j+1}^k\eta_i)^r}  \leq  k^{r\alpha}\sum_{j=1}^{k-1} j^{-\alpha}(k-j)^{-r}\\
   \leq& k^{r\alpha}\int_0^ks^{-\alpha}(k-s)^{-r}\d s = B(1-\alpha,1-r) k^{(1-r)(1-\alpha)}.
\end{align*}
For $r=1$, it can be derived directly
\begin{align*}
   \sum_{j=1}^{k-1} \frac{\eta_j}{\sum_{i=j+1}^k\eta_i}  
   \leq& k^{\alpha} \sum_{j=1}^{[\frac{k}{2}]} j^{-\alpha} (k-j)^{-1} + k^{\alpha} \sum_{j=[\frac{k}{2}]+1}^{k-1} j^{-\alpha} (k-j)^{-1}\\
   \leq &2k^{\alpha-1} \sum_{j=1}^{[\frac{k}{2}]}j^{-\alpha} + 2^{\alpha}\sum_{j=[\frac{k}{2}]+1}^{k-1} (k-j)^{-1}.
\end{align*}
Simple computation gives
$\sum_{j=[\frac{k}{2}]+1}^{k-1} (k-j)^{-1} \leq  \ln k$ and $
\sum_{j=1}^{[\frac{k}{2}]}j^{-\alpha} \leq (1-\alpha)^{-1}(\frac{k}{2})^{1-\alpha}.$
Combining the last two estimates yields the estimate \eqref{eqn:basicest2}.
\end{proof}

The next result gives some further estimates.
\begin{lemma}\label{lem:basicest2}
For $\eta_j = c_0j^{-\alpha}$, with $\alpha\in(0,1)$, $\beta\in[0,1]$, and $r\geq0$, there hold
\begin{align*}
  \sum_{j=1}^{[\frac{k}{2}]} \frac{\eta_j^2}{(\sum_{i=j+1}^k\eta_i)^r}j^{-\beta} &\leq c_{\alpha,\beta,r}k^{-r(1-\alpha)+\max(0,1-2\alpha-\beta)},\\
    \sum_{j=[\frac{k}{2}]+1}^{k-1} \frac{\eta_j^2}{(\sum_{i=j+1}^k\eta_i)^r}j^{-\beta} &\leq c_{\alpha,\beta,r}'k^{-((2-r)\alpha+\beta)+\max(0,1-r)},
\end{align*}
where we slightly abuse $ k^{-\max(0,0)}$ for $\ln k$, and the constants $c_{\alpha,\beta,r}$ and $c'_{\alpha,\beta,r}$ are given by
\begin{align*}
   c_{\alpha,\beta,r} & = c_0^{2-r} \left\{\begin{array}{ll}
      2^r(2\alpha+\beta-1)^{-1}, & 2\alpha +\beta>1,\\
      2, & 2\alpha +\beta = 1, \\
      2^{r-1+2\alpha+\beta}(1-2\alpha-\beta)^{-1}, & 2\alpha +\beta<1,
     \end{array}\right.\\
   c'_{\alpha,\beta,r} & = 2^{2\alpha+\beta}c_0^{2-r}\left\{\begin{array}{ll}
        (r-1)^{-1}, & r>1,\\
        1, & r = 1, \\
        2^{r-1}(1-r)^{-1}, &r <1.
      \end{array}\right.
\end{align*}
\end{lemma}
\begin{proof}
It follows from the inequality \eqref{eqn:est-sum} that
\begin{align*}
    &\quad c_0^{r-2}\sum_{j=1}^{[\frac{k}{2}]} \frac{\eta_j^2}{(\sum_{i=j+1}^k\eta_i)^r}j^{-\beta}
   = \sum_{j=1}^{[\frac{k}{2}]} \frac{j^{-(2\alpha+\beta)}}{(\sum_{i=j+1}^ki^{-\alpha})^r}\\
    & \leq k^{r\alpha}\sum_{j=1}^{[\frac{k}{2}]} j^{-(2\alpha+\beta)}(k-j)^{-r}
     \leq 2^rk^{-r+r\alpha}\sum_{j=1}^{[\frac{k}{2}]} j^{-(2\alpha+\beta)}\\
    & \leq 2^rk^{r\alpha-r}\left\{\begin{array}{ll}
        (2\alpha+\beta-1)^{-1}, & 2\alpha+\beta >1,\\
        \ln k, & 2\alpha+\beta = 1, \\
        (1-2\alpha-\beta)^{-1}(\frac{k}{2})^{1-2\alpha-\beta}, & 2\alpha +\beta<1.
      \end{array}\right.
\end{align*}
Collecting terms shows the first estimate. The second estimate follows similarly.
\end{proof}

Last, we give a technical lemma on recursive sequences.
\begin{lemma}\label{lem:iter-est}
Let $\eta_j=c_0j^{-\alpha}$, $\alpha\in(0,1)$. Given $\{b_j\}_{j=1}^\infty\subset\mathbb{R}_+$, $a_1\ge0$ and $c_i>0$,
$\{a_j\}_{j=2}^\infty \subset\mathbb{R}_+$ satisfies
\begin{equation*}
  a_{k+1} = c_1\sum_{j=1}^{k-1} \frac{\eta_j^2}{\sum_{i=j+1}^k\eta_i}a_j + c_2k^{-2\alpha}a_k + b_k.
\end{equation*}
If $b_j$ is nondecreasing, then for some $c(\alpha,c_i)$ dependent
of $\alpha$ and $c_i$, there holds
\begin{equation*}
a_{k+1} \leq c(\alpha,c_i)k^{-\min(\alpha,1-\alpha)}\ln k + 2b_k.
\end{equation*}
\end{lemma}
\begin{proof}
Let $c_\alpha = c(\alpha,0,1)+c'(\alpha,0,1)$ from Lemma \ref{lem:basicest2}.
Take $k_*\in\mathbb{N}$ such that $c_1c_\alpha k^{-\min(1-\alpha,\alpha)}\ln k+c_2k^{-2\alpha}<1/2$
for any $k\geq k_*$. The existence of a finite $k_*$ is due to the monotonicity of
$f(t)= t^{-\min(1-\alpha,\alpha)}\ln t$ for large $t>0$ and $\lim_{t\to\infty}f(t)=0$.
Now we claim that there exists $a_{*}>0$ such that $a_k\leq a_{*} + 2b_k$ for any $k\in\mathbb{N}.$
Let $a_{*}=\max_{1\leq k\leq k_{*}} a_k$. The claim is trivial for $k\leq k_{*}$. Suppose it holds
for some $k\geq k_{*}$. Then by Lemma \ref{lem:basicest2} and the monotonicity of $b_j$,
\begin{align*}
   a_{k+1}& \leq \max_{1\leq i \leq k} a_i \Big(c_1 \sum_{j=1}^{k-1} \frac{\eta_j^2}{\sum_{i=j+1}^k\eta_i} +c_2k^{-2\alpha}\Big) + b_k\\
           & \leq (a_{*}+2b_k)(c_1c_\alpha k^{-\min(\alpha,1-\alpha)}\ln k + c_2k^{-2\alpha}) + b_k\\
           & \leq \tfrac{1}{2}(a_{*}+2b_{k}) + b_k \leq a_{*}+2b_{k+1},
\end{align*}
This shows the claim by mathematical induction. Next, by
Lemma \ref{lem:basicest2}, for any $k> k_{*}$, we have
\begin{align*}
  a_{k+1} & \leq (a_{*}+2b_k)(c_1c_{\alpha}k^{-\min(\alpha,1-\alpha)}\ln k + c_2k^{-2\alpha}) + b_k\\
           &\leq c(\alpha,c_i) k^{-\min(\alpha,1-\alpha)}\ln k + 2b_k.
\end{align*}
This completes the proof of the lemma.
\end{proof}

\begin{remark}
By the argument in Lemma \ref{lem:iter-est} and a standard bootstrapping argument, we deduce the following
assertions. If $\sup_jb_j<\infty$, then $\{a_j\}_{j=1}^\infty$ is bounded by a constant dependent of $\alpha$,
$\sup_jb_j$ and $c_i$s. Further, if $b_j\leq c_3j^{-\gamma}$, $j\in\mathbb{N}$ with $\gamma>0$, then for some
$c(\alpha,\gamma,c_i,\ell)$ dependent of $\alpha$, $\gamma$, $\ell$ and $c_i$s, there holds
\begin{equation*}
a_{k+1} \leq c(\alpha,\gamma,c_i,\ell)k^{-\min(\ell\alpha,1-\alpha,\gamma)}\ln^\ell k.
\end{equation*}
\end{remark}
\end{document}